\documentclass[11pt]{amsart}
\usepackage{amscd,amsthm,amssymb,amsfonts,amsmath,euscript, stmaryrd}
\usepackage{mathrsfs,mathtools}
\usepackage{comment}
\usepackage{color}
\usepackage{xcolor}

\theoremstyle{plain}
\newtheorem{thm}[subsubsection]{Theorem}

\newtheorem{lemma}[subsubsection]{Lemma}
\newtheorem{prop}[subsubsection]{Proposition}
\newtheorem{cor}[subsubsection]{Corollary}

\theoremstyle{definition}
\newtheorem{defn}[subsubsection]{Definition}
\newtheorem{eg}[subsubsection]{Example}

\newtheorem{remark}[subsubsection]{Remark}
\newtheorem*{thank}{{\bf Acknowledgments}}
\newcommand{\nc}{\newcommand}

\def\makeop#1{\expandafter\def\csname#1\endcsname
  {\mathop{\rm #1}\nolimits}\ignorespaces}
\makeop{Hom}   \makeop{End}   \makeop{Aut}   \makeop{Isom}  \makeop{Pic} 
\makeop{Gal}   \makeop{ord}   \makeop{Char}  \makeop{Div}   \makeop{Lie} 
\makeop{PGL}   \makeop{Corr}  \makeop{PSL}   \makeop{sgn}   \makeop{Spf}
\makeop{Spec}  \makeop{Tr}    \makeop{Nr}    \makeop{Fr}    \makeop{disc}
\makeop{Proj}  \makeop{supp}  \makeop{ker}   \makeop{im}    \makeop{dom}
\makeop{coker} \makeop{Stab}  \makeop{SO}    \makeop{SL}    \makeop{SL}
\makeop{Cl}    \makeop{cond}  \makeop{Br}    \makeop{inv}   \makeop{rank}
\makeop{id}    \makeop{Fil}   \makeop{Frac}  \makeop{GL}    \makeop{SU}
\makeop{Nrd}   \makeop{Sp}    \makeop{Tr}    \makeop{Trd}   \makeop{diag}
\makeop{Res}   \makeop{ind}   \makeop{depth} \makeop{Tr}    \makeop{st}
\makeop{Ad}    \makeop{Int}   \makeop{tr}    \makeop{Sym}
   \makeop{can}
\makeop{length}\makeop{SO}    \makeop{torsion} \makeop{GSp} \makeop{Ker}
\makeop{Adm}   \makeop{Mat}
\makeop{Q-isog}
\makeop{rad}
\makeop{Ind}
\makeop{bas}
\DeclareMathOperator{\Mass}{Mass}

\DeclareMathOperator{\Sim}{sim}
\DeclareMathOperator{\Irr}{Irr}
\DeclareMathOperator{\GU}{GU}
\DeclareMathOperator{\K}{K}

\def\makebb#1{\expandafter\def
  \csname bb#1\endcsname{{\mathbb{#1}}}\ignorespaces}
\def\makebf#1{\expandafter\def\csname bf#1\endcsname{{\bf
      #1}}\ignorespaces} 
\def\makegr#1{\expandafter\def
  \csname gr#1\endcsname{{\mathfrak{#1}}}\ignorespaces}
\def\makescr#1{\expandafter\def
  \csname scr#1\endcsname{{\EuScript{#1}}}\ignorespaces}
\def\makecal#1{\expandafter\def\csname cal#1\endcsname{{\mathcal
      #1}}\ignorespaces} 

\def\doLetters#1{#1A #1B #1C #1D #1E #1F #1G #1H #1I #1J #1K #1L #1M
                 #1N #1O #1P #1Q #1R #1S #1T #1U #1V #1W #1X #1Y #1Z}
\def\doletters#1{#1a #1b #1c #1d #1e #1f #1g #1h #1i #1j #1k #1l #1m
                 #1n #1o #1p #1q #1r #1s #1t #1u #1v #1w #1x #1y #1z}
\doLetters\makebb   \doLetters\makecal  \doLetters\makebf
\doLetters\makescr 
\doletters\makebf   \doLetters\makegr   \doletters\makegr

\def\Gm{{{\bbG}_{\rm m}}}   

\normalsize

\makeop{Bl}

\def\Fpbar{\overline{\bbF}_p}

\def\Qpbar{\overline{{\bbQ}_p}}
\def\Qp{{\bbQ}_p}

\def\Zp{{\bbZ}_p}
\def\Qbar{\overline{\bbQ}}
\def\Sh{{\rm Sh}}

\newcommand{\Z}{\mathbb Z}
\newcommand{\Q}{\mathbb Q}
\newcommand{\R}{\mathbb R}
\newcommand{\C}{\mathbb C}
\newcommand{\D}{\mathbb D}
\renewcommand{\H}{\mathbb H}  
\newcommand{\A}{\mathbb A}    
\newcommand{\F}{\mathbb F}

\newcommand{\<}{\langle}   
\renewcommand{\>}{\rangle} 

\newcommand{\isoto}{\stackrel{\sim}{\longrightarrow}}
\nc{\embed}{\hookrightarrow}

\newcommand{\M}{\mathcal M}

\newcommand{\dieu}{Dieudonn\'{e} }

\nc{\ol}{\overline}
\nc{\wt}{\widetilde}
\nc{\opp}{\mathrm{opp}}
\def\ul{\underline}

\def\der{{\rm der}}
\def\wh{\widehat}

\makeop{Ram}
\makeop{Rep}

\def\sfF{\mathsf{F}}
\def\sfV{\mathsf{V}}

\def\Qpbar{\ol \Q_p}
\def\GmQpbar{\bbG_{\rm m,  \Qpbar}}

\begin{document}
\renewcommand{\thefootnote}{\fnsymbol{footnote}}
\setcounter{footnote}{-1}
\numberwithin{equation}{section}


\title[Mass formulas and the basic locus of  unitary Shimura varieties]{Mass formulas and the basic locus of  unitary  Shimura varieties}
 \author{Yasuhiro Terakado}
\address{
(Terakado)National Center for Theoretical Sciences 
\\
Cosmology Building 
\\
No.~1  Roosevelt Rd. Sec.~4 
\\ 
Taipei, Taiwan, 10617} 
\email{terakado@ncts.ntu.edu.tw} 
  
\author{Chia-Fu Yu}
\address{
(Yu)Institute of Mathematics, Academia Sinica \\
Astronomy Mathematics Building \\
No.~1, Roosevelt Rd. Sec.~4 \\ 
Taipei, Taiwan, 10617} 
\email{chiafu@math.sinica.edu.tw}

\address{
(Yu)National Center for Theoretical Sciences
\\
Cosmology Building
\\
No.~1  Roosevelt Rd. Sec.~4
\\
Taipei, Taiwan, 10617}


\date{\today}
\subjclass[2010]{} 


\begin{abstract}
In this article we compute the mass associated to any  unimodular lattice in a Hermitian space over an arbitrary CM field under a condition at $2$. We study the geometry and arithmetic of the basic locus of the $\GU(r,s)$-Shimura variety associated to an imaginary quadratic field modulo a good prime $p>2$.  
We give explicit formulas for the numbers of irreducible and connected components of the basic locus, and  of points of  the zero-dimensional Ekedahl-Oort (EO) stratum, as well as of the irreducible components of basic EO strata when the signature is either $(1, n-1)$ or $(2,2)$.  
\end{abstract} 

\maketitle 
\tableofcontents
\section{Introduction}
Let $p$ be a prime number and $k=\Fpbar$ an algebraic closure of $\F_p$. 
The classical Eichler-Deuring mass formula  gives a weighted number of the isomorphism classes of the supersingular elliptic curves over $k$:
\begin{align*}
\sum_{[E] \ : \ {\rm supersingular}} \frac{1}{\lvert \Aut(E) \rvert}=\frac{p-1}{24}.\end{align*}

A higher dimensional generalization was obtained by Ekedahl \cite{Ekedahl} and Katsura-Oort~\cite{KO} (based on the mass formula for  its arithmetic counterpart by Hashimoto-Ibukiyama~\cite{HI}) where supersingular elliptic curves are replaced by principally polarized superspecial abelian varieties. Recall that an abelian variety $A$ over $k$ is said be to \emph{superspecial} (resp.~\emph{supersingular}) if it is isomorphic (resp.~isogenous) to a  product of supersingular elliptic curves. 
Let $\calA_g$ denote the (coarse) moduli space over $k$ of principally polarized abelian  varieties of dimension $g$. 
Extending earlier works of Ibukiyama, Katsura, and Oort \cite{IKO, KO} in lower dimensions, Li and Oort \cite{LO} investigated  the geometry of the supersingular locus $\calS_g$ of $\calA_g$, the closed subvariety parameterizing supersingular points. They determined the dimension and expressed the number of the irreducible components as a class number 
of the quaternion Hermitian lattices in question. For the Siegel moduli spaces $\calA_{g,1,N}$ with a prime-to-$p$ level-$N$ structure, the second author \cite{Yu:mass_siegel} obtained an explicit formula for the number of irreducible components of the supersingular locus.


For Shimura varieties other than Siegel modular varieties, Bachmat and Goren \cite{BG} studied the supersingular locus of Hilbert-Blumenthal modular surfaces. The second author \cite{yu:thesis, yu:mass_hb} studied the supersingular locus of lower-dimensional Hilbert-Blumenthal moduli spaces. 
He also computed the number of superspecial points in an arbitrary PEL-type Shimura variety of type C modulo a good prime $p$ \cite{Yu2}, using  Shimura's mass formula for quaternionic  unitary groups \cite{Shimura2}.\\


The goal of this paper  is to  extend the above results to the reduction modulo $p$ of a unitary Shimura variety associated with a unimodular Hermitian lattice. 
   Let $E$ be an imaginary quadratic extension over $\Q$. 
Let $(V, \varphi)$ be 
 a Hermitian space of dimension  $n$ over $E$ with signature $(r,s)$. 
Let $\Lambda \subset V$ an $O_E$-lattice on which $\varphi$ is  unimodular. 
 Let $\mathbf G={\rm GU}(\Lambda, \varphi)$ be the  unitary  similitude group of $(\Lambda, \varphi)$. 
 Let $N \geq 3$ be an integer with $p \nmid N$. We consider the open compact subgroup ${\rm K} \coloneqq  {\bf G}(\Z_p){\rm K}^p(N) \subset {\bf G}(\A_f)$ where  ${\bf G}(\Z_p)$ is a hyperspecial subgroup of $\bfG(\Q_p)$ and 
 ${\rm K}^p(N)$ is  the kernel of the reduction homomorphism ${\bf G}(\widehat{\mathbb Z}^p) \to {\bf G}(\widehat{\mathbb Z}^p/N\widehat{\mathbb Z}^p).$ 
We write $\bfM_{\rm K}$ for the associated moduli space,  with good reduction at $p$,  classifying  \emph{principally} polarized abelian schemes  $A$ of relative dimension $n$ with an $O_E$-action and a level $N$ structure (Section~\ref{moduli}). 
Let $\mathcal M_{{{\rm K}}} \coloneqq \bfM_{\rm K} \otimes k$ denote the base extension of $\mathcal M_{{{\rm K}}}$ to $k$. 
To each $k$-point in  $\M_{\K}$ we can  attach a 
$p$-divisible group  $A[p^{\infty}]$ over $k$ equipped with additional structures. 
Its 
isogeny type defines the \emph{Newton stratification}  on $\mathcal M_{\rm K}$. 
There is a unique closed Newton stratum, called the basic locus  and denoted by $\mathcal M_{\rm K}^{\rm{bas}}$. 
Furthermore, the isomorphism type of the $p$-torsion of $A[p^{\infty}]$ with additional structures defines the  \emph{Ekedahl-Oort (EO)  stratification} \cite{Moonen}. 
There is a unique $0$-dimensional EO stratum, denoted by $\mathcal M_{\rm K}^e$. This is always non-empty and $\mathcal M_{\rm K}^e \subset   \mathcal M_{{{\rm K}}}^{\rm bas}$. 
If $p$ is inert in $E$, or $r=s$ and $p$ is split in $E$, the basic locus $\mathcal M_{\rm K}^{\rm bas}$ (resp.~$\mathcal M^e_{\rm K}$)  coincides with the supersingular locus (resp.~superspecial locus) of $\M_{\K}$. 

The geometry of the basic locus $\M_{\K}^{\bas}$  has been studied  by many people: 
 Vollaard \cite{Vollaard1} and Vollaard-Wedhorn \cite{Vollaard} for $(r,s)=(1, n-1)$ with $p$ inert in $E$; 
 Howard-Pappas \cite{HP0} for $(r,s)=(2,2)$ again with $p$ inert;  Fox \cite{Fox} for $(r,s)=(2,2)$ with $p$ split; 
  Imai and Fox \cite{FI} for $(r,s)=(2,n-2)$ with $p$ inert.
  The geometry of 
  the EO stratification for arbitrary signature $(r,s)$ 
  has been studied by Wooding in her thesis~\cite{Wooding}. 
  Further, 
  G\"{o}rtz-He \cite{GH}  and G\"{o}rtz-He-Nie  \cite{GHN1,GHN2}  developed more  group-theoretic approaches to giving a concrete  description of the basic locus of Shimura varieties. 
  Note that such a description of the basic locus has been used to compute intersection numbers  
of special cycles by Kudla and Rapoport   \cite{KuR1,KuR2}. 
  
 In \cite{DG},  
De~Shalit and Goren   have studied the basic locus  $\M_{\K}^{\bas}$ intensively 
when $(r, s)=(1, 2)$ and $p$ is inert in $E$. 
They 
derived a formula relating the number of irreducible components of $\M_{\K}^{\bas}$ (and the cardinality of  $\calM_{\rm K}^e$)   to the second Chern class 
of the complex algebraic surface $\mathbf M_{\rm K} (\C)$.
An explicit formula for the second Chern class of a  connected component of $\bf M_{{\rm K}}(\C)$  was given  by Holzapfel \cite{Holzapfel};   
see  Example \ref{Picard}. 

In this paper, we study the basic locus $\M_{\K}^{\bas}$  for an arbitrary signature $(r,s)$ (including $rs=0$) and any unramified prime $p>2$. 
We  give  explicit formulas (Theorem  \ref{intro}) for 
 \begin{itemize}
 \item[(i)]  the number of irreducible components  of the basic locus $\M_{\K}^{\bas}$, and 
 \item[(ii)]  the cardinality  of the $0$-dimensional stratum $\M_{\K}^e$. 
 \end{itemize}
      
We also treat irreducible components of basic EO strata of possibly positive dimension when  $(r,s)=(1, n-1)$ and $p$ is inert in $E$. 
    In this case, Vollaard and Wedhorn \cite[Section~6.3]{Vollaard} proved that for each odd integer $1 \leq t \leq n$  
 there exists a unique EO stratum $\mathcal M^{(t)}_{{{\rm K}}}$ of dimension 
 $\frac{1}{2}(t-1)$ 
which are contained in $\calM_{\rm K}^{\rm bas}$. 
 Let  $\overline{\mathcal M}^{(t)}_{{{\rm K}}}$ denote the Zariski closure 
 of $\mathcal M^{(t)}_{{{\rm K}}}$ 
 in $\mathcal M_{{{\rm K}}}$. 
 We give an explicit formula  (Theorem \ref{intro:t})  for 
 \begin{itemize}
     \item[(iii)] the number of irreducible components of $\overline{\M}_{\K}^{(t)}$. 
 \end{itemize}
We remark that if $(r, s)=(1, n-1)$ and  $p$ is split in $E$,  then $\mathcal M^{\rm bas}_{\rm K}=\calM_{\rm K}^e$. 
  
Further we study geometrically  connected components of the moduli space $\bfM_{\K}$. 
We give an explicit formula for the number of connected components of the  complex Shimura variety $\bfM_{\K} \otimes \C$ (Theorem \ref{connected}, Remark \ref{rem:MK},  and Example \ref{rs=0}). 
There exists a smooth compactification of the integral model  $\bfM_{\K}$ by K.-W. Lan~\cite{Lan}, hence $\bfM_{\K} \otimes \C$  has the same number of connected components as the   reduction  $\M_{\K}$ modulo $p$.  
We show that the  basic locus $\M_{\K}^{\bas}$ has the same number of connected components as $\M_{\K}$ has, except when 
$p$ is inert in $E$ and $(r,s)=(1,1)$,  or $p$ is split in $E$ and $\gcd(r,s)=1$. 
For the exceptional cases we also derive a formula relating the number of connected components of $\M_{\K}^{\bas}$ to the number  of its irreducible components (Theorem \ref{thm:con_comp}). \\

 We give a summary of each section. 
   In Section~\ref{sec:CM}, we study the mass of a unimodular Hermitian lattice over a CM field $E$. 
   We first compute the local densities of  unimodular Hermitian lattices  over local fields using  the results of   Gan and J.-K. Yu \cite{GY} and Cho  \cite{Cho1, Cho2}.  
   Then  we derive an exact mass formula  for a unimodular Hermitian  lattice  over a  CM field $E$ under the assumption that $2$ is unramified in  the  maximal totally real subfield of $E$  (Theorem  \ref{SMS}). 
   We note that 
    an exact formula  for the unimodular lattices defined by  identity matrices was  obtained earlier by Hashimoto and Koseki \cite[Theorem 5.7]{HK}  by  different techniques.
   
 From  Section~\ref{sec:P.2} we  restrict ourselves to the case where $E$ is an imaginary quadratic extension over $\Q$. 
  In Section~\ref{ssec:sim}  
    we describe the similitude factor of   $\bfG(\Z_{\ell})$ for each prime $\ell$.  
    This result will be used in Section~\ref{sec:bas}. 
   Let $(r,s)$ be a pair of non-negative integers, and assume  that  ${\bf G}_{\R}$ is isomorphic to the real Lie group $\GU(r,s)$. 
   When $rs>0$, the  number of connected components of  the  complex Shimura variety associated to $\mathbf G_{\Q}$ of level $\bfG(\widehat{\Z})$ can be expressed by  a  class number of  the quotient torus  $D\coloneqq \bfG_{\Q}/\bfG_{\Q}^{\rm der}$, where  $\bfG_{\Q}^{\rm der}$ denotes the  derived subgroup of $\bfG_{\Q}$. 
   The natural projection $\nu: \bfG_{\Q}\to D$ is   identified with the product of the similitude and  determinant characters. 
   We compute the class number of $D$,   
  using    our description of the  similitude factors and Kirschmer's description of the determinants  \cite{Kirschmer}. 
  We thus obtain  an explicit formula for the number of connected components  of the  complex Shimura variety in question.

 From  Section~\ref{sec:bas} we fix a  prime $p>2$ which is unramified in $E$. 
In this section we study the basic locus $\M_{\K}^{\bas}$ and the $0$-dimensional EO stratum $\M^e_{\K}$. 
We describe the    \dieu module $M$ of the  $p$-divisible group with additional structures attached  to  a point of $\M^e_{\K}$. 
  Further we compute the group $J_b$ of 
  automorphisms  of  the isocrystal $M[1/p]$ with additional structures, and  
also compute the stabilizer of $M$ in $J_b(\Qp)$ (Proposition \ref{J_Minert}).

In Section~\ref{sec:ADLV} 
we study  the affine Deligne-Lusztig variety $X_{\mu}(b)$ associated with a basic element $b$ of ${\bf G}(L)$ and the minuscule coweight $\mu$ given by the Shimura datum. 
The set of its irreducible components ${\rm Irr}(X_{\mu}(b))$  admits an action of $J_b(\Q_p)$. 
 By the work of 
Xiao-X.~Zhu \cite{XZ}, Hamacher-Viehmann 
\cite{HV}, Nie \cite{Nie}, and Zhou-Y.~Zhu \cite{ZZ}, 
The set of orbits 
$J_b(\Q_p) \backslash {\rm Irr}(X_{\mu}(b))$ is  in natural bijection with the 
``Mirkovic-Vilonen basis'' of a  certain weight space of a representation of 
the dual group  of ${\bf G}_{\Q_p}$. 
We compute the dimension of this weight space explicitly, and give a formula for the cardinality of the set  $J_b(\Q_p) \backslash {\rm Irr}( X_{\mu}(b))$. 
We also study an action of $J_b(\Q_p)$ on the set of    connected components of $X_{\mu}(b)$.

 In Section~\ref{sec:inner} we  first study a mass formula for the inner from associated to the basic locus, and then we state and prove the main theorems.
 For each $k$-point $(A,  \iota, \lambda, \bar{\eta})$ of   $\M_{\rm K}^{\rm bas}$,  we define a similitude group  $I$ of   auto-quasi-isogenies of the tuple $(A, \lambda, \iota)$. 
This group is  an inner form of $\mathbf G_{\Q}$  satisfying  $I_{\Q_p} \simeq J_b$ and $I_{\A_f^p} \simeq {\bf G}_{\A_f^p}$.  
The mass of $I$ with respect to an open compact subgroup $U$ of $I(\A_f)$  is then defined as a weighted cardinality of the  double coset space  
$I(\Q) \backslash I(\A_f)/U$,  and  denoted by $\Mass(I, U)$ (Definition~\ref{mass:I}). 
Let ${\rm I}_p$ be a maximal parahoric subgroup of $J_b(\Q_p)$ and regard ${\rm I}_p \bfG(\widehat{\Z}^p)$ as a subgroup of $I(\A_f)$. 
We  give an explicit formula for  $\Mass(I, {\rm I}_p\bfG(\widehat{\Z}^p))$   (Theorem \ref{Mass_inner}). 
The proof of this formula consists of two steps. 
First we derive an equality  relating the mass of $I$ to a mass of its subgroup $I^1$  consisting of elements with trivial similitude factor. 
Here we use the description of the similitude factor of $\bfG(\widehat{\Z}^p)$ given in Section~\ref{ssec:sim}. 
Then we show that the mass of $I^1$ equals the mass of the   unimodular lattice $\Lambda$  multiplied by   the reciprocal of the volume of $I^1(\Q_p) \cap {\rm I}_p$. 

 For each irreducible component $Z$ of $X_{\mu}(b)$,  its  stabilizer ${\rm I}_p^Z$ in $J_b(\Q_p)$  is a parahoric subgroup with maximum volume by  \cite{HZZ}.  
 Moreover, the $p$-adic uniformization theorem of Rapoport-Zink \cite{RZ} implies that  
\[ \lvert \Irr(\M_{\K}^{\bas}) \rvert=\lvert J_b(\Q_p)\backslash \Irr(X_{\mu}(b)) \rvert \cdot 
\Mass(I, {\rm I}_p^Z \bfG(\wh{\Z}^p)) 
\cdot [\bfG(\wh{\Z}^p) : \K^p(N)], \]
 and thus we obtain an  explicit formula. 
We also derive formulas for   the numbers of points of $\M_{\K}^e$ and 
  connected components of $\M_{\K}^{\bas}$. 
Further we illustrate these results in  low-dimensional examples. 
In particular, we compute the  number of irreducible components of  EO strata in $\M_{\K}^{\bas}$  when  $(r,s)=(1, n-1)$ or $(r,s)=(2,2)$ (Example \ref{2,2}). 
%

In Section~\ref{sec:bound}, we  give an application of the main theorems to the arithmetic of mod $p$  automorphic forms. 
In a letter to Tate \cite{Serre}, Serre  proved that the systems of  Hecke eigenvalues appearing in the space of mod $p$ modular forms are the same as those appearing in the space of  algebraic modular forms on  a  quaternion algebra over $\Q$. 
  This result can be regarded  as a mod $p$ analogue of the Jacquet-Langlands  correspondence.  
Serre's result was generalized to  the  Siegel case  by Ghitza \cite{Ghitza}, to the case of $\GU(r,s)$ Shimura varieties  with $p$ inert by Reduzzi \cite{Reduzzi}, and to the  Hodge type case by the authors \cite{TY}. 
 This correspondence,  combined  with 
the formula for the  cardinality of $\calM_{\K}^e$, gives  an explicit upper bound for the number of the systems of Hecke eigenvalues  appearing in automorphic forms on $\M_{\K}$ (Theorem \ref{bound}). 

\begin{thank}
It is a great pleasure to thank M.~Chen, S.~Cho, W.-T.~Gan, Gross, Hamacher, X.~He, Kisin, S.~Nie, Viehmann, Vollaard, Wedhorn, J.-K.~Yu, L.~Xiao, R.~Zhou, X.~Zhu, Y.~Zhu for their works on which the present paper relies. The second author is partially supported by the MoST grant 109-2115-M-001-002-MY3.  
\end{thank}

\section{Mass formula for  unimodular Hermitian lattices}\label{sec:CM}
In this section, we give an exact mass formula for  a  unimodular Hermitian lattice. 
\subsection{Unimodular Hermitian lattices over local fields}
\label{lattice}
\subsubsection{}\label{setting}
Let $F$ be a non-Archimedean local field of characteristic zero,  and let $O_F$ be its ring of integers. 
Let $\F_q$ be the residue field of $O_F$, and   $q$  its cardinality. 
If $2 \mid q$,  we assume that  $F$ is an \emph{unramified} finite extension of  $\Q_2$. 
Let $(E, \bar{\cdot})$ be one of the following $F$-algebras with involution: 
\begin{itemize}
\item 
$E$ is a quadratic field extension of $F$,  $a \mapsto \bar{a}$ is the non-trivial automorphism of $E/F$;
\item 
$E=F \oplus F$, $\overline{(a, b)}=(b, a)$. 
\end{itemize}
Denote by $O_E$ the maximal order in $E$, that is, $O_E$ is the ring of integers of $E$ and $O_E=O_F \oplus O_F$ in each case, respectively. 
We write $\bfN=\bfN_{E/F}$ for the norm map of $E/F$ given by  $\bfN(b)= \bar{b} \cdot b.$ 
\subsubsection{}\label{sec:hermitian}
A \emph{Hermitian space} over $E$ is a free   $E$-module $V$ of  finite rank equipped with a Hermitian form $\varphi  : V \times V \to E$. 
By definition, the form $\varphi$  satisfies that 
\begin{equation}
\label{herm}
\begin{gathered}
\varphi(x+y, z)  =\varphi(x, z)+ \varphi(y, z), \quad 
\varphi(ax, by)  =a \bar {b} \cdot \varphi(x, y), 
\\
\varphi(y, x) =\overline{\varphi(x, y)},
\end{gathered}
\end{equation}
for $x, y, z \in V$ and $a, b\in E$. 
A \emph{Hermitian lattice}   over $O_E$ is an $O_E$-lattice $\Lambda$ (that is, a free $O_E$-module of finite rank) equipped with a form  $\varphi : \Lambda \times \Lambda \to O_E$ satisfying relations  \eqref{herm} 
for $x, y, z \in \Lambda$ and $a, b\in O_E$. 
In this paper, 
we assume that $\varphi$ is non-degenerate on $V \coloneqq \Lambda  \otimes_{O_F} F$, 
namely, for any $x\in V$, the condition $\varphi(x, V)=0$ implies that $x=0$. A Hermitian lattice  $\Lambda$ is called  \emph{unimodular} if 
$\Lambda$  coincides with its dual lattice 
$\Lambda^{\vee}=\Lambda^{\vee, \varphi}\coloneqq \{x \in V \mid  \varphi(x, \Lambda) \subset O_E\}.$  

\subsubsection{} 
Let $E/F$ be a quadratic field extension. 
The \emph{scale} $s(\Lambda)$ and the \emph{norm}  $n(\Lambda)$ of  a Hermitian lattice $\Lambda$ are the $O_{E}$-ideals defined as 
  \[ s(\Lambda):=\{ \varphi(x, y) \mid  x, y \in  \Lambda \}, \quad  
 n(\Lambda):=\sum_{x \in \Lambda} \varphi(x, x) \cdot O_{E}. \]
 If $s(\Lambda)=n(\Lambda)$, 
 the lattice $\Lambda$ is called {\it normal}. 
 Otherwise the lattice $\Lambda$ is called {\it subnormal}. 
 For two lattices  $\Lambda$ and $\Lambda'$, one has
 \[ s(\Lambda \oplus \Lambda')=s(\Lambda) +s(\Lambda'), \quad n(\Lambda \oplus \Lambda')=n(\Lambda) +n(\Lambda').\]
If  $\Lambda$ is unimodular,  then  $s(\Lambda)=O_E$. 

We call a Hermitian space or lattice \emph{isotropic}
 if it contains an element $x$ with $\varphi(x, x)=0$. Otherwise we call it \emph{anisotropic}. 

Let $\{e_i\}$ be a basis of a Hermitian space $V$ over $E$. 
Let $d(V)$ denote the  image of $\det(\varphi(e_i, e_j))$ by the natural projection   $F^{\times} \to F^{\times}/\bfN(E^{\times})$. 
Then $d(V)$ is  independent of the choice of $\{e_i\}$. 
We call $d(V)$ 
 the  \emph{determinant} of $V$. 
 By \cite[Theorem 3.1]{Jacobowitz}, 
the dimension and  determinant are complete isomorphism invariants of Hermitian spaces. 
 For a Hermitian lattice 
$\Lambda$ with basis $\{e_i\}$ over $O_E$, 
we call the matrix $(\varphi(e_i, e_j))$ the \emph{Gram matrix} of $\{e_i\}$. 
 We often write $\Lambda =   (\varphi(e_i, e_j))$.  
 If $\Lambda$ admits an orthogonal basis $\{e_i\}$ such that $\varphi(e_i,e_i)=a_i$, we  also write $\Lambda=  (a_1)\oplus \dots \oplus (a_n)$.
Let $d(\Lambda)$ denote the  image of $\det(\varphi(e_i, e_j))$ by the natural projection   $F^{\times} \to F^{\times}/\bfN(O_E^{\times})$. 
Then $d(\Lambda)$ is  independent of the choice of $\{e_i\}$. 
We call $d(\Lambda)$ 
 the  \emph{determinant} of $\Lambda$. 
 If $\Lambda$ is unimodular, then $d(\Lambda) \in O_F^{\times}/\bfN(O_E^{\times})$. 
 If $E/F$ is unramified, any unimodular lattice may be written $\Lambda =  (1)\oplus \cdots \oplus  (1)$ by \cite[Section 7]{Jacobowitz}. 
 If $E/F$ is ramified and $2 \nmid q$, the rank and determinant are the complete isomorphism invariants of unimodular lattices by \cite[Section 8]{Jacobowitz}. 
 If $E/F$ is ramified and $2 \mid q$, then 
 the rank, norm, and determinant are complete isomorphism invariants of unimodular lattices by 
  \cite[Proposition 10.4]{Jacobowitz}.

\subsubsection{}\label{disc}
Assume now that $E/F$ is  a ramified field extension. 
Let $\mathfrak D=\mathfrak D_{E/F}$ denote  the relative different of  $E/F$, and let $d_{E/K}=\bfN(\mathfrak D)$ denote the discriminant ideal  of $E/F$. 
 When $E/F$ is a ramified  extension, we choose a uniformizer $\pi$ of $O_E$ as follows.  If $2 \nmid q $,   one can choose a uniformizer $\varpi$ of $O_F$ such that  $E=F(\sqrt{\varpi})$, and then 
 the element $\pi \coloneqq \sqrt{\varpi}$ is a uniformizer of $O_E$. 
We have 
 $\mathfrak   D=(2\pi)=(\pi)$ and $d_{E/K}=(\varpi)$.  
  
 Suppose that  $E/F$ is a ramified extension with $2 \mid  q$. 
 Recall we assume  $F/\Q_2$ is an  unramified extension.   
 Then 
we have the following two cases\footnote{RU and RP are the abbreviations for ramified unit and ramified prime  respectively.}:
\begin{itemize}
\item[(RU)] $E=F(\sqrt{1+2u})$  for some unit $u$ in $O_{F}$. 
In this case, the element $\pi \coloneqq  1+\sqrt{1+2u}$ is a uniformizer of $O_{E}$. 
We have 
 $O_{E}= O_{F}[\pi]\simeq  O_{F}[X]/(X^2-2X-2u)$ and hence $\mathfrak D=(2\sqrt{1+2u})=(2)$.  
 It follows that $d_{E/K}=(4)$. 
\item[(RP)]
$E=F 
(\sqrt{2\delta})$ for some  element $\delta \in O_{F}$ with $\delta \equiv 1\pmod 2$. 
The element $\pi \coloneqq  \sqrt{2\delta}$ is a uniformizer of $O_E$. 
 We have $\mathfrak D=(2 \sqrt{2\delta})$ and $d_{E/F}=(8)$. 
 \end{itemize}
\begin{lemma}\label{splitting}  
Assume that $E/F$ is a ramified quadratic field extension.
Let $\Lambda$ be a unimodular lattice of rank $n$ over $O_E$ and $H$ be the rank-two lattice defined by  $H \coloneqq   \begin{pmatrix}
 0  & 1
 \\ 
 1 & 0
\end{pmatrix}$.  

{\rm (1)} If $2 \nmid q$, then  
 \begin{align*}
 \Lambda \simeq \begin{cases}
 H^{(n-1)/2} \oplus ((-1)^{(n-1)/2} \cdot d(\Lambda)) & \text{if $n$ is odd}; 
\\
H^{(n-2)/2} 
\oplus
(1) \oplus ((-1)^{(n-2)/{2}} \cdot d( \Lambda)) & {\text{if $n$ is   even}}. 
\end{cases}
\end{align*}

{\rm (2)} Assume $2 \mid q$ and $F$ is unramified over $\Q_2$. If $n$ is odd, then 
\[  \Lambda\simeq   H^{(n-1)/2} \oplus ((-1)^{(n-1)/2} \cdot d(\Lambda)). \]
If $n$ is even, then
\begin{align*}
 \Lambda\simeq  
 \begin{cases}
H^{(n-2)/{2}} 
\oplus
(1) \oplus ((-1)^{(n-2)/{2}}\cdot d( \Lambda)) & {\text{if $\Lambda$ is normal}}; 
\\
H^{n/{2}}  &  {\text{if   
$\Lambda$ is  
subnormal in RU}}; 
\\
H^{(n-2)/2}\oplus \begin{pmatrix}
2\delta & 1
\\
1 & 2b 
\end{pmatrix}, \quad b \in O_F
& {\text{if $\Lambda$ is subnormal in  RP}}.
\end{cases}
\end{align*}
\end{lemma}
\begin{proof}
If $2 \nmid q$, then the assertion follows from  \cite[Section~8]{Jacobowitz}. 
If  $2 \mid q$, the assertion follows from \cite[Theorem 2.10]{Cho1}, or 
 \cite[Propositions 10.2 and 10.3]{Jacobowitz}.  
\end{proof}

\begin{lemma}\label{subnormal}
Suppose that $2 \mid q$ and  $E/F$ is ramified in RP. 
Then there are two distinct isomorphism classes of  unimodular subnormal lattices $L$ of rank two over $O_E$. 
Moreover,  the following conditions are equivalent: 

{\rm (1)} $L \simeq \begin{pmatrix}
2\delta & 1
\\
1 & 2b 
\end{pmatrix}$ for an element  $b \in O_F$ such that 
the  equation $z^2+z \equiv b \pmod {2}$ has a solution 
$z$ in $\F_q=O_F/2O_F$.

{\rm (2)} The quadratic form  $\psi : L/\pi L \to \F_q$  induced by the  reduction $\varphi(x, x) \mod 2$ 
 is isotropic.
 
 {\rm (3)} $d(L) = -1$ as elements of  the quotient group  $O_F^{\times}/\bfN(O_E^{\times})$.
 
 {\rm (4)} $L \simeq H$. 
\end{lemma}
\begin{proof}
Let $L$ be a unimodular  subnormal lattice of rank two. 
By \cite[(9.1) and Proposition 9.1.a]{Jacobowitz}, there are  inclusions  $O_E \supsetneq n(L) \supset n(H)=2O_E$. 
This implies that  $n(L)=2O_E$ since  $n(L)$ can not be $\pi O_E$ by definition.  
Hence the isomorphism types of $L$ are classified only by 
the determinant, which takes value in  $O_F^{\times}/\bfN(O_E^{\times}) \simeq \Z/2\Z$.  
This implies the equivalence of (3) and (4). 

As in Lemma  \ref{splitting}, there is an isomorphism $L \simeq \begin{pmatrix}
2\delta & 1
\\
1 & 2b 
\end{pmatrix}$ for some $b \in O_F$. 
The equivalence of (1) and (2) follows from \cite[Remark 4.6]{Cho2}. 
We show the equivalence of (1) and (3). 
Suppose that there exist $x, y \in O_F$ such that 
$\bfN(x+\sqrt{2\delta}y)=-d(L)$, and equivalently $x^2-2\delta y^2=1-4\delta b.$  
Then $x \equiv 1 \pmod2$, and we can write $x=1+2w$ for some $w \in O_F$. 
It follows that $4(w+w^2)-2\delta y^2=-4\delta b$. 
Since $\delta\equiv 1 \pmod 2$, we have $y\equiv 0$ and $w+w^2 \equiv -b \equiv b \pmod 2$.  
Conversely, suppose that $z+z^2\equiv b \equiv - \delta b \pmod 2$ for some $z \in \F_q$. 
By Hensel's lemma, there is an element $w \in O_F$ such that $w+w^2=- \delta b$. 
It follows that \[\bfN(1+2w)=(1+2w)^2=1+4(w+w^2)=1-4 \delta b=-d(L).\]
\end{proof}
\begin{eg}\label{eg:subRP}
Let $E/F$ be as in Lemma \ref{subnormal}. 
The lattice $L = \begin{pmatrix}
2 & 1
\\
1 & 2 
\end{pmatrix}$ is a unimodular subnormal lattice. 
The above proof shows that $L$ satisfies those equivalent conditions  if and only if the equation $z+z^2 =1$ has a solution in $\F_q$, 
that is,  
 $\F_q$ contains the quadratic extension of $\F_2$. 
\end{eg}

\subsection{Unitary groups and local densities}
\subsubsection{}\label{groups} 
Let $(V,\varphi)$ be a Hermitian space over $E$ as in Section~\ref{lattice} and let $\Lambda$ be a unimodular  $O_E$-lattice in $V$. 
Let $G^1={\rm U}(V, \varphi)$ be the \emph{unitary group} associated to $(V,\varphi)$.
By definition $G^1$ is the reductive group over $F$ whose group of $R$-values for any commutative $F$-algebra $R$ is given by 
\begin{equation}\label{unitary}
G^1(R)= \{ g \in (\End_R(V \otimes_F R))^{\times}  \mid  \varphi(gx, gy)=\varphi(x, y), \quad x, y \in V \otimes _F R\}.
\end{equation}
As well-known, the group $G^1$ is connected. 
We define a naive integral model $\underline{G}'$ over $O_F$ of $G^1$ by 
\[\underline{G}'(R)=\{ g \in (\End_R   (\Lambda  \otimes_{O_{F}} R))^{\times} \mid  \varphi(gx, gy)=\varphi(x, y), \quad x, y \in \Lambda
\otimes_{O_{F}} R \}\]
for any commutative $O_{F}$-algebra $R$. 
By the work of Gan and J.-K. Yu  \cite[Proposition 3.7]{GY}, there exists a unique smooth affine group scheme $\underline{G}^1$ over $O_{F}$ such that $\underline{G}^1 \otimes_{O_{F}} F=G^1$ and $\underline{G}^1(R)=\underline{G}'(R)$ for any \'{e}tale $O_{F}$-algebra $R$. 

\subsubsection{}
Let $\overline{G}^1$ denote the maximal reductive quotient of the special fiber $\underline{G}^1 \otimes_{O_{F}} \F_q$. 
By Gan-Yu \cite[Proposition 6.2.3 and Section~9]{GY}, there is an isomorphism of  groups over $\F_q$: 
\begin{align}\label{unram}
\overline{G}^1
 \simeq \begin{cases*}
{\rm U}_n &  if $E/F$ is an unramified quadratic field extension; 
 \\
 \GL_n &   if $E=F \oplus F$.
 \end{cases*}
 \end{align}
 Here, ${\rm U}_n$ denotes a unitary group in $n$ variables over $\F_q$, which is unique up to isomorphism. 
 
 Assume that  $E/F$ is a ramified quadratic field  extension. Then 
\begin{align}\label{ram}
\overline{G}^1 \simeq 
\begin{cases}
 {\rm O}_n &   {\text{if  $2 \nmid q$, $n$ is odd}}; 
 \\ 
 {\rm O}_n &    {\text{if  $2 \nmid q$, $n$ is even}},  \ d(\Lambda)=(-1)^{n/2}; 
\\
^2  {\rm O}_n &    {\text{if  $2 \nmid q$, $n$ is even}},   \ d(\Lambda)\neq (-1)^{n/2}; 
\\
\Sp_{n-1} \times \Z/2\Z &   {\text{if  $2 \mid q$, $n$ is  odd, in RU}};
\\ 
\SO_n \times \Z/2\Z & {\text{if  $2 \mid q$, $n$ is odd,  in RP}};
\\
\Sp_{n-2} \times \Z/2\Z &  {\text{if  $2 \mid q$, $n$ is even, $\Lambda$ is normal in RU}};
\\ 
\SO_{n-1} \times \Z/2\Z  & {\text{if  $2 \mid q$, $n$ is even, $\Lambda$ is normal in RP}};
\\
\Sp_n & {\text{if  $2 \mid q$, $n$ is even, $\Lambda$ is subnormal in RU}};
 \\
 {\rm O}_n & {\text{if  $2 \mid q$, $n$ is even, $\Lambda$ is subnormal in RP}}, \ 
 d(\Lambda) = (-1)^{n/2}; 
  \\
  ^2  {\rm O}_n & {\text{if  $2 \mid q$, $n$ is even, $\Lambda$ is subnormal in RP}}, \  d(\Lambda) \neq  (-1)^{n/2}. 
\end{cases}
\end{align}
Here, ${\rm O}_n$ denotes the split orthogonal group in $n$ variables, $^2{\rm O}_n$  denotes the quasi-split but non-split orthogonal group, 
$\Sp_n$ denotes the  symplectic group, ${\rm SO}_n$  denotes the split special orthogonal group, and
$\Z/2\Z$ denotes  the constant group scheme of order $2$. 
The determinant  $d(\Lambda)$ takes value in the quotient group  $O_F^{\times}/\bfN(O_E^{\times})$. 

These isomorphisms were  given by  Gan-Yu   \cite[Proposition 6.2.3]{GY} when $2 \nmid q$, and by Cho  \cite{Cho1, Cho2} when $2 \mid q$. 
Here we recall Cho's construction. 
The function $x \mapsto \varphi(x, x)\pmod 2$ induces a quadratic form $\psi : \Lambda/\pi\Lambda \to \F_q$, which we   regard  as an additive polynomial.  
We define a lattice $B$ as the sublattice of $\Lambda$ such that $B/\pi\Lambda$ is the kernel of  $\psi$. 
By  \cite[Remark 2.11]{Cho1},  one has 
\begin{align*}
 B=
 \begin{cases*}
 H^{(n-1)/{2}} \oplus (\pi)e_1 
  & if $n$ is odd; 
 \\
 H^{(n-2)/{2}}\oplus (\pi)e_1\oplus O_E e_2 & if $n$ is even and $\Lambda$ is normal; 
 \\
 H^{{n}/{2}} & if $n$ is even and $\Lambda$ is subnormal.
 \end{cases*}
 \end{align*}
In the RU case, 
 the  reduction $\varphi\pmod \pi$  induces an alternating and bilinear form on $B/\pi \Lambda$, and we define $Y$ as the sublattice of $\Lambda$ such that $Y/\pi \Lambda$ is the radical of this form. 
In the RP case, the reduction 
  $2^{-1} \psi \pmod 2$ 
  induces a quadratic form on $B/ \pi B$, and 
 we define $Z$ as the sublattice of $\Lambda$ such that $Z/\pi B$ is the radical of this form. 
Then, 
for any \'{e}tale local $O_F$-algebra $R$ and any element $\tilde{m} \in \underline{G}^1(R)$ with reduction $m \in \underline{G}^1(R \otimes_{O_F} \F_q)$,  the action of $\tilde{m}$ on $\Lambda \otimes _{O_F}R$ preserves $B$, $Y$, and $Z$. 
Furthermore, 
there exists a unique   morphism of algebraic groups over $\F_q$
\begin{align*}
f : \underline{G}^1 \otimes \F_q \to 
\begin{cases*}
\Sp(B/Y, \varphi \bmod \pi) & in the  RU case; 
\\ 
{\rm O}(B/Z, 2^{-1}\psi \bmod 2)_{\rm red} & in the  RP case,
\end{cases*}
\end{align*} 
 such that  the image  $f(m) \in 
\GL(B \otimes_{O_F}R/Y\otimes_{O_F}R)$ (resp.~
$\GL(B \otimes_{O_F}R/Z\otimes_{O_F}R)$)  
 is induced by the action of $\tilde{m}$ on $\Lambda \otimes _{O_F}R$. 
 Here, ${\rm O}(B/Z, 2^{-1}\psi \bmod 2)_{\rm red}$ denotes the reduced subgroup scheme of 
 ${\rm O}(B/Z, 2^{-1}\psi \bmod 2)$. 
 There is also a  surjective morphism 
 $g : \underline{G}^1 \otimes \F_q \to (\Z/2\Z)^{\epsilon}$,  where $\epsilon=0$ if $n$ is even and $\Lambda$ is subnormal, or  $\epsilon=1$ otherwise. 
 The product $f \times g$ gives the projection from $\underline{G}^1 \otimes \F_q$ to its maximal  reductive quotient, by  
   \cite[Theorem 4.12]{Cho1} and  \cite[Theorem 4.11]{Cho2}.
  The isomorphism types of  $\Sp(B/Y, \varphi \bmod \pi)$ and 
${\rm O}(B/Z, 2^{-1}\psi \bmod 2)_{\rm red}$ are given in \cite[Remark 4.7]{Cho1} and \cite[Remark 4.6]{Cho2} respectively.
\begin{defn}
The \emph{local density} of $\Lambda$ is the quantity
\[\beta_{\Lambda} \coloneqq 
\lim_{N \to \infty} 
q^{-N \dim G^1} \cdot 
\lvert  \underline{G}'
(O_{F}/\varpi^N O_{F})\rvert, \]
where $\varpi$ is a uniformizer of $O_F$.
\end{defn}
The limit stabilizes for $N$ sufficiently large. 
By the results of Gan-Yu \cite[Theorem 7.3]{GY} and Cho  (\cite[Theorem 5.2]{Cho1},  \cite[Theorem 5.2]{Cho2}, and \cite[Remark 5.3]{Cho1}), the local density of a unimodular lattice $\Lambda$  can be computed via  the formula 
\begin{align}\label{density}
\beta_{\Lambda}= 
q^{-N \dim \overline{G}^1}\cdot \lvert  \overline{G}^1 (\F_q) \rvert, \quad N= \begin{cases} n &  {\text{if $2 \mid q$ and  $\Lambda$ is  subnormal}}; 
\\ 
0 & {\text{otherwise}}.
\end{cases}
\end{align}
In Table \ref{table:finite} we see the dimensions and orders of finite classical groups  appearing in $\overline{G}^1$.  
\newpage
 \begin{table}[htbp]
  \caption{The dimensions and orders of finite classical groups}
  \label{table:finite}
 \renewcommand{\arraystretch}{1.5}
  \centering
  \begin{tabular}{|c|c|c|}
    \hline 
  $U$   & $\dim U$   &   $\lvert U(\F_q) \rvert$ 
    \\ 
    \hline 
    $\GL_n$  & $n^2$ &  $q^{\frac{n(n-1)}{2}} \cdot \prod_{i=1}^n(q^i- 1)$ 
\\
  \hline
 ${\rm U}_n$ & $n^2$ & 
   $q^{\frac{n(n-1)}{2}}
    \cdot  \prod_{i=1}^n (q^i-(-1)^i)$
 \\   
 \hline
 ${\rm O}_{2m+1}$ & 
 $m(2m+1)$ & 
 $q^{m^2} \cdot \prod_{i=1}^m (q^{2i}-1)$ 
 \\ 
   \hline
 ${\rm O}_{2m} $ & $m(2m-1)$  & $2q^{m(m-1)}(q^m-1) \cdot \prod_{i=1}^{m-1}(q^{2i}-1)$ 
 \\ 
   \hline
$^2{\rm O}_{2m}$ & $m(2m-1)$ & $2q^{m(m-1)}(q^m+1) \cdot \prod_{i=1}^{m-1}(q^{2i}-1)$  
 \\
   \hline
 $\Sp_{2m}$ & $m(2m+1)$ & $q^{m^2} \cdot 
 \prod_{i=1}^m (q^{2i}-1)$ 
 \\ 
   \hline
 ${\rm SO}_{2m+1}$ 
 & $m(2m+1)$ & 
 $q^{m^2} \cdot \prod_{i=1}^{m}(q^{2i}-1)$ 
 \\ 
    \hline
  \end{tabular}
\end{table} 

 A computation shows the following:
\begin{lemma}\label{LD}
The local density  $\beta_\Lambda$ of a unimodular lattice $\Lambda$ is given by  
\begin{equation*}
\beta_{\Lambda}=
\begin{cases}
\prod_{i=1}^{n} (1-(-1)^i \cdot q^{-i}) &  \text{if $E/F$ is an unramified quadratic field extension}; 
\\
\prod_{i=1}^n (1-q_v^{-i}) 
&  \text{if $E=F \oplus E$}, 
\end{cases}
\end{equation*}
and if $E/F$ is a ramified quadratic field extension, then 
\begin{equation*}
\beta_{\Lambda}=
\begin{cases}
2 \cdot \prod_{i=1}^{(n-1)/{2}} (1-q^{-2i}) &     \text{if $n$ is odd} ; 
\\ 
2 (1 + q^{-n/2})^{-1}
\cdot 
\prod_{i=1}^{n/2}
(1-q^{-2i}) &  {\text{if $n$ is even}}, \
  2 \nmid q,  \ 
  d(\Lambda)=(-1)^{n/2}; 
  \\ 
  2  (1 - q^{-n/2})^{-1}
\cdot 
\prod_{i=1}^{n/2}
(1-q^{-2i}) &  {\text{if $n$ is  even}}; \
  2 \nmid q,  \ 
  d(\Lambda)\neq (-1)^{n/2}; 
 \\ 
 2  (1-q^{-n})^{-1}\cdot \prod_{i=1}^{n/2}(1-q^{-2i}) &  {\text{if $n$ is even}},   \ 
 2 \mid q,   \  \Lambda \ 
 {\text{normal}};
\\
q^n \cdot \prod_{i=1}^{n/2}
(1- q^{-2i}) &  {\text{if $n$ is even}}, \ 2 \mid q, \   \Lambda \ {\text{subnormal in RU}}; 
\\
2  q^n \cdot (1+ q^{-n/2})^{-1}\cdot \prod_{i=1}^{n/2}(1-q^{-2i}) & {\text{if $n$ is even}},  \ 2 \mid q,  \ 
 \Lambda \ 
  {\text{subnormal in RP}}, \ 
  d(\Lambda) = (-1)^{n/2};
  \\
2  q^n \cdot (1- q_v^{-n/2})^{-1}\cdot \prod_{i=1}^{n/2}(1-q^{-2i}) 
&  {\text{if $n$ is  even}},  \ 2 \mid q,   \ 
 \Lambda \ 
  {\text{subnormal in RP}}, \  d(\Lambda) \neq (-1)^{n/2}. 
\end{cases}
\end{equation*}
\end{lemma}
\subsection{Exact mass formula for unimodular lattices over CM-fields}
\subsubsection{}
Let $F$ be a totally real number field of degree $d$ over $\Q$, with ring of integers $O_F$ and ring of adeles $\A_F$. 
We assume that $2$ is unramified in $F$. 
Let $E$ be a totally imaginary quadratic extension of $F$,  with ring of integers $O_E$. 
The non-trivial automorphism of $E$ over $F$ is denoted by $a \mapsto \bar{a}$. 

For a finite place $v$ of $F$,  $F_v$ denotes the corresponding completion of $F$, $O_{F_v}$ denotes the ring of integers of $F_v$,  $\F_v$ denotes the residue field of $O_{F_v}$,  and $q_v$ denotes the cardinality of $\F_v$.  
We write $E_v=E \otimes_F F_v$ and $O_{E_v}=O_E \otimes_{O_{F}}O_{F_v}$. 
If $E_v/F_v$ is a ramified quadratic field extension, let $\pi_v$ denote a uniformizer of $O_{E_v}$. 

Let $(V, \varphi)$ be a \emph{Hermitian space} over $E$, by which we mean a finite dimensional vector space $V$ over $E$  
equipped with a non-degenerate  Hermitian form $\varphi : V \times V \to E$,  in the sense that $\varphi$ satisfies relations \eqref{herm} for $x, y, z \in V$ and $a, b \in E$. 
Let $G^1={\rm U}(V, \varphi)$ be the unitary group  associated with $(V, \varphi)$, which is a connected reductive group over $F$. 

 \subsubsection{}
 Let $\Lambda$ be a \emph{lattice}  in $V$, by which we mean a finitely generated $O_E$-submodule of $V$ such that $\Lambda\otimes_{O_F}F \simeq V$ and $\varphi(\Lambda, \Lambda)\subset O_E$. 
   For a finite place $v$ of $F$, we write $\Lambda_v=\Lambda \otimes_{O_{F_v}} O_{E_v}$. 
   Let ${\rm K}_v$ be the stabilizer of $\Lambda_v$  in $G^1(F_v)$, and let 
\[ {\rm K}=  G^1(F \otimes \R) \times \prod_{v : {\rm finite}}{\rm K}_v \subset G^1(\A_F).\]
The set of isomorphism classes of the genus of $\Lambda$ is indexed by  $\Sigma \coloneqq  G^1(F) 
\backslash G^1(\A_F)/{\rm K}$. 
For each class $[g] \in \Sigma$, represented by $g \in G^1(\A_F)$,  let 
  $\Gamma_{g} \coloneqq G^1(F \otimes \R)\cap g{\rm K}g^{-1}$. 
  Then the quotient $\Gamma_{g} \backslash G^1(F \otimes \R)$ is of finite volume with respect to any Haar measure on  $G^1(F \otimes \R)$ \cite[Section~10.4]{GY}. 
  
   Let $\mu_{c, \R}$ be the Haar measure on the compact form of the real Lie group  $G^1(F \otimes \R)$ which gives  the group  volume $1$. 
We then transfer $\mu_{c, \R}$ to $G^1(F \otimes \R)$ and 
 obtain a Haar measure on $G^1(F \otimes \R)$, which is denoted again by $\mu_{c, \R}$. 
  \begin{defn}\label{def:mass}
 The \emph{mass of} $\Lambda$ is defined by 
\[\Mass(\Lambda)\coloneqq \sum_{[g] \in \Sigma}
\int_{\Gamma_{g}\backslash G^1(F \otimes \R)}
\mu_{c, \R}. 
\]
\end{defn}

Let $\chi=\chi_{E}$ be the Dirichlet character corresponding to $E/F$. 
It satisfies that 
$
\chi (v)
=1, -1$, or $0$ according as $v$ is  split, inert, or ramified  in $E$. 
We use the convention that $
\chi^j=1$ if $j$ is even, and $\chi^j=\chi$ if $j$ is odd. 
Let $L_{F}(s, \chi^j)$ be the $L$-series 
 over $F$. 
\begin{thm}\label{SMS} 
Let  $\Lambda$ be a  unimodular  lattice in $V$. 
Then we have that
\begin{align*}
\Mass(\Lambda)=\frac{(-1)^s \cdot 2}{2^{nd+w}}
\cdot  
\prod_{j=1}^n L_{F}(1-j, \chi_E^j)  
\cdot 
\prod_{v \mid d_{E/F}} \kappa_v
\end{align*}
where $n \coloneqq \dim_E V$, $d\coloneqq [F : \Q]$,    $d_{E/F}$ denotes  the discriminant of $E/F$,  $w$ denotes the number of places $v$ with $v\mid d_{E/F}$, and the quantities  
$s$ and $\kappa_v$ are given by 
\begin{align*}
s & =\begin{cases}
0 &   {\text{if $n$ is odd}}; 
\\
\frac{nd}{2} &  {\text{if $n$ is  even}}, 
\end{cases} 
\\
\kappa_v & =
\begin{dcases}
1 &
  {\text{if $n$ is  odd}}; 
\\
q_v^{n/2} +1 
 & {\text{if $n$ is  even}},   
  \ v \nmid 2, \ 
  d(\Lambda_v) =(-1)^{n/2}; 
\\
q_v^{n/2} -1 
 & 
 {\text{if $n$ is  even}},  \ 
 v \nmid 2, \ 
 d(\Lambda_v) \neq (-1)^{n/2}; 
\\
q_v^n-1 &  
{\text{if $n$ is even}}, \ v \mid 2, \ \Lambda_{v} \ {\text{normal in  RU}}; 
\\
q_v^{n/2} \cdot (q_v^n -1) 
 &  {\text {if $n$ is  even}}, 
 \ v \mid 2, \ \Lambda_{v} \  \text{normal in   RP}; 
\\ 
2  & 
 {\text{if $n$ is  even}}, 
\ v \mid 2,  \ \Lambda_{v} \ {\text{subnormal in  RU}};  
\\ 
q_v^{n/2} + 1 
& 
 {\text{if $n$ is  even}}, \ v \mid 2, 
\ \Lambda_{v} \ 
{\text{subnormal in   RP}},  \ 
d(\Lambda_v) =  (-1)^{n/2}; 
\\ 
q_v^{n/2} - 1 
& 
 {\text {if $n$ is even}}, \ v \mid 2, 
\ \Lambda_{v}
\ {\text{subnormal in   RP}},  \ d(\Lambda_v) \neq  (-1)^{n/2}. 
\end{dcases}
\end{align*}
Here, the determinant $d(\Lambda_v)$ takes value in the quotient group  $O_{F_v}^{\times}/\bfN_{E_v/F_v}(O_{E_v}^{\times})$. 
\end{thm}
\begin{proof}
The mass formula of Gan-Yu \cite[Theorem 10.20]{GY} shows that 
\begin{align}\label{GY}
\Mass(\Lambda)=
\left(\prod_{i=1}^n \frac{(2\pi)^i}{(i-1)!}\right) ^{-d} 
 \cdot  
 \lvert  O_F/d_{E/F} \rvert ^{\frac{n(n+1)}{4}}
\cdot 
\frac{\tau(G^1) \cdot d_F^{\frac{n^2}{2}}}{\prod_{v :  {\rm finite}}\beta_{v}},\end{align}
where $\beta_v=\beta_{\Lambda_v}$ denotes the local density of $\Lambda_v$, $\tau(G^1)$ denotes the  Tamagawa number of $G^1$, and  $d_F \in \Z$ denotes the absolute discriminant of $F$. 
By Kottwitz \cite{Kottwitz:Tamagawa} and Ono \cite[p.~128]{Ono}, we have $\tau(G^1)=2$. 

We describe the product $\prod_{v :  {\rm finite}}\beta_{v}$ by  the $L$-series. 
Let $L_{F_v}(j, \chi^j)$ be the local factor of the $L$-series at a finite place $v$ of $F$:
\[L_{F_v}(j, \chi^j)=\left( 1-\frac{\chi^j(v)}{q_v^j}\right)^{-1}. \]
If $v$ is unramified in $E$, then Lemma 
\ref{LD} implies that 
$ \beta_{v}^{-1}=\prod_{j=1}^n L_{F_v}(j, \chi^j).$ 
 For a place $v$ which is ramified in $E$,  we define a rational number $\lambda_v$ 
 by the relation 
\begin{align}\label{nu}
\beta_{v}^{-1} & =\frac{1}{2} \cdot 
\prod_{j=1}^n L_{F_v}(j, \chi^j)  \cdot \lambda_{v}. 
\end{align}
Then, Lemma \ref{LD} implies that 
\begin{align}\label{lambda}
\lambda_v=
\begin{dcases}
1 &
  {\text{if $n$ is  odd}}; 
\\
1+q_v^{-n/2}
 & {\text{if $n$ is even}},
  \ v \nmid 2, \ 
  d(\Lambda_v)=(-1)^{n/2}; 
\\
1-q_v^{-n/2}
 & {\text{if $n$ is  even}},  \ 
 v \nmid 2, \ 
 d(\Lambda_v) \neq (-1)^{n/2}; 
\\
1-q_v^{-n}
 & {\text{if $n$ is  even}}, \ v \mid 2, \ \Lambda_v \ {\rm normal};
\\ 
2 \cdot q_v^{-n} &  {\text{if $n$ is  even}}, 
\ v \mid 2,  \ \Lambda_v \ {\text{subnormal in RU}}; 
\\ 
q_v^{-n} \cdot 
(1+ q_v^{-n/2}) 
& 
{\text{if $n$ is even}}, \ v \mid 2, 
\  
 \Lambda_v \ 
{\text{subnormal in RP}}, \  d(\Lambda_v) = (-1)^{n/2};
\\ 
q_v^{-n} \cdot 
(1- q_v^{-n/2}) 
& 
{\text{if $n$ is  even}}, \ v \mid 2, 
\  \Lambda_v \ {\text{subnormal in  RP}}, \ d(\Lambda_v) \neq  (-1)^{n/2}.
\end{dcases}
\end{align}
The product of the local densities $\beta_{v}$ of $\Lambda_v$ over all finite places $v$ of $F$ can be written as
\begin{equation*}
  \prod_{v : {\rm finite}} 
\beta_{v}^{-1} = \frac{1}{2^w}\cdot 
\prod_{j=1}^n L_F(j, \chi^j) \cdot \prod_{v \mid d_{E/F}}\lambda_v.\end{equation*}

Let $\zeta_F(s)$ and $\zeta_E(s)$ be the Dedekind zeta functions of $F$ and $E$, respectively.
Then 
\begin{equation*}
\zeta_E(s)=\zeta_F(s) \cdot L_F(s, \chi).  
\end{equation*}
The functional equations for $\zeta_F(s)$ and $\zeta_E(s)$ imply  that 
\begin{align*}
L_F(j, \chi^j) & =\zeta_F(j)= 
\left(\frac{(-1)^{\frac{j}{2}}}{2} 
\cdot 
\frac{(2\pi)^j}{(j-1)!} \right)^d
\cdot  \zeta_F(1-j) \cdot d_F^{\frac{1}{2}-j}  & &  {\text{if $j$ is even}};
\\ 
 L_F(j, \chi^j) 
  & =L_F(j, \chi)
 \\
 &  = \left( \frac{(-1)^{\frac{j-1}{2}}}{2} \cdot \frac{(2 \pi)^{j}}{(j-1)!}\right)^d \cdot  L_F(1-j, \chi) \cdot d_F^{\frac{1}{2}-j} \cdot \lvert O_F/ d_{E/F}  \rvert^{\frac{1}{2}-j} & &  {\text{if $j$ is odd}}.
 \end{align*}
 
These equalities and    \eqref{GY} imply   that 
\begin{align}\label{eq:fe}\Mass(\Lambda)=\frac{(-1)^s \cdot 2\cdot \gamma}{2^{nd+w}}
\cdot 
\prod_{j=1}^n L_{F}(1-j, \chi^j) 
\cdot 
\prod_{v \mid d_{E/F}} \lambda_v
\end{align}
where 
\[
 \gamma  =
\begin{cases} 
1 &  {\text{if $n$ is  odd}}; 
\\
\lvert O_F/ d_{E/F} \rvert ^{n/2} &  {\text{if $n$ is  even}}. 
\end{cases}
\]

For a finite place $v$ which is ramified in $E$,   we put  
\begin{align}\label{kappa}
\kappa_v=
\begin{cases}
 \lambda_v & 
  {\text{if $n$ is odd}};
\\
\lvert O_{F_v}/  d_{E_v/F_v} \rvert  ^{n/2} \cdot \lambda_v & {\text{if $n$  is  even}}, 
\end{cases}
\end{align}
where $d_{E_v/F_v}$ denotes the discriminant of $E_v/F_v$. 
Then $\prod_{v \mid d_{E/F}} \kappa_v=\gamma \cdot \prod_{v \mid d_{E/F}} \lambda_v$. 

As in Section ~\ref{disc}, 
if $v$ is   ramified in $E$, then 
 \begin{align}\label{norm}
 \lvert O_{F_v}/ d_{E_v/F_v} \rvert =
 \begin{cases}q_v & {\text{if}} \  v \nmid 2; 
 \\
 q_v^2 & {\text{if $v \mid 2$ and $E_v/F_v$ is RU}}; 
 \\ 
 q_v^3 & {\text{if $v \mid 2$ and $E_v/F_v$ is RP}}.
 \end{cases} 
 \end{align}
 Equalities (\ref{lambda}), (\ref{eq:fe}),  (\ref{kappa}),  and  (\ref{norm}) imply the assertion. 
\end{proof}
\section{Connected components of complex unitary Shimura varieties} 
\label{sec:P.2}
\subsection{Similitude factors  of  Hermitian lattices}\label{ssec:sim}
\subsubsection{}\label{simV}
Let $\ell$ be a prime. 
 Let $(E, \bar{\cdot})$ be an \'{e}tale quadratic algebra over $F=\Q_{\ell}$ with involution as in Section ~\ref{setting}, that is,  
 $E/\Q_{\ell}$ is a quadratic field extension, or  $E \simeq 
 \Q_{\ell}\oplus \Q_{\ell}$.  
 Denote by $O_E$ the maximal order in $E$. 
Let 
 ${\bf N}=\mathbf N_{E/\Q_{\ell}}$ denote the norm map of $E/\Q_{\ell}$. 
 
Let $(V, \varphi)$  be a Hermitian space over $E$, and let  $n=\dim_{E} V$.  
We write $G=\GU(V, \varphi)$ for the \emph{unitary similitude group} associated to $(V, \varphi)$, which is a reductive group over $\Q_{\ell}$ defined by 
\begin{align}\label{def:GU}
G(R)= \{(g, c)  \in (\End_{E \otimes_{\Q_{\ell}} R}( V_R) )^{\times} \times R^{\times}  \mid 
\varphi(gx, gy)=c \cdot \varphi(x, y) \ {\text{for}} \  x, y \in V_R\}\end{align}
for any $\Q_{\ell}$-algebra $R$. Here we write  $V_R=V \otimes_{\Q_{\ell}}R$. 
The similitude character is defined by the second projection: 
\[\Sim : G \to \mathbb G_m, \ (g, c) \mapsto c.\]
Its kernel $G^1$  is the unitary group ${\rm U}(V, \varphi)$ (Section ~\ref{groups}).
Note that the similitude factor $c$    is uniquely determined by $g$. 
We write ${\rm sim}(g)\coloneqq {\rm sim}((g, c))=c$ by abuse of notation.  

Let $E = \Q_{\ell} \oplus  \Q_{\ell}$. 
Let $\{e_i\}$ be a basis of $V$ over $E$ and write $\varphi(e_i, e_j)=(a_{ij}, b_{ij}) \in E$. 
Then  $a_{ij}=b_{ji}$.  The matrix  $\Phi_1 \coloneqq (a_{ij})$ is invertible since $\varphi$ is  non-degenerate.
We have that 
\begin{equation}\label{eq:sp}
    G(\Q_{\ell}) \simeq \{(c, A, B) \in \Q_{\ell}^{\times} \times  \GL_n(\Q_{\ell})^2  \mid  B^{t} \Phi_1  A=c \cdot  \Phi_1\} 
\simeq \Q_{\ell}^{\times} \times 
 \GL_n(\Q_{\ell})
 \end{equation}
 where the second morphism is given by  $(c, A, B)\mapsto (c, A)$. 
Then the  similitude character is identified with the first  projection 
and hence it  maps $G(\Q_{\ell})$ onto $\Q_{\ell}^{\times}$. 
If $E/\Q_{\ell}$ is a quadratic field extension, 
then the similitude character maps $G(\Q_{\ell})$ onto $\bfN(E^{\times})$ or $\Q_{\ell}^{\times}$ according as $n$ is odd or even (see \eqref{eq:nu_ell}).

\subsubsection{}\label{sss:direct}
Here we assume $n=2$. We recall the  quaternion algebra associated to $(V, \varphi)$ which was 
 constructed by Shimura  \cite[Section 2]{Shimura0}.  
We fix a basis of $V$ and write $\Phi$ for its Gram matrix. We identify   $\End_{E}(V)$ with $\Mat_2(E)$. 
Let $\iota$ denote its 
 canonical involution,  given by 
$
\begin{pmatrix}
a & b 
\\
c & d 
\end{pmatrix}^{\iota} 
 =\begin{pmatrix}
d & -b
\\
-c & a
\end{pmatrix}$. 
We define a subring 
$B$ of $\End_{E}(V)$ by 
\begin{align}\label{B}
B   \coloneqq 
\{  
g \in \End_{E}(V)  \mid  
g^{\iota}
= 
\Phi^{-1} 
\cdot 
\bar{g}^t 
\cdot \Phi
\}.
\end{align}
We observe that each element $g$ of $B$ satisfies $\bar{g}^t \cdot \Phi \cdot g=\Phi \cdot g^{\iota} \cdot g=\det(g) \cdot \Phi$. 
By \cite[Propositions 2.6 and 2.8]{Shimura0}, the subring $B$ satisfies the following properties:
 \begin{itemize}
 \item 
 $B$ is a quaternion algebra over $\Q_{\ell}$; 
    \item $B^{\times}= 
   \{ g \in G(\Q_{\ell})  \mid \det(g)={\rm sim}(g)\}$ where we regard $G(\Q_{\ell})$ as a subset of $\End_{E}(V)$; 
   \item $B$ is a division algebra if and only if $(V, \varphi)$ is anisotropic.
   \end{itemize}
   \subsubsection{}
Let $\Lambda \subset V$ be a Hermitian  $O_E$-lattice, not necessarily unimodular for a moment. 
Let ${\rm K}={\rm K}_{\Lambda} \coloneqq {\rm Stab}_{G(\Q_{\ell})}\Lambda$  denote the stabilizer of $\Lambda$ in $G(\Q_{\ell})$. 
Then the similitude character induces  a homomorphism 
from ${\rm K}$ to $\Z_{\ell}^{\times}$. 
We compute its image, denoted by  $\Sim (\K)$.

 We fix a basis $\{e_i\}$ of  the lattice $\Lambda$ over $O_E$, inducing an isomorphism $\End_{O_E}(\Lambda) \simeq \Mat_n(O_E)$. 
Let $\Phi$ be the Gram matrix of $\{e_i\}$.  
Then we have an identification 
 \begin{align}\label{mat}
 {\rm K} \simeq \{ (g, c) \in \GL_{n}(O_{E}) \times \Z_{\ell}^{\times} \mid 
\bar{g}^t \cdot \Phi \cdot  g= c  \cdot  \Phi
 \}. 
\end{align} 

Let  $c \in 
 \mathbf N(O_{E}^{\times})$ and 
  let $u$ be an element  in $O_{E}^{\times}$ such that $\mathbf N(u)=c$. 
Then $(uI_n, c)\in \K$ where $I_n$ denotes the identity matrix. Hence we always have  inclusions  
\begin{equation}\label{eq:inc} 
 \mathbf N(O_{E}^{\times}) \subset {\rm sim}({\rm K}) \subset \Z_{\ell}^{\times}. 
\end{equation}
\begin{lemma}\label{multunram}
If $E/\Q_{\ell}$ is an unramified quadratic field extension or 
$E = \Q_{\ell} \oplus \Q_{\ell}$, 
then 
${\rm sim}({\rm K})=\Z_{\ell}^{\times}.$ 
\end{lemma}
\begin{proof}
If $E/\Q_{\ell}$ is an unramified field extension, then 
$\Z_{\ell}^{\times} = \mathbf N(O_{E}^{\times})$. 
Hence  ${\rm sim}({\rm K})=\Z_{\ell}^{\times}$ 
by \eqref{eq:inc}. 
If $E = \Q_{\ell} \oplus  \Q_{\ell}$, then 
${\rm K}  
\simeq \Z_{\ell}^{\times} \times  \GL_n(\Z_{\ell})$,  
similarly to \eqref{eq:sp}. 
The similitude character is identified with the first projection, and hence   $\Sim(\K)=\Z_{\ell}^{\times}$.  
\end{proof}
\begin{lemma}\label{mult}
Assume that $E/\Q_{\ell}$ is a  ramified quadratic field extension. 

{\rm (1)}  
If  $n$ is odd, then 
 $
{\rm sim}({\rm K})= 
 \mathbf N(O_{E}^{\times}).
$
 
 {\rm (2)} 
 If $n$ is even and 
  $\Lambda$ is unimodular,  then 
 $
 {\rm sim}({\rm K})= \Z_{\ell}^{\times}. 
$
\end{lemma}
 \begin{proof}
 (1)   
As in (\ref{mat}), any element $g\in {\rm K}$ satisfies that ${\bf N}(\det(g))={\rm sim}(g)^n$, and equivalently $\bfN(\det(g) \cdot {\rm sim}(g)^{\frac{1-n}{2}})={\rm sim}(g).$ 
Thus we have ${\rm sim}({\rm K}) \subset \mathbf N(O_E^{\times})$ and hence ${\rm sim}({\rm K})=\mathbf N(O_E^{\times})$ by \eqref{eq:inc}.

(2) 
Let $\Lambda$ be a unimodular  lattice  and let $\Phi$ be the Gram matrix of a basis of $\Lambda$ over $O_E$. 
We show that for any element $c\in \Z_{\ell}^{\times}$ there exists $g \in \GL_{n}(O_{E})$ such that $\bar{g}^t \cdot \Phi \cdot  g= c \cdot  \Phi$.
We may assume that $\Lambda$ is  of rank two:  
In fact, the lattice $\Lambda$ has a splitting  $\Lambda=\bigoplus_{1 \leq i\leq n/2}\Lambda_i$ for some lattices $\Lambda_i$ of rank two as in  Lemma \ref{splitting}. 
Let  $\{e_1^i, e_2^i\}$  be a basis of $\Lambda_i$ over $O_E$ with Gram matrix $\Phi_i$ for $1 \leq i \leq n/2$. 
Then $\{(e_1^i, e_2^i)_{1 \leq i \leq n/2}\}$ is a basis of   $\Lambda$ and has Gram matrix $\Phi=\diag(\Phi_1, \ldots, \Phi_{n/2})$. 
Suppose that for an element $c \in \Z_{\ell}^{\times}$  there exist matrices $g_i \in \GL_{2}(O_{E})$ such that $
\bar{g_i}^t \cdot \Phi_i \cdot  g_i=c  \cdot  \Phi_i$ for all $i$.  
 Then the matrix  $g= \diag(g_1, \ldots, g_{n/2})$ satisfies $\bar{g}^t \cdot \Phi \cdot  g= c \cdot  \Phi$,  as desired.  
By Lemma \ref{splitting}, we only need to consider  three cases: (i) The case $\Lambda=H$; (ii) The case $\ell=2$, $\Lambda$ is subnormal in RP, and $\Lambda \not\simeq H$; (iii) The case
$\Lambda=(1) \oplus (-\alpha)$ for an $\alpha \in O_F^{\times}$. 
 
(i) {\it{The case}} $\Lambda=H=\begin{pmatrix}
 & 1 
\\ 
1 & 
\end{pmatrix}$. 
In this case, 
 for any $c \in \Z_{\ell}^{\times}$, the matrix 
$g= \begin{pmatrix}
c &  
\\ 
& 1
\end{pmatrix}$ satisfies the desired property.

(ii) {\it The case $\ell=2$, $\Lambda$ is subnormal in RP,  and $\Lambda \not\simeq H$}. 
Before discussing the case $\ell=2$, we recall some facts on a ramified quadratic  field extension $E/\Q_2$. We have that  $E=\Q_2(\sqrt{\theta})$ where  $\theta=3, 7, 2,6, 10$, or $14$. 
We consider the  inclusions 
 \[1+8\Z_{2}=(\Z_2^{\times})^{2} \subset \mathbf N(O_{E}^{\times}) \subset 
  {\rm sim}({\rm K}) \subset \Z_{2}^{\times}.\]
The quotient group $
\Z_2^{\times}/(\Z_2^{\times})^{2}$ is  identified with $(\Z/8\Z)^{\times}\simeq (\Z/2\Z)^{\oplus 2}$. Further, 
the subgroup $\mathbf N(O_{E}^{\times})/(\Z_2^{\times})^{2} \subset \Z_2^{\times}/(\Z_2^{\times})^{2}$ is of order two.  In the second column in Table \ref{table:2}, we see an example of  an element $a \in O_{E}^{\times}$ such that  $\mathbf N(a)\mod 8$ generates this  subgroup.  
To prove  ${\rm sim}({\rm K)}=\Z_2^{\times}$,  
it suffices to  find an element $g \in{\rm K}$  such  that ${\rm sim}(g) \not\equiv \bfN(a) \pmod 8$. 

Now let $\Lambda$ be as above. 
By Lemma \ref{subnormal} and Example~\ref{eg:subRP}, 
 we may choose a basis of $\Lambda$ so that $\Phi=\begin{pmatrix}
2 & 1
\\
1 & 2
\end{pmatrix}$.  
 Table~\ref{table:4} shows  explicit examples of elements  $g \in \GL_2(O_E)$  such that $\bar{g}^t \cdot \Phi \cdot  g= c \cdot  \Phi$ with  $c=\Sim g  \not\equiv \bfN(a) \pmod 8$. 

\begin{table}[hbtp]
  \caption{$\ell=2$, $\Lambda$ : subnormal, $E=\Q_2(\sqrt{\theta})$ :  RP, $\Lambda\not \simeq  H$}
  \label{table:4}
  \renewcommand{\arraystretch}{1.2}
  \centering
  \begin{tabular}{|c|c|c|c|}
    \hline
  $\theta$   & $g \in \GL_2(\Z_{2})$   &  ${\rm sim}(g)$ & ${\rm sim}(g) \mod 8$  \\ 
    \hline  
     $2$ & 
     $\begin{pmatrix}
2 & 3-\sqrt{2}
 \\ 
 1+ \sqrt{2} & 
 -2+\sqrt{2}
\end{pmatrix}$ & $5$ & $5$  \\ 
 \hline 
    $10$  & $\begin{pmatrix}
3- \sqrt{10} & 
2
 \\
 -2+ \sqrt{10} & 
 1+\sqrt{10}
\end{pmatrix}$ & $-3$ & $5$ 
    \\ 
\hline
    $6$ & $\begin{pmatrix}
-1- \sqrt{6} & 
1
 \\
 -1+ \sqrt{6} & 
 -2+\sqrt{6}
\end{pmatrix}$ & $-3$ & $5$  
    \\ 
    \hline 
    $14$ & $\begin{pmatrix}
-4- \sqrt{14} & 
-1
 \\
 1+ \sqrt{14} & 
 -3+\sqrt{14}
\end{pmatrix}$ & $-1$ & $7$  \\
    \hline
  \end{tabular}
\end{table}

(iii) {\it{The case 
$\Lambda =(1) \oplus (-\alpha)$ 
for an element $\alpha \in \Z_{\ell}^{\times}$}}. 
In this case, $\Lambda$ is  normal.  
 If $\alpha \in \bfN(O_{E}^{\times})$, then $d(\Lambda)=-1$ as elements of $\Z_{\ell}^{\times}/\bfN(O_{E}^{\times})$ and hence 
 $\Lambda \simeq H$. 
 Therefore we may assume that 
 $\alpha \notin \bfN(O_{E}^{\times})$. 
Then $V$ is  anisotropic by \cite[(3.1)]{Jacobowitz}. 
Let $B$ be the quaternion division algebra over $\Q_{\ell}$ associated with the space $V=\Lambda \otimes_{\Z_{\ell}} \Q_{\ell}$, as in \eqref{B}.  
A computation shows 
\begin{align}\label{direct} 
B  & = 
 \left\lbrace  \begin{pmatrix}
a & \alpha  \bar{c}
\\
c & \bar{a} 
\end{pmatrix} \in \Mat_2(E) 
  \mid 
 a, c \in E 
\right\rbrace.
 \end{align} 
 We put $O_{\Lambda} \coloneqq 
 \{ 
g \in B \mid  g \cdot \Lambda  \subset \Lambda 
\}=B \cap \Mat_2(O_E)$.  
Then $O_{\Lambda}$ is a $\Z_{\ell}$-order of $B$.  
Further, the second property of $B$  after \eqref{B} implies that  
$\det (O_{\Lambda}^{\times})\subset {\rm sim}({\rm K})( \subset \Z_{\ell}^{\times}).$ 
Hence it suffices to show that $\Z_{\ell}^{\times} \subset 
\det
(O_{\Lambda}^{\times})$. 

First assume  that $\ell \neq 2$. 
 Let $E \oplus E z$ be a two-dimensional 
 $E$-vector space,      equipped with a structure of a quaternion algebra over $\Q_{\ell}$  defined  by the relations 
\begin{align*} 
z^2  =\alpha, \ za=\bar{a}z, \ a  
\in E. 
\end{align*} 
Then there is an isomorphism of quaternion algebras
\[E \oplus Ez  \xrightarrow{\sim} B, \quad  
a+ bz \mapsto \begin{pmatrix}
a & \alpha  b
\\
\bar{b} & \bar{a}
\end{pmatrix}, \] 
which identifies the $\Z_{\ell}$-order $O_E \oplus O_E z$ with $O_{\Lambda}$. 
 Let $\pi$ be a  uniformizer of  $O_E$ such that $\bar{\pi}=-\pi.$ 
 Then  
 $\pi  O_{\Lambda}$ is  a two-sided ideal of $O_{\Lambda}$. 
 Further, we have an  isomorphism of $\F_{\ell}$-algebras  
\[O_{\Lambda}/\pi O_{\Lambda} \simeq  \F_{\ell}[T]/(T^2- \tilde\alpha),\]
where we write $\tilde{\alpha}=\alpha \mod \ell \in \F_{\ell}$. 
 This shows that $O_{\Lambda}/\pi O_{\Lambda}$ is a separable $\F_{\ell}$-algebra since $\ell \neq 2$. 
Note that $(\pi  O_{\Lambda})^{2} \subset \ell O_{\Lambda}$, and hence $\pi  O_{\Lambda}$ is contained in  the Jacobson radical of $O_{\Lambda}$, as in   \cite[Exercise 39.1]{Reiner}. 
 Suppose   $O_{\Lambda}/\pi O_{\Lambda}
 \simeq \F_{\ell} \times \F_{\ell}$.  
 Then 
 its idempotents have lifts in $O_{\Lambda}$ by 
 \cite[Theorem 6.7]{CR}, and this 
 contradicts to the fact that $B$ is a division algebra. 
It follows that $O_{\Lambda}/\pi O_{\Lambda}$ is  a quadratic field extension of $\F_{\ell}$, generated by $z \mod \pi$.
Therefore $\Q_{\ell}(z)$ is an unramified quadratic field extension of  $\Q_{\ell}$, with the  ring of integers $\Z_{\ell}[z]$. 
Hence we have  $\Z_{\ell}^{\times} =\mathbf N_{\Q_{\ell}(z)/\Q_{\ell}}(\Z_{\ell}[z]^{\times}) \subset \det (O_{\Lambda}^{\times})$, as desired. 

Next we assume that $\ell= 2$ and $\alpha \notin \mathbf N (O_{E})^{\times}$. 
By equality  \eqref{direct}, the  group $\det(O_{\Lambda}^{\times})$ consists of elements of the form $\mathbf N(a) - \alpha \mathbf N(c)$ for some $a, c \in O_E$. 
   Table \ref{table:2} shows  some examples of  $a \in O_{E}^{\times}$ and $c \in O_{E}$  such that  $\bfN(a)$ and $\mathbf N(a) - \alpha \mathbf N(c)  \mod 8$ generate the multiplicative group $(\Z/8\Z)^{\times}$. 
 The argument in case (ii) thus  implies $\det(O_{\Lambda}^{\times})=\Z_2^{\times}$.

 \begin{table}[hbtp]
  \caption{$E=\Q_2(\sqrt{\theta})$, $\Lambda=(1) \oplus (-\alpha)$, $\alpha\notin \mathbf N(O_{E}^{\times})$}
  \label{table:2}
  \renewcommand{\arraystretch}{1.2} 
  \centering
  \begin{tabular}{|c|c|c|c|c|c|c|}
    \hline
  $\theta$   & $a \in O_{E}^{\times}$ &  $\mathbf N(a) \ (8)$ &   $c \in O_{E}$  & $\bfN (c) \ (8)$ & $\alpha \  (8)$ & $\mathbf N(a)-
 \alpha \mathbf N(c) \ (8)$    
    \\ 
    \hline 
    $3$  & $1+2\sqrt{3}$ & 
    $5$ & $1+\sqrt{3}$ & $6$ & $3, 7$  &  $3$ 
    \\ 
    \hline
    $7$ & $1+2\sqrt{7}$ & 
    $5$ &  $1+\sqrt{7}$ 
    & $2$  
    & $3, 7$  &  $7$ 
    \\ 
    \hline
    $2, 10$ & $1+\sqrt{\theta}$ &  
    $7$ &   $2$ 
    & $4$  &$3, 5$  &  $3$ 
    \\  
    \hline
    $6, 14$ & $1+\sqrt{\theta}$ & $3$  & $2$ & $4$   & $5, 7$ & $7$ 
      \\
    \hline
  \end{tabular}
\end{table}
\end{proof}

\subsection{Unitary Shimura varieties}
\label{ss:Sh}
\begin{defn}
A \emph{unitary PEL datum} is a $5$-tuple $\mathscr D_{\Q}=(E, \bar{ \cdot}, V, \< \, , \, \>, h_0)$ where 
\begin{itemize}
    \item 
 $E$ is an  imaginary quadratic extension of $\Q$; 
\item  $b \mapsto \bar{b}$ for $b \in E$ is the non-trivial automorphism of $E/\Q$; 
\item $V$ is an $E$-vector space of dimension $n>0$ with a  non-degenerate alternating 
  $\Q$-bilinear  form  $\< \, , \, \> :V\times V\to \Q$  such that
  \[\< bx,y\>=\< x, \bar{b} y\>\]
  for all $x,y\in V$ and $b\in E$; 
\item $h_0:\C \to \End_{E \otimes \R}(V_\R)$ is an $\R$-linear algebra
  homomorphism, 
  such that 
\[ \< h_0(z)x, h_0(z)y
\>=\<x,y\> \quad \text{for all\ }  z \in \C, \   x,y \in
 V_\R=V\otimes\R, \]
and that the pairing  $(x,y)\mapsto \<x,h_0(\sqrt{-1})y\>$ is symmetric and  
definite (positive or negative) on $V_\R$. 
\end{itemize}
\end{defn}
  Let ${\bf G}=\GU(V,\< \, , \, \>)$ be the $\Q$-group of $E$-linear
$\< \, , \, \>$-similitudes on $V$: For every commutative $\Q$-algebra $R$, the
group of its $R$-values is given by 
\begin{equation}
  {\bf G}(R)=\{(g, c) \in \End_E(V_R)^{\times} \times R^{\times}  \mid   \< gx,gy\>=c \cdot \< x,y\> \ {\text{for all}} \ x,y\in
  V_R \}
\end{equation}
where $V_R=V\otimes_\Q R$.

We may write $E=\Q(\sqrt{d_E})$, where $d_E$ is the discriminant of $E$. 
Then there exists (see \cite[Lemma A.7]{Milne})  a unique non-degenerate Hermitian form $\varphi  : V \times V \to E$ such that 
$\< \, , \, \>$ can be written as 
\begin{equation}\label{psi}
    \< x, y\>=\Tr_{E/\Q}\Big(\big(\sqrt{d_E}\big)^{-1} \cdot  \varphi(x,  y)\Big). 
    \end{equation}
Further, the algebraic $\Q$-group $\mathbf G$ is  isomorphic 
to the unitary similitude group 
$\GU(V, \varphi)$ as defined in \eqref{def:GU}. 
In particular, 
one has ${\bf G}_{\R}\simeq \GU(r,s)$ over $\R$, where $(r,s)$ is the signature of $\varphi_{\R}$. 
If we replace $\sqrt{d_E}$ by $-\sqrt{d_E}$ in \eqref{psi}, then ${\bf G}_{\R}\simeq \GU(s,r)$; however, we have $\GU(r,s)=\GU(s,r)$.
 
We define a homomorphism $h : \Res_{\C/\R}\bbG_{{\rm m}, \C} \to \bfG_{\R}$ by restricting $h_0$ to $\C^{\times}$. 
Composing $h({\C})$ with the map $\C^{\times} \to \C^{\times} \times \C^{\times}$ where $z \mapsto (z, 1)$ then gives $\mu_h : 
\C^{\times} \to \bfG(\C) \simeq \GL_{n}(\C) \times \C^{\times}$.  
Up to conjugation, 
we have $\mu_h=(\diag(z^r, 1^s),z)$ or $(\diag(z^s, 1^r),z)$. Let $\mathbf E$ be the reflex field associated with $\mathscr D_{\Q}$, that is, the field of definition of the  conjugacy class of  $\mu_h$. 
Then 
 either $\bfE=\Q$ when $r=s$ or $\bfE=E$ when $r \neq s$.

Let $X$ be the ${\bf G}(\R)$-conjugacy class of $h$. 
Then $({\bf G},X)$ is  a Shimura datum\footnote{If $rs=0$, it does not satisfy the axiom  \cite[(2.1.1.3)]{Deligne}, however, we still can consider the Shimura variety associated to it using \eqref{eq:ShC}.}. 
For any open compact
subgroup ${{\rm K}}\subset  {\bf G}(\A_f)$, we write $\Sh_{{{\rm K}}}({\bf G},X)$ for the Shimura variety of level
${{\rm K}}$ associated to $({\bf G},X)$. 
Then  $\Sh_{{{\rm K}}}({\bf G},X)$  is a smooth quasi-projective variety of dimension $rs$ over $\mathbf E$. 
The set of complex points of $\Sh_{{{\rm K}}}({\bf G},X)$ is identified, 
as a complex manifold, with
\begin{equation}\label{eq:ShC}
\Sh_{\rm K}(\mathbf G, X)({\C})  = 
{\bf G}(\Q)\backslash X \times {\bf G}(\A_f)/{\rm K}.     
\end{equation}



\begin{defn}\label{P.6}  
   An \emph{integral  unitary PEL  datum}  is a tuple 
    $\mathscr D=(E, \bar{\cdot},O_E, V,\< \, , \, \>, \Lambda, h_0)$, where 
    \begin{itemize}
    \item $(E,\bar{\cdot},V,\< \, , \, \>, h_0)$ is a unitary PEL datum;
    \item $O_E$ is the ring of integers of $E$; 
    \item $\Lambda$ is a full $O_E$-lattice in $V$ on which $\< \, , \, \>$ takes value in $\Z$. 
    \end{itemize}
\end{defn}
Equality \eqref{psi} implies the following. 
 \begin{lemma}\label{selfdual}
Let $\Lambda^{\vee, \< \, , \, \>}=\{ x \in V \mid \< x, \Lambda\> \subset \Z\}$ be the dual lattice of $\Lambda$ with respect to $\< \, , \, \>$, and let $\Lambda^{\vee, \varphi}$ be the dual lattice with respect to $\varphi$ as defined in Section ~\ref{sec:hermitian}. 
Then $\Lambda^{\vee, \< \, , \, \>}=\Lambda^{\vee, \varphi}$ as $O_E$-submodules of $V$. 
In particular, $(\Lambda, \< \, , \, \>)$ is self-dual if and only if $(\Lambda, \varphi)$ is self-dual (unimodular). 
 \end{lemma}
 Let us fix an 
 integral unitary PEL datum $\mathscr D=(E, \bar{\cdot}, O_E, V, \< \, , \, \>, \Lambda,h_0)$. 
This gives a model  $\GU(\Lambda, \< \, , \, \>)=\GU(\Lambda, \varphi)$ of $\bfG$ over $\Z$, which is again denoted by ${\bf G}$. 
 Its group of $R$-values for a commutative ring $R$ is  given by 
 \begin{equation*}
  \mathbf G(R)=\{(g, c) \in \End_{O_E \otimes R}(\Lambda_R)^{\times} \times R^{\times}  \mid   \varphi(gx,gy)=c \cdot \varphi(x,y) \ {\text{for all}} \ x,y\in
  \Lambda_R \}
\end{equation*}
where $\Lambda_R= \Lambda \otimes R$. 
\subsubsection{}\label{tori}
Let ${\bf G}^{\der}=\SU(V, \varphi)$ be the derived subgroup of $\bfG_{\Q}$. 
Then ${\bf G}^{\der}$  is a simply connected semisimple algebraic group over $\Q$. 
Kneser's theorem 
 states that  the group $H^1(\Q_{\ell}, \bfG_{\Q_{\ell}}^{\der})$ is trivial for every finite prime $\ell$ \cite[Theorem 6.4, p.~284]{PR}.  Further, the Hasse principle holds true, that is,  the map $H^1(\Q, \bfG^{\der}) \to \prod_{\ell \leq \infty}H^1(\Q_{\ell}, \bfG^{\der})$ is injective 
 \cite[Theorem 6.6, p.~286]{PR}.

Let $D$  be the quotient  torus  $D\coloneqq \bfG_{\Q}/{\bf G}^{\der}$ over $\Q$, and let $\nu : \bfG_{\Q} \to D$ be the 
natural projection.  
 Kottwitz \cite[Section ~7]{Kottwitz} 
described $D$ and $\nu$ as follows. 
  Let $T^E$ denote the Weil restriction 
$\Res_{E/\Q} 
\mathbb G_{{\rm m}, E}$, and let $T^{E,1}$ denote the kernel of the norm homomorphism  $\bfN=\bfN_{E/\Q} : T^E \to \mathbb G_{{\rm m},\Q}$.  
Then,  
 for any $\Q$-algebra $R$ we have an isomorphism 
\begin{align}
\label{def:D}
D(R) \simeq \{ (x, c) \in T^E(R)  \times R^{\times}  \mid 
\mathbf N(x)=c^n \}, 
\end{align}
under which  $\nu$ is identified with the product  $\det \times {\rm sim},  (g, c)\mapsto (\det g, c)$. 
Furthermore  we have isomorphisms 
\begin{align}\label{torus}
f : D  \xrightarrow{\sim} 
\begin{cases}
T^E 
\\ 
T^{E,1} \times \bbG_{{\rm m}, \Q}
\end{cases}
\quad 
(x, c) \mapsto 
\begin{cases}
x c^{\frac{1-n}{2}} & {\text{if $n$ is odd}}; 
\\ 
(xc^{-n/2}, c) & 
{\text{if $n$ is even}}.
 \end{cases}
\end{align}
The similitude character $\Sim: \bfG_{\Q} \to \bbG_{{\rm m}, \Q}$  is  equal to $\bfN \circ  f \circ \nu$ if $n$ is odd,  and ${\rm pr}_2 \circ f \circ \nu$ if $n$ is even. 

Let $D(\R)^0$ be the identity component of $D(\R)$. 
 From \eqref{torus} it follows  that $D(\R)^0\simeq \C^{\times}$ if $n$ is odd and $D(\R)^0 \simeq S^1 \times \R_{>0}$ if $n$ is even, where $S^1$ denotes the  unit circle. 
We write $E^1 \coloneqq T^{E, 1}(\Q)= \{x \in E^{\times} \mid \bfN(x)=1\}$. Then 
\begin{align}\label{D_infty}
    D(\Q)_{\infty} \coloneqq D(\Q) \cap D(\R)^0=
    \begin{cases*}
    E^{\times} & if $n$ is odd; 
    \\ 
    E^{1} \times \Q_{>0} & if $n$ is even. 
    \end{cases*}
\end{align}

For a prime $\ell$, we write  $E_{\ell}= E \otimes_{\Q} \Q_{\ell}$. 
By Kneser's theorem, we have that 
\begin{align}\label{eq:nu_ell}
    \nu(\bfG(\Q_{\ell}))=D(\Q_{\ell})=
    \begin{cases*}
    E_{\ell}^{\times} 
    \\
    E_{\ell}^{1} \times \Q_{\ell}^{\times},  
    \end{cases*}
    \quad 
    \Sim(\bfG(\Q_{\ell}))=\begin{cases*}
    \bfN(E_{\ell}^{\times}) & if $n$ is odd; 
    \\ 
    \Q_{\ell}^{\times} & if $n$ is even. 
    \end{cases*} 
\end{align}

 For any algebraic torus $T$ over $\Q$, let $T(\Z_{\ell})$ denote the unique maximal open compact subgroup of $T(\Q_{\ell})$. We write 
  $O^1_{E_{\ell}}  \coloneqq T^{E, 1}(\Z_{\ell})= \{x \in O_{E_{\ell}}^{\times} \mid \bfN(x)=1\}$. 
  For $\ell=2$, 
we also define a subgroup $O_{E_{2}}^0$ of $O^{\times}_{E_{2}}$ by 
$O_{E_{2}}^0  \coloneqq \{ u \bar{u}^{-1}  \mid  u \in O_{E_2}^{\times} 
\}$. Then $O_{E_2}^0 \subset O_{E_2}^1$. 
When $E_2/\Q_2$ is ramified,  
 Hilbert's Theorem 90 implies that  $[O^1_{E_{2}}:O^0_{E_2}]=2$, as in \cite[Remark 3.4]{Kirschmer}. 

We write $\Lambda_{\ell} \coloneqq \Lambda \otimes \Z_{\ell}$.  
 By 
 Kirschmer \cite[Theorem 3.7]{Kirschmer}, the subgroup $\det ({\bf G}^1(\Z_{\ell})) \subset O^1_{E_{\ell}}$ is   described as 
\begin{align}\label{det}
\det({\bf G}^{1}(\Z_\ell))=\begin{cases}O_{E_{2}}^0  &  {\text{if}} \ \ell=2,  \  E_2/ \Q_2 \text{\ is RU}, 
 \ \Lambda_2 \simeq H^{n/2}; 
\\ 
O^1_{E_{\ell}} &   {\text{otherwise}}, 
\end{cases}
\end{align}
where $H=\begin{pmatrix}
   & 1
 \\ 
 1 & 
\end{pmatrix}$  denotes  a rank-two $O_{E_2}$-lattice. 
\begin{lemma}\label{D=nu}
Let $\ell$ be a prime. 
If $n$ is odd, or $n$ is even and $(\Lambda, \varphi)$ is unimodular, then
\[ D(\Z_{\ell}) =\nu(\bfG(\Z_{\ell})), \]
unless $n$ is even, $E_2/\Q_2$ is RU,  and  $\Lambda_2 \simeq H^{n/2}$, in which case
\[ [D(\Z_2) : \nu(\bfG(\Z_2))]=2.\] 
\end{lemma}
\begin{proof}
There is a commutative diagram 
\begin{align}\label{diagram}
 \begin{CD}
  1 @>>>
  {\bf G}^{1}(\Z_{\ell})   
  @>>> 
 {\bf G}(\Z_{\ell}) 
  @>{{\rm sim}}>> 
  {\rm sim}({\bf G}(\Z_{\ell})) 
  @>>> 1
  \\ 
  @. 
   @V{\det}VV 
 @V{\nu}VV
 @VVV
 @. 
  \\ 
  1 @>>> 
  O^1_{E_{\ell}}
  @>>>
  D(\Z_{\ell}) 
  @>{\rm{pr}_2}>>
 {\rm{pr}}_2(D(\Z_{\ell})) 
  @>>> 1
  \end{CD}
\end{align}
  where the horizontal sequences are exact. 
 By the description of $D$ in \eqref{torus},  the group  
${\rm {pr}}_2(D(\Z_{\ell}))$ equals  $
\mathbf N(O_{E_{\ell}}^{\times})$ or $
{\Z}_{\ell}^{\times}$ according as $n$ is odd or even.  
 It follows from   Lemmas \ref{multunram} and \ref{mult}  that the right vertical arrow is an isomorphism. 
By the snake lemma, we have 
\begin{equation} \label{D_ell}
     [D(\Z_{\ell}) : \nu\big({\bf G}(\Z_{\ell})\big)] =
[ O^1_{E_{\ell}}: \det ({\bf G}^{1}(\Z_{\ell}))]. 
\end{equation}
This and \eqref{det} imply the assertion.
\end{proof}
\subsection{The number of connected components} 
\subsubsection{}
Let $\K$ be any open compact subgroup of $\bfG(\A_f)$. 
Let  $\Sh_{\K}({\bf G}, X)_{\C}$ be the Shimura variety of level $\K$, 
and  $\pi_0(\Sh_{\K}({\bf G}, X)_{\C})$ be the set of its connected components. 
Let $X^+$ be the connected component of the Hermitian symmetric domain $X$ containing  the base point $h_0$. 
We write  $\bfG(\R)_+$ for  the stabilizer of $X^+$ in $\bfG(\R)$, and write $\bfG(\Q)_+ \coloneqq \bfG(\Q) \cap \bfG(\R)_{+}$. 
As in \cite[2.1.3]{Deligne}, we have 
\[\pi_0(\Sh_{\K}({\bf G}, X)_{\C}) \simeq \bfG(\Q)\backslash 
\bfG(\A)/ \bfG(\R)_{+} {\K} \simeq  \bfG(\Q)_+  \backslash \bfG(\A_f)/\K. \] 
Kneser's theorem and the Hasse principle for $\bfG^{\der}$ imply  that  $\nu(\bfG(\Q)_+)=D(\Q)_{\infty}$, and hence 
$\nu$ induces  a  surjective map 
\begin{equation}\label{eq:comp_doublecoset}
\nu : 
\bfG(\Q)_+  \backslash \bfG(\A_f)/\K 
\to 
D(\Q)_{\infty}    \backslash D(\A_f)/\nu(\K).    
\end{equation}

Now we assume that $rs>0$. Then the group  $\bfG^{\rm der}(\R)$ is not compact. 
The strong approximation theorem (\cite{Kneser}, \cite[Theorem 7.12, p.~427]{PR})  therefore implies that  $\bfG^{\der}(\Q)$ is dense in $\bfG^{\der}(\A_f)$. 
It follows that the map in  \eqref{eq:comp_doublecoset} is a bijection. 
For more details, see  \cite[Section 5]{Milne}.  


\begin{thm}\label{connected}
Assume that $rs>0$. We write    $D(\widehat{\Z})\coloneqq \prod_{\ell : {\rm prime}}D(\Z_{\ell}) \subset D(\A_f)$. 

\noindent 
{\rm (1)} 
If $n$ is odd, or $n$ is even and  $\Lambda$ is unimodular, then  
\begin{align*}
\lvert \pi_0(\Sh_{{\bf G}(\widehat{\Z})}({\bf G}, X)_{\C}) \rvert
 = 
\begin{dcases}
h(E) \ &   {\text{if $n$ is odd}}; 
\\ 
2^{2-w} \cdot h(E) \ &    {\text{if $n$ is even}}, \  \Lambda_{2} \simeq H^{n/2},  \ d_E \equiv 4 \pmod 8, \ E \neq \Q(\sqrt{-1}); 
\\ 
2^{1-w} \cdot h(E)  \ & {\text{otherwise}},
\end{dcases}
\end{align*} 
where $h(E)$ denotes the class number of the field $E$, 
$w$ denotes the number of primes which are ramified in $E$. 

\noindent 
{\rm (2)}
For an integer $N \geq 3$,  let  $\K(N)$ be the kernel of the reduction  map  ${\bfG}(\widehat{\Z}) \to {\bfG}(\widehat{\Z}/N\widehat{\Z})$. 
 Then 
\begin{align*}
\lvert \pi_0(\Sh_{{\rm K}(N)}({\bf G}, X)_{\C}) \rvert
 = 
\begin{cases*}
[D(\widehat{\Z}) : \nu(\K(N))] \cdot \lvert \mu_E \rvert^{-1} \cdot  h(E) & if $n$ is odd; 
\\
[D(\widehat{\Z}) : \nu(\K(N))] \cdot \lvert \mu_E \rvert^{-1}   \cdot 2^{1-w} \cdot h(E) & if $n$ is even,
\end{cases*}
\end{align*} 
where $\mu_E$ denotes the group of roots of unity in $E$.
\end{thm}
The case $rs=0$ will be treated in Example \ref{rs=0}. 
\begin{proof}


We regard the group  $\mu_E$  as  a subgroup of 
$T^{E, 1}(\A_f)$ via  the natural embedding.   Note that 
$T^{E,1}$ can be embedded in  $D$ by \eqref{def:D}. 
Let $\mu_E \cap \nu(\K)$ be the intersection  taken in  $D(\A_f)$. 
Further, let 
$h(D)\coloneqq \lvert  D(\Q) \backslash D(\A_f)/ D(\widehat{\Z}) \rvert$ denote   the class number  of $D$.
 Then, by 
 \cite[Lemma 6.1 and Theorem 6.3]{GSY}, we have an equality 
\begin{align}\label{pi_0G_1}
\lvert  \pi_0(\Sh_{\K}({\bf G}, X)_{\C}) \rvert   = \lvert D(\Q)_{\infty}  \backslash D(\A_f)/\nu(\K) 
\rvert = \frac{\Big[D(\widehat{\Z}) : \nu(\K)\Big]}{\Big[  
\mu_E:\mu_E \cap \nu(\K)\Big]} \cdot h(D). 
\end{align}

(1)
We apply formula \eqref{pi_0G_1} to $\K=\bfG(\widehat{\Z})={\rm Stab}_{\bfG(\A_f)}{\Lambda}$. 
First 
 recall that  $E_2/\Q_2$ is RU if and only if $d_E \equiv 4 \pmod 8$ as in Section ~\ref{disc}. 
It follows from Lemma \ref{D=nu} that 
\begin{align}\label{eq:hat}
\begin{split} 
  & [D(\widehat{\Z}) : \nu\big({\bf G}(\widehat{\Z})\big)]  = \prod_{\ell : {\text{prime}}}
 [ D(\Z_{\ell}): \nu ({\bf G}(\Z_{\ell})) ] 
  =   
\begin{cases}
2 
    & \text{if $d_E \equiv 4\!\!\!\! \pmod 8$ and }   
  \Lambda_2 \simeq H^{n/2}; 
\\ 
1 &   {\text{otherwise}}.
\end{cases}
\end{split}
\end{align}

Next we compute the index  $[\mu_E: \mu_E
 \cap  \nu({\bf G}(\widehat{\Z}))]$. 
  If $d_E \not\equiv 4 \pmod 8$ or $\Lambda_2\not\simeq H^{n/2}$, then 
   formula \eqref{eq:hat} implies that $D(\widehat{\Z})=\nu(\bfG(\widehat{\Z}))$ and hence 
$\mu_E= 
  \mu_E
 \cap 
  \nu\big({\bf G}(\widehat{\Z})\big).$ 
Now we assume that $d_E \equiv 4 \pmod 8$ and $\Lambda_2 \simeq H^{n/2}$.  
Note that we have 
$\det(\bfG^1(\widehat{\Z}))=\widehat{O}^1_{E} \cap \nu(\bfG(\widehat{\Z}))$ by diagram chasing in \eqref{diagram}, and hence $\mu_E \cap \nu(\bfG(\widehat{\Z}))= \mu_E \cap \det
   (\bfG^1(\widehat{\Z}))$. 
   This and \eqref{det} imply that 
   $[ \mu_E: \mu_E \cap \nu(\bfG(\widehat{\Z}))] 
   =[\mu_E : \mu_E \cap O_{E_2}^0]$,  
   where we regard $\mu_E$ as a subgroup of  $O_{E_2}^{\times}.$ 
Recall that  $E_2=\Q_2(\sqrt{\theta})$ for a $\theta \in \Z_2^{\times}$. If we put $a={\sqrt{\theta}}$,   then $a\in O_{E_2}^{\times}$  and   $\bar{a}=-a$. 
 Hence  $-1=a \bar{a}^{-1} \in O_{E_{2}}^0$. 
If we further assume $E \neq \Q(\sqrt{-1})$, then 
$\{ \pm 1\}=\mu_E=\mu_E \cap O_{E_{2}}^0$. 
Suppose that  $E=\Q(\sqrt{-1})$. Then the elements $\pm \sqrt{-1}-1$ are uniformizers of $O_{E_2}$.  
On the other hand,  
we have $O_{E_{2}}^0 = \{ x \in O_{E_2}^1 \mid  x -1 \in \mathfrak D_{E_2/\Q_2}\}$ by  \cite[Lemma 3.5]{Kirschmer}    where $\mathfrak D_{E_2/\Q_2}$ denotes the different of $E_2/\Q_2$, and further  $\mathfrak D_{E_2/\Q_2}=(2)$ as in Section ~\ref{disc}. 
It follows that 
 $\pm \sqrt{-1} \notin O_{E_{2}}^0$ and hence  $[ \mu_E : \mu_E \cap O_{E_{2}}^0 ]=[\{ \pm 1, \pm \sqrt{-1}\} : \{\pm 1\}]= 2$.   
As a result, 
\begin{align}
\label{eq:inf}
\begin{split}
  [ \mu_E 
:
  \mu_E 
 \cap 
  \nu\big(\bfG(\widehat{\Z})\big)
  ]
= 
\begin{cases}  
2 
 &  \text{if $E = \Q(\sqrt{-1}) \ (d_E \equiv 4 \bmod 8)$ and   $\Lambda_2 \simeq H^{n/2}$;} \\  
 1 & {\text{otherwise}}.
\end{cases}
\end{split}
\end{align}

Equalities \eqref{eq:hat} and \eqref{eq:inf} imply that 
\begin{equation}\label{eq:index} \frac{\Big[D(\widehat{\Z}) : \nu(\bfG(\wh\Z))\Big]}{\Big[  
\mu_E:\mu_E \cap \nu(\bfG(\wh\Z))\Big]}=\begin{cases}
2 &  \text{if $d_E \equiv 4\!\!\! \pmod 8$, $\Lambda_2 \simeq 
H^{n/2}$ and  $E \neq \Q(\sqrt{-1})$}; \\
1 & \text{if $d_E \equiv 4\!\!\! \pmod 8$, $\Lambda_2 \simeq 
H^{n/2}$ and  $E = \Q(\sqrt{-1})$}; \\
 1 & {\text{otherwise}}. 
\end{cases}
\end{equation}

Finally, 
the formulas of Shyr  \cite[Formula (16)]{Shyr} and 
 Guo-Sheu-Yu \cite[Section ~5.1]{GHY} show  that
\begin{align}\label{class}
h(D)  =\begin{cases}h(T^E)
\\ 
h(T^{E,1})
\end{cases} \!\!\!\!\!
=\  \begin{cases}
 h(E) & {\text{if $n$ is odd}};
\\ 
 2^{1-w} \cdot h(E) & {\text{if $n$ is even}}.
\end{cases}
\end{align}

Equalities (\ref{pi_0G_1}),  (\ref{eq:index}), and (\ref{class}) imply  assertion (1). 

(2) We apply formula \eqref{pi_0G_1} to $\K=\K(N)$ for $N\geq 3$. 
By definition we have $\nu(\K(N)) \subset 1+N\widehat{\Z}$. 
Furthermore $\mu_E \cap (1+N\widehat{\Z})=1$ since $N\geq 3$ (cf. 
 \cite[Lemma, p.~207]{mumford:av}). 
 It follows that 
 \begin{equation}\label{eq:mu_E}
  [\mu_E: \mu_E \cap \nu(\K(N))]=\lvert \mu_E \rvert. 
 \end{equation}
 Equalities \eqref{pi_0G_1},  \eqref{eq:mu_E}, and \eqref{class} imply assertion (2).
\end{proof}

\section{The  basic locus of unitary Shimura varieties}\label{sec:bas}

\subsection{Moduli spaces with good reduction at $p$}
\subsubsection{}\label{moduli}
 We recall the moduli space of abelian varieties with good reduction at $p$ following Kottwitz \cite{Kottwitz} and Lan \cite{Lan}. 
  Let 
    $\mathscr D=(E, \bar{\cdot},O_E, V,\< \, , \, \>, \Lambda, h_0)$ be an integral  unitary  PEL  datum as in Definition \ref{P.6}.  Let $\varphi$ be the  Hermitian form on $V$ associated to the pairing $\< \, , \, \>$, and  $(r,  s)$ its signature.  
    We assume that $\Lambda$ is self-dual with respect to $\< \, , \, \>$ (or equivalently, $(\Lambda, \varphi)$ is unimodular; see Lemma \ref{selfdual}).  
  We fix a prime $p>2$ which is unramified in $E$. 
Let  $\Qbar$ be   the algebraic closure of $\Q$ in $\C$ and fix an embedding $\iota_p:\Qbar \embed \Qpbar$. Denote by  $\Sigma_{E}=\Hom(E,\C)=\Hom(E,\Qbar)$ the set of embeddings of $E$ into $\C$. 
Let $d_E$ be  the discriminant of $E/\Q$. 
We write $\Sigma_E=\{\tau,\bar \tau\}$, where we choose $\tau:E\to \C$ such that  ${\rm Im}(\tau(\sqrt{d_E}))<0$.

 Let $\bfG\coloneqq \GU(\Lambda, \< \, , \, \>)\simeq \GU(\Lambda, \varphi)$. 
 Then $\bfG(\Z_p)$ is a hyperspecial subgroup of $\bfG(\Q_p)$. 
 For an integer  $N\ge 3$ with
 $p \nmid N$, we define an open compact subgroup  $\K^p(N)$ of $\bfG(\A_f^p)$  by  
 \[{\rm K}^p(N) \coloneqq \ker \Big(\mathbf G(\widehat{\Z}^p) \to 
 \mathbf G(\widehat{\Z}^p/N\widehat{\Z}^p)\Big).\]
 Further we put ${\rm K} \coloneqq \bfG(\Z_p) \cdot {\rm K}^p(N) \subset \bfG(\A_f)$. 
 
 Let $O_{\bfE}$ be the ring of integers of the reflex field  $\bfE$ of $\mathscr D$ and set $O_{\bfE, (p)}\coloneqq O_{\bfE} \otimes \Z_{(p)}$. 
        Let  $\bfM_{{\rm K}}= \bfM_{{\rm K}}(\mathscr D)$ be the  contravariant functor from the category of locally Noetherian schemes 
 over $O_{\bfE, (p)}$ to the category of sets which takes a connected scheme $S$ over  $O_{\bfE, (p)}$ to the set of  isomorphism classes of tuples $(A,\iota,\lambda,\bar \eta)$ where 
\begin{itemize}
\item $A$ is an abelian scheme  over $S$; 
 \item $\iota$ is a homomorphism $O_E\to \End(A)$ which 
 satisfies the Kottwitz determinant condition of signature $(r,s)$, that is, we have an equality of polynomials
 \begin{equation}\label{eq:det}
\det(T- \iota(b); \Lie(A))=(T-\tau(b))^r(T- \bar{\tau}(b))^s \in \calO_S[T] 
 \end{equation} 
for all $b \in
O_E$, where $\{\tau,\bar \tau\}$ are embeddings of $E$ into $\C$; 
 \item 
  $\lambda : (A, \iota)  \xrightarrow{\sim} (A^t, \iota^t)$ is an $O_E$-linear \emph{principal} polarization,  that is,   $({A}^t, \iota^t)$ is the  dual abelian scheme of $A$ with a  homomorphism ${\iota}^t :  O_E \to \End {A^t}$ given by $b\mapsto (\iota(\bar{b}))^t$, and $\lambda : A \xrightarrow{\sim} A^t$ is a principal polarization  preserving $\iota$ and $\iota^t$;  
 \item 
 $\bar{\eta}$ is a $\pi_1(S,\bar s)$-invariant ${{\rm K}}^p(N)$-orbit of $O_E \otimes \wh \Z^{p}$-linear isomorphisms $\eta: \Lambda\otimes \wh \Z^p\isoto \wh T^p(A_{\bar s})$  which preserve the pairings 
 \begin{align*}
 \< \, , \, \>:   \Lambda\otimes \wh \Z^p \times \Lambda\otimes \wh \Z^p &\to \wh \Z^p, 
 \\
\lambda:  \wh T^p(A_{\bar{s}}) \times \wh T^p(A_{\bar{s}}) & \to \wh \Z^p(1)
   \end{align*}
    up to a scalar in 
   $(\wh \Z^p)^{\times}$. Here, $\widehat{T}^p(A_{\bar s})$ is  the prime-to-$p$ Tate module and 
  $\bar{s}$ is a geometric point of $S$. 
   \end{itemize}
   Two tuples $(A, \iota, \lambda, \bar \eta)$  and $(A',\iota',\lambda',\bar \eta')$ are said to be isomorphic if there exists an $O_E$-linear isomorphism 
  of  abelian schemes 
  $f : A\xrightarrow{\sim} A'$  such that  $\lambda=f^{t} \circ \lambda' \circ f$  and  $\bar{\eta}'=\overline{f \circ \eta}$.  
  See \cite[1.4.1]{Lan} for more details. 
   \begin{thm}[{\cite{Kottwitz},  \cite[Ch.2]{Lan}}]
The contravariant functor $\mathbf M_{{\rm K}}$ is represented by a smooth  quasi-projective scheme (denoted again by) $\mathbf M_{{\rm K}}$ over $O_{\bfE, (p)}$.
\end{thm}

\begin{remark}\label{rem:MK}
(1) In \cite{Kottwitz} Kottwitz defined a moduli problem using prime-to-$p$ isogeny classes of abelian schemes with a $\Z_{(p)}^{\times}$-polarization. 
By \cite[Proposition  1.4.3.4]{Lan},
this moduli problem is isomorphic to $\bfM_{\K}$ defined above, 
under the assumption that 
$\Lambda$ is self-dual.

(2) By \cite[Section  8]{Kottwitz} the complex algebraic variety $\bfM_{\rm K}\otimes \C $ is a finite disjoint union of complex Shimura varieties indexed by the finite set $\ker^1(\Q,{\bf G}) := \ker\Big( H^1(\Q, {\bf G})\to \prod_{\ell \leq \infty} H^1(\Q_{\ell},{\bf G})\Big) $. 
By \cite[Lemma 8.8]{Yu:arithmetic}, the group ${\bf G}=\GU(V,\< \, , \, \>)$ (for an  imaginary quadratic field $E$) satisfies the Hasse principle and hence that  $\bfM_{\rm K} \otimes \C \simeq \Sh_{{\rm K}}({\bf G}, X)_{\C}$. The number of connected components of $\bfM_{\rm K}\otimes \C $ then is computed by Theorem~\ref{connected} (2).
\end{remark}

 Let $\mathfrak{p}$  be a prime ideal of $O_{\bfE}$ lying on $p$ that corresponds to the embedding $\iota_p:\Qbar\embed \Qpbar$. 
 Let $O_{\bfE, \mathfrak{p}}$ be the localization    of $O_{\bfE}$ at $\mathfrak{p}$,  and 
let $k$ be an  algebraic closure  of the residue field of  $O_{\bfE, \mathfrak{p}}$. 
  Let $\mathcal M_{{\rm K}}\coloneqq \mathbf M_{\rm K} \otimes_{O_{\bfE, \mathfrak{p}}} k$ denote the base change of $\mathbf M_{{\rm K}}$ to  $k$.

 \subsection{Newton and Ekedahl-Oort strata}\label{stratification} 
 \subsubsection{}
 Let $\F$ be a perfect field of characteristic $p$. 
Let $W(\F)$ be the ring of Witt-vectors over $\F$ with the Frobenius morphism  $\sigma : W(\F) \to W(\F)$. 
  Let $W(\F)[{\sf F, V}]$ be the  quotient ring   of the associative free $W(\F)$-algebra generated by the indeterminates ${\sf F}$ and ${\sf V}$ with respect to the relations 
 \begin{align}\label{FV}
 {\sf {FV}=VF}=p, \quad {\sf F}a=a^{\sigma}{\sf F}, \quad {\sf V}a^{\sigma}=a{\sf V} \quad {\text{for all}} \quad  a \in W(\F).
 \end{align}

 A \emph{Dieudonn\'e module} $M$ over $W(\F)$ is a left $W(\F)[{\sfF,\sfV}]$-module which is
 finitely generated as a $W(\F)$-module. 
 A \emph{polarization} on $M$ is an alternating form 
 $\<\,,\,\> :  M \times  M \to W(\F)$
  satisfying  
  \begin{equation}\label{adj}
   \<{\sf F}x,y\>=\<x, {\sf V}y\>^{\sigma} \quad \text{for all} \quad x, y \in M. 
  \end{equation}
 A polarization is called a \emph{principal polarization} 
 if it is a perfect pairing. 
 By \dieu theory, there is an equivalence of categories between the category of $p$-divisible groups over $\F$ and the category of \dieu modules that are free  $W(\F)$-modules. 

Let $\mathscr D=(E, \bar{\cdot},O_E, V,\< \, , \, \>, \Lambda, h_0)$ and  $p>2$ be as in Section ~\ref{moduli}. 
We set $E_p \coloneqq E \otimes_{\Q} \Q_p$ and $O_{E_p} \coloneqq O_E\otimes \Z_{p}$. 
A $p$-\emph{divisible group with $\mathscr D$-structure} over $k$  is a triple $(X,  \iota, \lambda)$ consisting of a $p$-divisible group $X$ over $k$ of height $2n$, an action $\iota: O_{E_p} \to \End(X)$, and an $O_{E_p}$-linear isomorphism  $\lambda : X \to X^{t}$ 
such that $\lambda^{t}=-\lambda$. 
Furthermore, the induced  $O_{E_p}$-action on $\Lie(X)$ is assumed to satisfy the determinant condition governed by $\mathscr D$. 

For a $p$-divisible group with $\mathscr D$-structure $(X,  \iota, \lambda)$, 
let $M$ be the covariant \dieu module of the $p$-divisible group $X$. 
 Then $\iota$ induces  an action of  
 $O_{E_p}$ on $M$  and $\lambda$ induces a 
    principal polarization $\<\,,\,\> : M \times  M \to W(k)$, 
    satisfying that 
   \begin{itemize}
   \item[$\bullet$] the action of $O_{E_p}$  commutes with the operators $\sf F$ and $\sf V$, and 
   \item[$\bullet$]
    $\<ax, y\>=\<x, a^*y\>$ for all $x, y \in M, a\in O_{E_p}$. 
   \end{itemize}
   \subsubsection{}\label{div}
   In the rest of this section, we write $G\coloneqq \bfG_{\Z_p}$. 
   This is a connected reductive group scheme over $\Z_p$, and  its generic fiber $G_{\Q_p}$ is unramified.  
   Let $L$ denote the fraction field of $W(k)$. 
   Let $C(G)$ denote the set of $G(W(k))$-$\sigma$-conjugacy classes   $\llbracket b \rrbracket \coloneqq \{ g^{-1}b \sigma (g) \mid g \in G(W(k))\}$ of elements $b \in G(L)$. 
   By \cite[Theorem 3.16]{RZ}, 
there exists an isomorphism $M \xrightarrow{\sim} \Lambda \otimes W(k)$ preserving $O_{E_p}$-actions and the  pairings $\< \, , \,\>$, under which $\sfF$ corresponds to $b( {\rm id} \otimes \sigma)$ for some $b \in {G}(L)$. 
The class   $\llbracket b \rrbracket \in C(G)$ does not depend on the choice of a fixed  isomorphism, and we therefore obtain an injective map 
\begin{equation}\label{map}
\left\{ 
\begin{array}{c}
\text{isomorphism classes of $p$-divisible groups} 
\\
\text{with $\mathscr D$-structure over $k$} 
\end{array}
\right\} 
\hookrightarrow 
C(G).
\end{equation}
Let $B(G)$ be the set of $G(L)$-$\sigma$-conjugacy classes 
$[b]  \coloneqq  \{g^{-1} 
b \sigma(g) \mid g \in G(L)\}$ of elements  $b \in G(L)$. 
The above map with the natural projection $C(G) \to B(G)$ then induces an injective map 
\begin{equation}\label{isog}
\left\{ 
\begin{array}{c}
\text{isogeny  classes of $p$-divisible groups} 
\\
\text{with $\mathscr D$-structure over $k$} 
\end{array}
\right\} 
\hookrightarrow 
B(G). 
\end{equation}

We  recall a   description   of $B(G)$.   
 We fix a maximal torus $T$ of $G$ and a Borel subgroup $B$ containing $T$ both defined over $\Qp$. 
 Let $(X^*(T), \Psi, X_*(T), \Psi^{\vee}, \Delta)$ be the corresponding based root datum. 
 Let $W \coloneqq N_{G}(T)/T$ be the associated Weyl group.
By work of 
Kottwitz \cite{Kottwitz85}, a  class $[b] \in B(G)$ is uniquely determined by two invariants: the Kottwitz point 
$\kappa_G(b) \in \pi({G})_{\sigma}=\pi_1(G)/(1-\sigma)(\pi_1(G))$  and  the Newton point $\nu_G(b) \in (X_*(T)_{\Q}^{+})^{\sigma}$. 
Here $\pi_1(G)$ denotes the Borovoi's fundamental group, and $X_*(T)_{\Q}^{+}$ denotes the set of dominant elements in $X_*(T)_{\Q}\coloneqq X_*(T) \otimes \Q$. 
For $\lambda, \lambda' \in X_*(T)_{\Q}$, 
we write $\lambda \leq \lambda'$ if the difference $\lambda'-\lambda$ is 
a non-negative rational linear combination of  positive coroots. 
The set $B(G)$ naturally forms a partially ordered set  with $[b] \leq [b']$ if $\kappa_G(b) = \kappa_G(b')$ and $\nu_G(b)\leq \nu_G(b')$. 
We call a class $[b] \in B(G)$ \emph{basic} if $\nu_G(b) \in X_*(Z_G)_{\Q}$, where  $Z_G\subset T$ is the center of $G$. 
From \cite[Section ~5]{Kottwitz85} it follows that 
 $[b] \in B(G)$ is basic if and only if it is minimal  with
respect to the above partial order. 

We now describe the image of the map   \eqref{isog}.
Let  $[\mu]$ be the conjugacy class of  the cocharacter $\mu_h$  defined by the PEL  datum $\mathscr D$ as in  Section  \ref{ss:Sh}. 
Under the fixed embedding  $\iota_p : \ol{\Q} \hookrightarrow \ol{\Q}_p$,  we can  regard  $[\mu]$ as a $W$-orbit in $X_*(T)$. 
The dominant  representative of $[\mu]$ in $X_*(T)$ is denoted by $\mu$. We set 
$\mu^{\diamond} \coloneqq  \frac{1}{m}  \sum_{i=0}^{m-1} \sigma^i(\mu)  \in X_*(T)_{\Q}$ where $m$ is the degree of a splitting field of $G$  over $\Q_p$. Further  
$\mu^{\natural}$ 
denotes the image of $\mu$ under the projection 
$X_*(T) \to  \pi_1(G)_{\sigma}$. 
As in \cite[Section  6]{Kottwitz97} and  \cite[Section  4]{Rapoport}, 
we define a set $B(G, \mu)$  by 
\begin{equation}
    \label{bgmu}
 B(G, \mu)=\{ [b] \in B(G) \mid  \nu_G(b) \leq \mu^{\diamond},  \kappa_G(b)=\mu^{\natural}\}. 
\end{equation}
Then by  \cite{RR, KR, Lucarelli, Gashi}, 
 the  image of the map  \eqref{isog} is equal to the subset 
  $B(G, \mu)$. 
    There is a unique basic class in $B(G, \mu)$ by \cite[Section ~5]{Kottwitz85}. 

\subsubsection{}\label{b}
 Let $(A, \iota,  {\lambda}, \bar{\eta})$  be a $k$-point of $\mathcal M_{\rm K}$. 
 Then  ${\iota}$  and $\lambda$ induce a $\mathscr D$-structure on 
 the $p$-divisible group $A[p^{\infty}]$ of $A$, and  
 we can attach to it a class $\llbracket b \rrbracket \in C(G)$. We thus obtain a map 
 \[ \calM_{\K}(k) \to C(G).\] 
 Each fiber of this map is called a \emph{central leaf} on $\M_{\K}$.

 Further, there exists a map
 \[{\rm Nt} : \mathcal M_{\rm K} \to  B(G, \mu)\]
 such that it takes  a $k$-point $(A, \iota,  {\lambda}, \bar{\eta})$ to the class $[b]$ attached to  $(A[p^{\infty}],  \iota, \lambda)$. 
  Each fiber 
    is a locally closed subset of $\mathcal M_{\rm K}$ \cite{RR}.  
   By definition, the \emph{Newton stratum  attached to} $[b] \in B(G, \mu)$ is 
    the fiber 
  ${\rm Nt}^{-1}([b])$ endowed with the 
reduced scheme structure. 
 It is non-empty for every $[b] \in B(G, \mu)$ \cite[Theorem 1.6 (1)]{VW}. 
 The Newton stratum  attached to the unique basic class  $[b] \in B(G, \mu)$  is  called the \emph{basic locus} and denoted by $\mathcal M_{\rm K}^{\rm bas}$. 
 This is closed  in  $\M_{\K}$. 

 Let $S \subset W$ denote the set of simple reflections  corresponding to $B$. 
 To the cocharacter $\mu$ associated to  $\mathscr D$, one can attach a  non-empty subset $J \subset S$ \cite[A.5]{VW}. Let $W_J$ denote the subgroup of $W$ generated by $J$, and let 
    ${}^J W$ denote the set of  representatives of minimal length of $W_J\backslash W$: 
\[ {}^J W \coloneqq \{ w \in W \mid \ell(sw) > \ell(w)  \ \text{for all}  \ s \in J\}.\] 
 For a $k$-point $(A, \iota,  {\lambda}, \bar{\eta})$ of $\M_{\K}$, let $(A, \iota, \lambda)[p]$ be the $p$-torsion $A[p]$ of $A$ endowed with the $O_E$-action induced by $\iota$ and the isomorphism $A[p] \xrightarrow{\sim} A[p]^t$ induced by $\lambda$. 
 By 
 \cite[Theorem 6.7]{Moonen}, 
 we can associate to $(A, \iota, \lambda)[p]$ an element $\zeta(x) \in {}^J W$, and 
 the element $\zeta(x)$ determines the isomorphism class of $(A, \iota, \lambda)[p]$. 
  The \emph{Ekedahl-Oort (EO) stratum attached to}
   $w \in {}^J W$ is the  locally closed  reduced subscheme  $\mathcal M_{\rm K}^{w}$ of $\M_{\K}$ whose $k$-points consist of $x\in \M_{\K}(k)$ with  $\zeta(x)=w$. 
   The EO stratum is non-empty for all $w \in {}^J W$ \cite[Theorem 1.2 (1)]{VW}. 
 The EO stratum $\mathcal M_{\rm K}^e$ attached to the identity element $e \in {}^J W$ is the unique  EO stratum  of dimension zero. 
 
 \subsubsection{}
Let $(X, \iota, \lambda)$ be a $p$-divisible group with $\mathscr D$-structure over $k$  and let $b \in G(L)$ such that the  isomorphism class of $(X, \iota, \lambda)$ 
corresponds to  $\llbracket b \rrbracket \in C(G)$ via \eqref{map}. 
Further we put $H \coloneqq \ker( G(W(k)) \to G(k))$. 
By \cite[Remark 8.1]{VW}, the following conditions are equivalent.
\begin{itemize}
    \item For all $b' \in H  b H$ there exists an element $g \in G(W(k))$ such that $gb'\sigma(g)=b$.
    \item Any $p$-divisible group with $\mathscr D$-structure $(X', \iota', \lambda)$ with $(X', \iota', \lambda')[p] \simeq (X, \iota, \lambda)[p]$ is isomorphic to 
 $(X, \iota, \lambda)$.
\end{itemize}
An element $b \in G(L)$ and also $(X, \iota, \lambda)$    are  called \emph{minimal}  if they satisfy these equivalent conditions. 
 \begin{prop}[{\cite[Proposition 9.17]{VW}}]\label{CL} 
 {\rm (1)} 
 For each $k$-point $(A, \iota,  {\lambda}, \bar{\eta})$ on the $0$-dimensional EO stratum $\M_{\K}^e$, the 
associated $p$-divisible group with $\mathscr D$-structure $(A[p^{\infty}],  \iota,  \lambda)$  is minimal. 
 Furthermore the set $\mathcal M_{\rm K}^e(k)$ is a central leaf. 

 {\rm (2)} 
 There is an inclusion $\mathcal M_{\rm K}^e \subset \mathcal M_{\rm K}^{\rm bas}$. 
 \end{prop}
 
\subsubsection{}
Let $\Sigma_{O_{E_p}}=\Hom_{\Zp}(O_{E_p},W(k))$ be the set of embeddings of $O_{E_p}=O_E \otimes {\Z_p}$ into $W(k)$. We write $\Sigma_{O_{E_p}}=\{\sigma_1,\sigma_2\}$ in the way that they correspond to $\tau$ and $\bar \tau$, respectively, under the identification 
\[ \Sigma_E=\Hom(E,\Qbar)=\Hom_{\Qp}(E_p, \Qpbar)=\Hom_{\Zp}(O_{E_p},W(k))= \Sigma_{O_{E_p}}\] 
via the embedding $\iota_p$. It is convenient to write $\sigma_i$ with $i\in\Z/2\Z$, then one has $\sigma\circ \sigma_i=\sigma_{i+1}$ for $i\in \Z/2\Z$.

Let  $(A, \iota,  {\lambda}, \bar{\eta})$ be a $k$-point of $\M_{\K}$, and  $(A[p^{\infty}], \iota, \lambda)$  be the associated  $p$-divisible group with $\mathscr D$-structure. 
Let $M$ be the covariant  \dieu module of $A[p^{\infty}]$. The induced  $O_{E_p}$-action on $M$ gives rise to a decomposition 
$M=M_1 \oplus M_2=\oplus_{i\in \Z/2\Z} M_i$, where $M_i$ is the $\sigma_i$-component of $M$. 
For each $i\in \Z/2\Z$, we have that $\<M_i,M_i\>=0$, and 
\begin{align}
\begin{aligned}\label{eq:i}
     \sfF(M_i) &   \subset 
     M_{i+1}, & 
  \sfV(M_{i}) & \subset 
  M_{i+1}, & 
  \dim_{k}(M_1/\sfV(M_{2}))& =r &  & \text{if $p$ is inert};  
    \\ 
    \sfF(M_i)  & \subset 
     M_i, & 
  \sfV (M_{i}) & 
  \subset M_i,  
  &  
  \dim_{k}(M_1/\sfV(M_{1})) & =r & 
    &\text{if $p$ is split.}
\end{aligned}
\end{align}
 Here, the right equalities follows from  the determinant condition  \eqref{eq:det}. 
\subsubsection{}\label{model}
Now we fix a $k$-point  $(A, \iota,  {\lambda}, \bar{\eta})$ of  
the $0$-dimensional EO stratum $\M_{\K}^e$. 
We assume first  that $p$ is inert in $E$. 
In this case, the abelian variety  $A$ over $k$ is  superspecial  \cite[Section~3.5.1]{Wooding}. 
As in \cite[Proposition 6]{Ghitza}, $A$   has a canonical model $\tilde{A}$ over  $\F_{p^2}$, in which the geometric Frobenius is $[-p]$.  
 Moreover, the covariant \dieu module $\tilde{M}$ of its $p$-divisible group $\tilde{A}[p^{\infty}]$  is isomorphic as  a \dieu module over  $W(\F_{p^2})$ to the direct sum of  $n$-copies of the \dieu module $\D$ which is defined by 
 \begin{equation*}
 (\D, \sfF, \sfV)=
  \left( W(\F_{p^2})^{\oplus 2}, \quad \begin{pmatrix}
    & 1 
    \\ 
    -p & 
  \end{pmatrix}\sigma', \quad 
  \begin{pmatrix}
    & -1 
    \\ 
    p & 
  \end{pmatrix}\sigma'^{-1}
  \right). 
 \end{equation*} 
 Here $\sigma'$ is the Frobenius morphism of $W(\F_{p^2})$. 
In particular one has $\sfF=-\sfV$ on $\tilde{M}$. 

 Since the  correspondence $A \mapsto \tilde{A}$ is functorial, $\tilde{A}$ is equipped with additional structure corresponding to $\iota$ and $\lambda$. 
 Thus $\tilde{M}$ has  a principal polarization 
 $\<\,,\,\> :  \tilde{M} \times  \tilde{M} \to W(\F_{p^2})$ and an $O_{E_p}$-action, inducing 
 a decomposition  $\tilde{M}=\tilde{M}_1 \oplus \tilde{M}_2$. 
\begin{prop}\label{Minert}
 Assume that $p$ is inert in $E$. 
Then there is a basis $e_1, \ldots, e_n, f_1, \ldots, f_n$ of $\tilde{M}$ over $W(\F_{p^2})$ such that 
\begin{itemize}
\item[(i)] it is a standard symplectic basis with respect to $\< \, , \, \>$, that is,  it satisfies that 
\[\left\langle e_{i_1}, e_{i_2} \right\rangle=0=\left\langle f_{j_1}, f_{j_2} \right\rangle, \  \left\langle e_i, f_j \right\rangle=\delta_{ij}=-\left\langle f_j, e_i \right\rangle;\]
\item[(ii)] the elements 
$e_1, \ldots, e_n$ span $\tilde{M}_1$ and $f_1 \ldots, f_n$ span $\tilde{M}_2$; 
\item[(iii)]
 the $\sigma'$-linear (resp.~$\sigma'^{-1}$-linear)  operator $\sfF$ (resp.~$\sfV$) is given by 
\begin{equation}\label{eq:Finert}
 \mathsf{F}=\begin{pmatrix}
 &  &  -pI_r  & 
 \\   &   &   & I_s
\\ 
 I_r &   &  & 
\\
 & -pI_s &   & 
\end{pmatrix} \sigma', \quad 
\mathsf{V}=\begin{pmatrix}
 &  &  pI_r  & 
 \\   &   &   & -I_s
\\ 
 -I_r &   &  & 
\\
 & pI_s &   & 
\end{pmatrix} \sigma'^{-1}
.
\end{equation}
\end{itemize}
\end{prop}
\begin{proof}
By \cite[Corollary 9]{Ghitza}, for any principal polarization $\< \, , \, \>$ on the \dieu module $\tilde{M} \simeq \D^{\oplus n}$, there is an element $x \in M$ such that $\<x, \sfF(x)\>$ is a unit of $W(\F_{p^2})$. 
We may assume that $x$ belongs to $\tilde{M}_1$ or $\tilde{M}_2$ since $\<\tilde{M}_i, \sfF (\tilde{M}_{i+1})\>=0$. 
We have that 
$\<x, \sfF(x)\>
=\<x, -\sfV(x)\>
=-\<\sfF(x), x\>^{\sigma'}=\<x, \sfF(x)\>^{\sigma'}$ and hence $\<x, \sfF(x)\>$ belongs to $W(\F_p)^{\times}$. 
Multiplying $x$ by an element of  $W(\F_{p^2})^{\times}$, we may further assume that $\<x, \sfF(x)\>=1$ or $-1$ according as  $x$ belongs to  $\tilde{M}_1$  or $\tilde{M}_2$, 
since the norm map $W(\F_{p^2})^{\times} \to W(\F_p)^{\times}, a \mapsto a a^{\sigma'}$ is surjective. 
Let $\tilde{N} \coloneqq W(\F_{p^2})x \oplus W(\F_{p^2})\sfF(x)$. 
Then $\tilde{M}=\tilde{N} \oplus \tilde{N}^{\perp}$ 
where 
$\tilde{N}^{\perp}$ denotes the dual of $\tilde{N}$ with respect to $\< \, , \, \>$. 
Furthermore one has  $\sfF=-\sfV$ on  $\tilde{N}^{\perp}$, and it follows that $\tilde{N}^{\perp} \simeq \D^{\oplus n-1}$. 
Moreover we have 
$\dim_k\big((\tilde{N} \cap \tilde{M}_{i+1})/\sfV(\tilde{N} \cap \tilde{M}_{i})\big)=1$ 
if $x \in \tilde{M}_i$. 

 By induction, 
 there exist elements $e_1, \ldots, e_r$ of  $\tilde{M}_1$ and  $f_{n-r}, \ldots, f_{n}$ of $\tilde{M}_2$ such that 
 $\tilde{M}_1$ is spanned by $e_1, \ldots, e_r, \sfF(f_{n-r}), \ldots, \sfF(f_n)$, $\tilde{M}_2$ is spanned by $\sfF(e_1), \ldots, \sfF(e_r),  f_{n-r}, \ldots, f_n$, and further  
 $\<e_i, \sfF(e_i)\>=\<\sfF(f_j), f_j\>=1$. 
\end{proof}
\subsubsection{}
Next we consider the case  $p$ is split in $E$. 
We start by recalling the definition of minimal $p$-divisible groups (without additional structure) 
introduced by Oort   \cite{Oort}.  
Let $a, b$ be  coprime non-negative integers. 
Let  $\mathbb M$ be the \dieu module with basis $e_1, e_2, \ldots, e_{a+b}$ over $W(k)$ on which 
the actions of $\sfF$ and $\sfV$ are  given by $\sfF(e_i) = e_{i+b}$ and $\sfV(e_i)=e_{i+a}$. 
Here we use the notation $e_{i+a+b}=pe_{i}$. 
We write $H_{a,b}$ for the 
the $p$-divisible group over $k$  corresponding to $\mathbb M$. 
Then $H_{a,b}$ has height $a+b$ and is isoclinic of slope ${b}/(a+b)$. 
We call a  $p$-divisible group $X$ over $k$  \emph{Oort-minimal} if $X$ is isomorphic to a product  $\prod_j H_{a_j, b_j}$ for some $a_j, b_j$. 
By \cite{Oort, Oort2}, $X$ is Oort-minimal if and only if it satisfies the property  that  
any $p$-divisible group $X'$ with $X'[p] \simeq X[p]$ is isomorphic to $X$ . 

Let $(A, \iota, \lambda, \bar{\eta}) \in \M_{\K}^e(k)$ and let 
$M$ be the \dieu module associated with   $(A[p^{\infty}], \iota, \lambda)$.  We set $m =\gcd(r,s)$, $r'=r/m$, $s'=s/m$, and $n'=n/m$. 

\begin{prop}\label{Msplit}
Assume that $p$ is split in $E$. 
Then there is a decomposition $M=(M')^{\oplus m}$ of $M$ into mutually orthogonal  copies of a  \dieu module $M'$ over $W(k)$  with  an $O_{E_p}$-action,  which has a  basis $e_1, \ldots, e_{n'}, f_1, \ldots, f_{n'}$ such that  
\begin{itemize}
\item[(i)]  it is a standard symplectic basis with respect to $\< \, , \, \>$, that is, 
\[\left\langle e_{i_1}, e_{i_2} \right\rangle=0=\left\langle f_{j_1}, f_{j_2} \right\rangle, \  \left\langle e_i, f_j \right\rangle=\delta_{ij}=-\left\langle f_j, e_i \right\rangle;\]
\item[(ii)] the elements 
$e_1, \ldots, e_{n'}$ span $M_1'$ and $f_1 \ldots, f_{n'}$ span $M_2'$ where $M'=M'_1 \oplus M'_2$ is  the decomposition induced by the action of $O_{E_p}$; 
\item[(iii)]
 the $\sigma$-linear (resp.~$\sigma^{-1}$-linear)  operator $\sfF$ (resp.~$\sfV$) is given by  
\begin{equation}\label{eq:Fsplit}
  \mathsf{F}=\left(\begin{matrix}
 & pI_{s'} &    & 
 \\  I_{r'} &   &   & 
\\ 
 &   &  & I_{s'}
\\
 &  &  pI_{r'} & 
\end{matrix}\right) \sigma, 
\quad 
\mathsf{V}=\left(\begin{matrix}
 & pI_{r'} &    & 
 \\  I_{s'} &   &   & 
\\ 
 &   &  & I_{r'}
\\
 &  &  pI_{s'} & 
\end{matrix}\right) \sigma^{-1}.  
\end{equation}
\end{itemize}
\end{prop}
\begin{proof}
Let $\Lambda_{W(k)}=\Lambda_1 \oplus  \Lambda_2$ be the decomposition induced by an action of $O_E \otimes W(k)=W(k) \oplus W(k)$.  Then  we have 
$G(L) \simeq \Gm(L) \times \GL(\Lambda_{1}[\frac{1}{p}])$. 
Furthermore there exists an isomorphism $M \simeq \Lambda_{W(k)}$ identifying  $M_1$ with $\Lambda_1$, under which $\sfF$ corresponds to an element $b$ of     $\Gm(L) \times \GL(\Lambda_{1}[\frac{1}{p}])$. 
Note that its  $G(L)$-$\sigma$-conjugacy  class $[b] \in B(G, \mu)$ is basic by  Proposition~\ref{CL}~(2).
On the other hand, one has  $\sfF(M_1), \sfV(M_1) \subset M_1$ by \eqref{eq:i},  and 
hence $M_1$ can be regarded as a \dieu module   without additional structure. 
We write $b_1 \in   \GL(\Lambda_{1}[\frac{1}{p}])$ for the element corresponding to the operator $\sfF\vert_{M_1}$ under the identification $M_1 \simeq \Lambda_2$. 
Then $b_1={\rm pr}_2(b)$ by the construction. 

Let $X$ be the $p$-divisible group  over $k$ corresponding  to the \dieu
 module $M_1$. 
By Proposition \ref{CL} (1), the element $b \in G(L)$ is minimal, and 
therefore $b_1$ is minimal as an element of $\GL(\Lambda_{1}[\frac{1}{p}])\simeq \GL_n(L)$. 
It follows from  \cite{Oort} (see also  \cite[Corollary 9.13]{VW}) that 
 $X$ is Oort-minimal.   
 Further, the above argument shows that   the   $\GL_n(L)$-$\sigma$-conjugacy class of $b_1$ is basic, and hence  $X$ is isoclinic. 
 In addition we have that  $\dim_{k}M_1/\sfV(M_1)=r$ and $\dim_k M_1/\sfF(M_1)=s$ 
 by \eqref{eq:i}. 
 These conditions imply that $X \simeq H^m_{r', s'}$, i.e., 
there is a decomposition 
 $M_1=(M'_1)^{\oplus m}$ of $M_1$ into a \dieu module $M'_1$ with a basis $e_i, \ldots, e_{n'}$ such that  $\sfF(e_i)=e_{i+s'}$  and $\sfV(e_i)=pe_{i+r'}$. 
 
There exist  elements $f_1, \ldots, f_{n'}$ of $M_2$  satisfying condition (i) since $\< \, , \, \>$ is a  perfect pairing and $M_i^{\perp}=M_i$. 
 Let $M_2'$ be the submodule of $M_2$ spanned by $f_1, \ldots, f_{n'}$ and let  $M'=M_1' \oplus M_2'$, so that condition (ii) is  satisfied.  
 Using relation \eqref{adj} we determine the actions of $\sfF$ and $\sfV$ on $M_2'$ and obtain \eqref{eq:Fsplit}. 
\end{proof}
By  \eqref{eq:Fsplit},  the abelian  variety $A$ over $k$ is superspecial (i.e.,  $\sfF M= \sfV M$) if and only if $r=s$. See also  \cite[Section~3.5.1]{Wooding}. 
\subsection{Maximal parahoric subgroups of  unitary groups}\label{mprh} 
\subsubsection{}
We briefly recall some facts about  maximal parahoric subgroups of unitary (similitude) groups   and of  inner forms of $\GL_n$. 
Our reference is Tits \cite{Tits}. 

First let $E_p/\Q_p$ be the unramified quadratic field extension. 
Let $\bfN=\bfN_{E_p/\Q_p}$ denote the norm map of $E_p/\Q_p$. The quotient group $\Q_p^{\times}/\mathbf N E_p^{\times}$ is of order two,  having $\{1, p\}$ as a system of representatives.  
The  determinant $d(\calV)$ of a  Hermitian  space $(\calV, \phi)$ takes value   in  $\Q_p^{\times}/\mathbf N E_p^{\times}$,  as  in Section ~\ref{sec:hermitian}. 

Let $H$ be the  Hermitian lattice  of rank two over $O_{E_p}$ having a basis $\{e_1, e_2\}$ with Gram matrix $\begin{pmatrix}
 & 
 1 
 \\ 
 1 &
\end{pmatrix}$, 
and 
let $\bbH \coloneqq H \otimes_{\Z_p} \Q_p$. 
Then
$d(\bbH)=-1= 1 \in \Q_p^{\times}/\mathbf N E_p^{\times}$. 
Let $(1)$ and $(p)$  denote  the lattices of rank one  with Gram matrix $(1)$ and $(p)$, respectively. Further, let $\bbV_1\coloneqq (1) \otimes _{\Z_p}\Q_p$ and $\bbV_p \coloneqq (p) \otimes _{\Z_p}\Q_p$. 

Let $(\calV, \phi)$ be a  non-degenerate Hermitian  space of dimension $n$ over $E_p$. 
Then there is an isomorphism  
\begin{align}\label{phi}
\calV \simeq 
\left\lbrace
\begin{array}{lll}
  \mathbb H^{(n-1)/2}
\oplus 
\bbV_{1}
&  
\text{if $n$ is odd,} &   d(\calV) =  1; 
\\ 
  \mathbb H^{(n-1)/2}
\oplus 
 \bbV_{p}
&   \text{if $n$ is odd,} &  d(\calV) =  p;
\\ 
  \mathbb H^{n/2} 
  &  \text{if $n$ is even,}  &  d(\calV) =  1; 
 \\ 
  \mathbb H^{ (n/2)-1}\oplus  
\bbV_{1} \oplus \bbV_{p}
&  \text{if $n$ is even,} &   d(\calV) = p.
\end{array}
\right.
\end{align}
Let $\calG=\GU(\calV, \phi)$ (resp.~$\calG^1={\rm U}(\calV, \phi)$)   denote the unitary similitude group (resp.~the unitary group) of $(\calV, \phi)$ over $\Q_p$. 
When $n$ is odd,  
the groups  $\calG$ of the above two 
spaces with $d(\calV)=1$ and  $p$ are isomorphic and unramified. 
When $n$ is even, the group $\calG$  is unramified  or not according as  $d(\calV) = 1$ or $p$. 
The same statements hold true for the groups $\calG^1$. 

Let $H(1)$ be the $O_{E_p}$-lattice in $\H$ spanned by the basis $\{e_1, pe_2\}$, whose   Gram matrix  is 
  $\begin{pmatrix}
 & p
 \\ 
 p  & 
\end{pmatrix}$.  

 Let $t$ be an integer such that $0 \leq t \leq  n$ and   $d(\calV)=p^t \in \Q_p^{\times}/\bfN(E_p^{\times})$. 
 Under the identification in \eqref{phi}, 
we define an $O_{E_p}$-lattice $\calL_t$ in $\calV$ by 
\begin{align}\label{L_t}
\calL_t \coloneqq  
\left\lbrace
 \begin{array}{llll}
 H^{ (n-t-1)/{2}}
 \oplus 
H(1)
 ^{t/2} 
 \oplus 
 (1)   
   &  {\text{if $n$ is odd}}, &  d(\calV)= 1 &  ({\text{$t$ is even}}); 
  \\ 
  H^{ (n-t)/{2}} \oplus 
H(1)^{(t-1)/{2}} 
\oplus 
(p) 
 &  {\text{if $n$ is odd}}, & d(\calV)= p & ({\text{$t$ is odd}}); 
\\
H^{ (n-t)/{2}} \oplus 
 H(1)^{t/{2}} 
 &  {\text{if $n$ is even}}, & d(\calV)= 1 &  ({\text{$t$ is even}});
 \\ 
 H^{ (n-t-1)/{2}}
 \oplus 
 H(1)^{(t-1)/{2}}
 \oplus 
 (1) 
  \oplus 
(p)  &  {\text{if $n$ is even}}, & d(\calV) = p &  ({\text{$t$ is odd}}).
\end{array} 
\right.
\end{align}
Note that $\calL_t$ is unimodular if and only if $d(\calV)=1$ and $t=0$. 
Conversely, any unimodular lattice $\Lambda$ of rank $n$ over $O_{E_p}$ satisfies $d(\Lambda \otimes_{\Z_p}\Q_p)=1$ and  $\Lambda \simeq \calL_0$. 

We define subgroups $\calP_t$ of $\calG(\Q_p)$ and $\calP_t^1$ of $\calG^1(\Q_p)$ by 
\[ \calP_t \coloneqq {\rm Stab}_{\calG(\Q_p)}\calL_t, \quad 
\calP_t^1\coloneqq \{g \in \calP_t \mid {\rm sim}(g)=1\}.\]
Then $\calP_t$ (resp.~$\calP_t^1$)  is a maximal  parahoric subgroup of $\calG(\Q_p)$ (resp.~$\calG^1(\Q_p)$). 
Further, any maximal parahoric subgroup of $\calG(\Q_p)$ (resp.~$\calG^1(\Q_p)$) is conjugate to a subgroup $\calP_t$ (resp.~$\calP_t^1$) for some $t$. 
 In Table \ref{table:special}, we give 
special and hyperspecial parahoric subgroups among $(\calP_t)_t$. 
\begin{table}[hbtp]
  \caption{Special and hyperspecial parahoric subgroups among  $(\calP_t)_t$}
  \label{table:special}
  \renewcommand{\arraystretch}{1.2}
  \centering
  \begin{tabular}{|c|c|c|c|}
    \hline
  $n$   & $d(\calV)$  &  special & hyperspecial  \\ 
    \hline 
     {\text{odd}} & 
     $1$ & $\calP_0$, $\calP_{n-1}$ & $\calP_0$  \\ 
     \hline
    \text odd & 
    $p$  
 & 
 $\calP_1$, $\calP_n$ & 
 $\calP_n$ 
    \\  
    \hline
    \text even  
    & $1$ 
& $\calP_0$, $\calP_n$ 
& $\calP_0$, $\calP_n$ 
 \\
 \hline
 \text even & 
 $p$ 
  & $\calP_1$, $\calP_{n-1}$ & \text none 
   \\
    \hline
  \end{tabular}
\end{table}

Let $\underline{\calP}_t$ (resp.~$\underline{\calP}^1_t$) denote the smooth model  of $\calP_t$ (resp.~$\calP_t^1$) over $\Z_p$,   and let $\overline{\calP}_t$ (resp.~$\overline{\calP}_t^1$) denote the maximal reductive quotient of the special fiber $\underline{\calP}_t\otimes \F_p$ (resp.~$\underline{\calP}^1_t \otimes \F_p$). 
  Then there are isomorphisms  of algebraic groups over $\F_p$
  \begin{align*}
  \overline{\calP}_t \simeq  {\rm G}({\rm U}_{n-t} \times {\rm U}_{t}), \quad \overline{\calP}_t^1 \simeq  {\rm U}_{n-t}\times {\rm U}_{t}.
  \end{align*}
  Here, by ${\rm G}({\rm U}_{n-t} \times {\rm U}_{t})$ we mean the group of pairs of matrices  $(g_1,  g_2) \in  \GU_{n-t} \times \GU_{t}$ 
having the same similitude factor.
\subsubsection{}\label{ss:B}  
Next let $m$ be a divisor of $n$ and  $n'\coloneqq n/m$. 
Let $B$ be a division algebra over $\Q_p$ of degree $n'^2$  and  
let  $O_B$ be the  unique maximal order of $B$.   
Then the algebraic  group $\calG^1\coloneqq \Res_{B/\Q_p} \GL_{m, B}$ is an inner form of the unramified group   $\GL_{n, \Q_p}$.  
The group $\calG^1(\Q_p)=\GL_m(B)$ 
contains a unique conjugacy class of maximal parahoric subgroup,  represented by   $\GL_m(O_B)$. 
In particular, any maximal parahoric subgroup of $\calG^1(\Q_p)$  is special. 
The maximal reductive quotient of the special fiber over $\F_p$ of its smooth model is isomorphic to  $\Res_{\F_q/\F_p}\GL_{m, \F_q}$, where $\F_q$ denotes the finite field of order  $q \coloneqq p^{n'}$.
\subsection{The group of automorphisms of a basic isocrystal}
\subsubsection{}
For each $b \in G(L)$ we define 
 an algebraic group 
$J_b$ over $\Q_p$ with functor of points
\begin{equation}\label{def:J} 
J_b(R) 
= \{ g \in   G(L \otimes _{\Q_p}R ) \mid  g^{-1} b \sigma(g)=b\}
\end{equation} 
for any $\Q_p$-algebra $R$. 
Up to isomorphism, $J_b$ depends only on the $G(L)$-$\sigma$-conjugacy class $[b]$.
By \cite{Kottwitz85}, the class  $[b]$ is basic if and only if $J_b$ is an inner form of $G_{\Q_p}$. 

Let $(A, \iota,  {\lambda}, \bar{\eta})$ be a $k$-point of  $\M_{\K}$, and let  $(A[p^{\infty}], \iota,  {\lambda})$ be the associated $p$-divisible group  with $\mathscr D$-structure. 
Let  $(M, \sfF, \sfV)$ be the \dieu module of $(A[p^{\infty}], \iota,  {\lambda})$,  and  
let $M[1/p]=M\otimes_{W(k)} L$ be the associated isocrystal. 
Let $b$ be the element of  $G(L)$ corresponding to the operator $\sfF$ via an identification   $M\simeq \Lambda_{W(k)}$ as in Section ~\ref{div}. 
Then there is an isomorphism 
\begin{align*}
{J}_b(R) \simeq \{ (g, c)  \in \GL_{
(E \otimes_{\Q}L)
\otimes_{\Q_p}{R}}   (M[1/p] _R) \times R^{\times}  \mid  
(\sfF \otimes {\rm id}_{R}) g= g ({\sfF \otimes {\rm id}_R}), \  
\<gx, gy\>=c \<x,y\> \} 
\end{align*} 
for any   $\Q_p$-algebra $R$, where  $M[1/p] _R \coloneqq M[1/p] \otimes_{\Q_p} R$. 


Assume now that $(A, \iota,  {\lambda}, \bar{\eta})$ is lying on the $0$-dimensional EO stratum $\M_{\K}^e$. 
Then the class $[b]$ is the unique basic element of  $B(G, \mu)$ by Proposition \ref{CL} (2). 
Let  
\begin{equation}\label{I_p^e}
    {\rm I}_p^e \coloneqq {\rm Stab}_{J_b(\Q_p)} M.
    \end{equation}


\begin{prop}\label{J_Minert}
{\rm (1)}
Assume that $p$ is inert in $E$. 
Let $(\calV, \phi)$ be  a Hermitian space over $E_p$ with $d(\calV) = p^s$. 
Then 
\[J_b \simeq \GU(\calV, \phi).\]
In particular,  $J_b$ is unramified if and only if  $rs$ is even. 

  The subgroup ${\rm I}_p^e$ of $J_b(\Q_p)$ is a maximal parahoric subgroup, which is   conjugate to the subgroup  $\calP_s$  defined in Section ~\ref{mprh}. 

{\rm (2)}
 Assume that $p$ is split in $E$. 
Let $m=\gcd(r,s)$ and $n'=n/m$.  
Let $B$  be a division algebra over $\Q_p$ of degree $n'^2$  with $\inv(B)={r}/{n}$. 
Then  
\begin{align*}
J_b \simeq \Res_{B/\Q_p}\GL_{m, B} \times \mathbb G_{{\rm m}, \Q_p}. 
\end{align*}

Let $O_B$ be the unique maximal order of $B$. 
Then the subgroup ${\rm I}_p^e$  is a maximal parahoric subgroup of $J_b(\Q_p)$  
conjugate to  $\GL_m(O_B) \times \Z_p^{\times}$,  under the identification $J_b(\Q_p) \simeq \GL_m(B) \times \Q_p^{\times}$.  
\end{prop}
\begin{proof}
(1) 
As in Section \ref{model}, $M$ has a model $\tilde{M}$ over $W(\F_{p^2})$. 
By functoriality, we have 
\[ J_b(\Q_p)\simeq \{ (g, c)  \in \GL_{E  \otimes_{\Q}(W(\F_{p^2})[1/p])}  \tilde{M}[1/p]  \times \Q_p^{\times}  \mid  \sfF g= g {\sfF}, \  
\<gx, gy\>=c \<x,y\>
\}.\] 
 Let $(g, c)$ be an element of ${\rm I}_p^e \subset J_b(\Q_p)$. 
 Then $c \in \Z_p^{\times}$. 
 As in Proposition~\ref{Minert} we choose bases $e_1, \ldots, e_n $ and $f_1, \ldots, f_n$  of $\tilde{M}_1$ and $\tilde{M}_2$, respectively. 
 Since $g$ preserves  $\tilde{M}_1$ and $\tilde{M}_2$,  one can write $g=\diag(g_1, g_2)$  
where we regard $g_1, g_2$ as elements of   $\Mat_n(W(\F_{p^2}))$ via the fixed bases. 
Then we have 
\[\begin{pmatrix}
g_1^t & 
\\ 
& g_2^t
\end{pmatrix}
\begin{pmatrix}
 & I_n
 \\
 -I_n & 
\end{pmatrix}
\begin{pmatrix}
g_1 & 
\\
 & g_2
\end{pmatrix}
= c \cdot 
\begin{pmatrix}
 & I_n 
 \\ 
 -I_n & 
\end{pmatrix},\]
 and hence 
$g_2^t \cdot g_1=c \cdot I_n$. 
Further, 
the relation $\sfF g=g \sf F$ and equality 
\eqref{eq:Finert} show 
 that 
\begin{align} 
\Psi \cdot  g_1^{\sigma'}=g_2 \cdot  \Psi, \quad \Psi \coloneqq \diag(1^r, (-p)^s), 
\label{eq1}
\end{align} 
where $\sigma'$ denotes the Frobenius morphism of $W(\F_{p^2})$.  

Now we identify $O_{E_p}$ with $W(\F_{p^2})$. 
Then  $\det(\Psi)=p^s$ as elements of  $\Q_p^{\times}/\bfN_{E_p/\Q_p}(E_p^{\times})$. 
Further, 
the Hermitian lattice of rank $n$ over $O_{E_p}$ defined by $\Psi$ is isomorphic to the lattice $(\calL_s, \phi)$ as  in \eqref{L_t}.  It follows that  \[ {\rm I}_p^e \simeq  \{(g_1, c)  \in \GL_n(O_{E_p}) \times \Z_p^{\times} \mid (g_1^ {\sigma'})^t \cdot \Psi  \cdot 
g_1 = c\cdot \Psi\} \simeq \GU(\calL_s, \phi)(\Z_p).\]

A similar argument can be applied to the module  $\tilde{M}[1/p]  \otimes_{\Q_p} R$ for any  $\Q_p$-algebra $R$, and hence 
\[J_b \simeq  \GU(\calL_s\otimes_{\Z_p}\Q_p, \phi)=\GU(\calV, \phi).\]

 (2) 
 Assume first that $r$ and $s$ are relatively prime.  We choose a basis of $M=M'=M'_1 \oplus M'_2$ over $W(k)$ as in Proposition \ref{Msplit}. 
Then for each element $(g, c)$ of ${\rm I}_p^e$ we can write $g=\diag(g_1, g_2)$ with $g_1, g_2 \in \Mat_n(W(k))$. 
Similarly to the previous case, 
we have 
  $g_2^t \cdot g_1 =c \cdot I_n$.
Further, equality \eqref{eq:Fsplit} shows  that 
\begin{align}\label{eq4}
g_i \left(\begin{matrix}
  & pI_s
 \\ I_r  &  
\end{matrix}\right)=
\left(\begin{matrix}
  & pI_s
 \\ I_r  &  
\end{matrix}\right)g_i^{\sigma},  \quad i=1,2.
\end{align}
  It follows that ${\rm I}_p^e \simeq {\rm I}_p^{e,1} \times \Z_p^{\times}$ where 
  \[ 
  {\rm I}_p^{e,1} = 
  \left\lbrace g_1 \in \GL_n(W(k))  \mid  g_1 \left(\begin{matrix}
  & pI_s
 \\ I_r &  
\end{matrix}\right)
=
\left(\begin{matrix}
  & pI_s
 \\ I_r &  
\end{matrix}\right)
g_1^{\sigma}\right\rbrace.\]
The group ${\rm I}_p^{e,1}$ is therefore  isomorphic to the automorphism  group of the  \dieu module $M'_1$ without additional structure given in Proposition \ref{Msplit}. 
Recall that $M'_1$  corresponds to an Oort-minimal $p$-divisible group over $k$.  
Let $B$ be a division algebra of degree $n^2$ with $\inv B=\frac{r}{n}.$ 
From \cite[Lemmas 3.1 and 3.4]{Yu:end} it follows that $
{\rm I}_p^{e,1} \simeq  
  O_B^{\times}$ as groups and 
$J  \simeq \Res_{B/\Q_p}\mathbb G_{{\rm m}, B} \times \mathbb G_{{\rm m}, \Q_p}$ as algebraic groups over $\Q_p$.

For general $r$ and $s$, there is a decomposition $M=(M')^{\oplus m}$ and thus we see the assertion. 
\end{proof}
Let $[b] \in B(G, \mu)$ be as above and 
let $J_b^{\der}$ be the derived subgroup of $J_b$. 
The following lemma will be used to describe the  connected components of $\M^{\bas}_{\K}$ in Section \ref{sec:inner}. 
\begin{lemma}\label{cpt}
Assume $r, s > 0$. 
Then  $J_b^{\der}({\Q_p})$ is compact if and only if $p$ is inert in $E$ and $(r, s)=(1, 1)$, or $p$ is inert in $E$ and $\gcd(r, s)=1$. 
\end{lemma}
\begin{proof}
 Assume first that $p$ is inert in $E$. 
 Then we have that $J_b^{\der} \simeq  \SU(\calV, \phi)$ for a Hermitian space $(\calV, \phi)$ of dimension $n\geq 2$. 
 When $n \geq 3$, 
the space $\calV$ is always isotropic and  $J_b^{\der}(\Q_p)$ is not compact. 
When $n=2$, namely $(r, s)=(1, 1)$, we have 
$d(\calV)=p$. 
Hence $\calV$ is anisotropic and 
 $J_b^{\der}(\Q_p)$ is compact. 
Next assume that $p$ is split in $E$. 
Then $J_b^{\der}({\Q_p})$ is isomorphic to  the group 
$\{ g \in  \GL_m(B) \mid {\rm nrd}(g)=1\}$, where ${\rm nrd}$ denotes the reduced norm on $\GL_m({B})$. This group   is compact if and only if $m =1$. 
\end{proof}
 \section{Irreducible and connected  components of  affine Deligne-Lusztig varieties}\label{sec:ADLV}
\subsection{Irreducible  components of ADLVs   for unramified  groups}
\subsubsection{}
Let $k$ be an algebraic closure of $\F_p$. 
Let $W(k)$ be the ring of Witt vectors over $k$ with Frobenius automorphism $\sigma$,  and let $L=W(k)[1/p]$ be  the fraction field of $W(k)$.  
Let $G$ be a connected reductive group scheme over $\Z_p$. Then $G_{\Q_p}$ is automatically unramified, and $G(\Z_p)$ is a hyperspecial subgroup of $G(\Q_p)$. 
%
We fix a maximal torus $T$ of $G$ and a Borel subgroup $B$ containing $T$. 
We use the same notations as those in Section \ref{div}. 
For a dominant coweight $\mu \in X_*(T)^{+}$ and an element $b \in G(L)$, we set  
\[X_{\mu}(b)(k)=\{ g \in G(L) /G(W(k)) \mid  g^{-1}b \sigma(g)\in G(W(k)) \mu(p)  G(W(k))\}. 
\]
We can identify 
 $X_{\mu}(b)(k)$ with  the set of $k$-points of a locally closed  subscheme $X_{\mu}(b)$ of the Witt vector partial  affine flag variety ${\rm Gr}_G$  constructed by Zhu \cite{Zhu} and Bhatt-Scholze \cite{BS}.  The scheme $X_{\mu}(b)$ is locally of perfectly finite type over $k$  \cite[Lemma 1.1]{HV}. 
We call 
$X_{\mu}(b)$ the \emph{affine Deligne-Lusztig variety}  associated to $(G, \mu, b)$. 
Note that affine Deligne-Lusztig varieties  $X_{\mu}(b)$  depends only on the $\sigma$-conjugacy class $[b] \in B(G)$ up to isomorphism. 
Further $X_{\mu}(b)$ 
 is non-empty if and only if $[b] \in B(G, \mu)$~\cite{Gashi}. 
 
 We write ${\rm Irr}(X_{\mu}(b))$ for the set of irreducible components of the affine Deligne-Lusztig variety $X_{\mu}(b)$, and ${\rm Irr}^{{\rm top}}(X_{\mu}(b))$ for the  set of those components which are top-dimensional. 
 It is conjectured that $X_{\mu}(b)$ is equi-dimensional (see  \cite[Theorem 3.4]{HV}). 
 However, it always  holds true  if  $\mu$ is minuscule. 


For each  $b \in G(L)$, let $J_b$ be an algebraic group  over $\Q_p$  as defined in \eqref{def:J}. 
Then $J_b(\Q_p)$  acts on the set $X_{\mu}(b)$ by left multiplication. 
This action induces an action of $J_b(\Q_p) $ on ${\rm Irr}^{\rm top}(X_{\mu}(b))$. 

\subsubsection{} 
Let $X_*(T)^{\sigma}$ and  $X_*(T)_{\sigma}$ denote the $\sigma$-invariants and coinvariants of $X_*(T)$, respectively.  
For each $\lambda \in X_*(T)$, let $\lambda^{\natural}$ denote its image in $\pi_1(G)_{\sigma}$,  $\underline{\lambda}$ denote its image in $X_*(T)_{\sigma}$, and  $\lambda^{\diamond}$ denote its $\sigma$-average in $X_*(T)_{\Q}$. 
Then we have a canonical isomorphism $X_*(T)_{\sigma,\Q} \xrightarrow{\sim} X_*(T)^{\sigma}_{\Q}$ where  $\underline{\lambda}\mapsto \lambda^{\diamond}$. 

Now we fix  $\mu \in X_*(T)^{+}$ and  $[b] \in B(G, \mu)$.
Let $\kappa_G(b) \in \pi_1(G)_{\sigma}$ denote the Kottwitz point and $\nu_G(b) \in X_*(T)_{\Q}^{+, \sigma}$  denote  the Newton point. 
Hamacher and Viehmann \cite[Lemma/Definition 2.1]{HV} proved that 
the set 
\[ \{ 
\underline{\lambda} \in X_{*}(T)_{\sigma} \mid  \underline{\lambda}^{\natural}=\kappa_{G}(b), \  \underline{\lambda}^{\diamond} \leq \nu_{G}(b)\} 
\]
has a unique maximum  element $\underline{\lambda}_G(b)$ characterized by the property that  $\underline{\lambda}_G(b)^{\natural}=\kappa_G(b)$ and that $\nu_G(b)-\underline{\lambda}_G(b)^{\diamond}$ is equal to a linear combination of simple coroots with coefficients in $[0, 1) \cap \Q$. 
This element can be regarded as ``the best integral approximation" of the newton point $\nu_G(b)$. 

 Let $\widehat{G}$ be the Langlands dual of $G$ defined over $\overline{\Q}_{\ell}$ with $\ell \neq p$. 
 Let $\wh{B}$ be a Borel subgroup of $\wh{G}$ with maximal torus $\wh{T}$ such that there exists a bijection    $X_*(T)^+ \simeq X^*(\wh{T})^+$. 
We write $V_{\mu}$ for the irreducible $\widehat{G}$-module of highest weight $\mu$. 
Let $V_{\mu}(\underline{\lambda}_G(b))$ be the sum of $\lambda$-weight spaces $V_{\mu}(\lambda)$ for $\lambda \in X_{*}(T)=X^*(\wh{T})$ satisfying  $\underline{\lambda}  = \underline{\lambda}_G(b)$. 

The following  theorem was  conjectured by Chen and X.~Zhu, and proved by  
Zhou-Y.~Zhu  and Nie. 
\begin{thm}[{\cite[Theorem A]{ZZ}, \cite[Theorem 4.10]{Nie}}]\label{MVcycle}
There exists a canonical bijection  between the set $J_b(\Q_p)\backslash {\rm Irr}^{{\rm top}}(X_{\mu}(b))$ and the  
Mirkovi\'{c}-Vilonen basis of $V_{\mu}(\underline{\lambda}_G(b))$ constructed in \cite{MV}. 
In particular, 
\[\lvert   J_b (\Q_p)\backslash {\rm Irr}^{{\rm top}} (X_{\mu}(b)) \rvert  = \dim V_{\mu}(\underline{\lambda}_G(b)).\]
\end{thm}
Note that this theorem  has first been shown by 
Xiao-X.~Zhu \cite{XZ}  for $[b]$ such that the  $\Q_p$-ranks of $J_b$ and $G$ coincide, and by  Hamacher-Viehmann \cite{HV} if  $\mu$ minuscule and either $G$ is split or $b$ is superbasic.


Let $W$ denote the absolute Weyl group of $G$. 
If $\mu$ is  minuscule, 
then  the dimension of  $V_{\mu}(\lambda)$ for a  $\lambda \in X_*(T)$  equals $1$  or $0$ according as $\lambda$ belongs to the $W$-orbit of $\mu$ or not. 
 Hence we have 
\begin{align}\label{V_mu} \dim V_{\mu}(\underline{\lambda}_G(b))
=\# ( W\cdot \mu \cap  \{\lambda \in X_{*}(T)  \mid \underline{\lambda} = \underline{\lambda}_G(b) \} 
). \end{align} 
\subsection{Irreducible components of ADLVs  for unitary   similitude groups}\label{ss:ur}
\subsubsection{}
Let $\mathscr D$, $p$, $\bfG$ be as in Section \ref{moduli}. We write $E_p =E\otimes _{\Q}\Q_p$, $V_p=V \otimes_{\Q}\Q_{p}$, and $\Lambda_p=\Lambda \otimes \Z_p$. 
Then $V_p$ is of rank $n$ over $E_p$ and equipped with a Hermitian form $\varphi$. 
Further, the $O_{E_p}$-lattice $\Lambda_p$ is unimodular.
In the rest of this section, let $G\coloneqq \bfG_{\Q_p} \simeq  \GU(V_p, \varphi)$. 

Assume first that $p$ is inert in $E$. 
For an $i\in I=\{1,\dots, n\}$,
write $i^\vee=n-i+1$. 
Let $\Phi$ denote the  anti-diagonal matrix whose $(i, i^{\vee})$-entry is $1$. 
By the classification of Hermitian spaces and lattices in \eqref{phi} and  \eqref{L_t},  the space $V_p$ has   a  basis $\{e_i,\dots, e_n\}$ over the field $E_p$  with Gram matrix $\Phi$. 
In other words, the Hermitian form $\varphi$ on $V_p$ is defined by  
$\varphi(e_{i}, e_{j^{\vee}})=\delta_{ij}$ for $i,j\in I$.

Fix an isomorphism $\End (V_p \otimes_{\Q_p}  \Qpbar)=\End(V_p)\otimes_{\Q_p} \Qpbar
\xrightarrow{\sim}  \Mat_n(\Qpbar)^2$ by $A\otimes \alpha\mapsto (A\alpha,
\bar A \alpha)$.  The matrix $\Phi\in \End(V_p)$ corresponds to $(\Phi,\Phi)$  and
we have 
\begin{align*}
G(\Qpbar)& 
=\left 
\{ 
(t_0, A,B) 
\in  \mathbb G_{{\rm {m}}}(\Qpbar) \times \GL_n(\Qpbar)^2 \mid (B^t,
A^t)(\Phi,\Phi)(A,B)
=(t_0\Phi,t_0\Phi) \right \} 
\\ 
  &
  =\left \{(t_0, A, B) \in  \mathbb G_{{\rm{m}}}(\Qpbar)
  \times \GL_n(\Qpbar)^2
   \mid  B=t_0 \Phi^{-1} (A^t)^{-1}
  \Phi\, \right \} 
  \\ 
  &
  \simeq  \mathbb G_{{\rm{m}}}(\Qpbar)  \times \GL_n(\Qpbar) 
  \quad 
  \text{(via the first and second  projections)}. 
 \end{align*}
 Let $\GmQpbar^n \subset \GL_n$ be  the diagonal maximal torus,  and 
 let  $T_{\Qpbar}=\GmQpbar \times
   \GmQpbar^n$ be a maximal torus of  $G _{\Qpbar} \simeq  \GmQpbar \times \GL_{n, \Qpbar}$.  
   Further, let 
  $B$ be the Borel subgroup of $G_{\Qpbar}$ given by the product of $\GmQpbar$ and the group of  upper triangular matrices in $\GL_{n, \Qpbar}$. 
  Let $\varepsilon_0, \varepsilon_1, \dots, \varepsilon_n$ be the $\Z$-basis of $X^*(T)$ defined by $\varepsilon_i(t)=t_i$ for any $t=(t_0, t_1, \ldots, t_n) 
 \in  T_{\Qpbar}$. 
  Let $\varepsilon_i^{*}\in X_*(T)$ denote   the dual of $\varepsilon_i$. 
Then the  Frobenius $\sigma$  acts on $X_*(T)$ by 
\begin{equation}\label{eq:sigma}
    \sigma(\varepsilon_0^*)= \varepsilon_{0}^*+ \cdots + \varepsilon_n^*, \quad  \sigma(\varepsilon_i^*)=-\varepsilon^*_{i^\vee} \   \text{for}  \ 1 \leq i \leq n. 
    \end{equation}
The $\sigma$-average of a cocharacter $\lambda=\sum_{i=0}^n \lambda_i \varepsilon_i^* \in X_*(T)$ is 
\[ \lambda^{\diamond}=\frac{1}{2}(\lambda+\sigma \lambda)=\lambda_0\varepsilon_0^*+ \frac{1}{2}\{ (\lambda_0+ \lambda_1-\lambda_n)\varepsilon_1^*+\dots +
(\lambda_0+\lambda_n-\lambda_1)\varepsilon_n^*\}.  \]
  
The map  $X_*(T) \to \Z^2$ where   $\sum_{i=0}^n a_i \varepsilon_i^* \mapsto (a_0, \sum_{i=1}^n a_i)$  induces an identification 
\begin{equation}\label{eq:pi_1}
  \pi_{1}(G) = X_*(T)/\sum_{\alpha \in \Phi^+} \Z \alpha^* \xrightarrow{\sim}\Z^2. 
\end{equation}
Here, $\Phi^+$ denotes  the set of positive roots  corresponding to $B$,  and $\alpha^*$ denotes the coroot corresponding to $\alpha$. 
Under this identification,  we have that $\sigma(1, 0)=(1, n)$ and $\sigma(0,1)=(0,-1)$.   
Hence $(0, n)$ and $(0, 2)$ generate the submodule $(1-\sigma)\pi_1(G)$. Further,  $\pi_1(G)^{\sigma}$ 
is generated by either $(2, n)$ when $n$ is odd, or $(1, n/2)$ when $n$ is even. 
It follows that  
\begin{equation}\label{eq:pi1s}
 \pi_{1}(G)_{\sigma}
 \simeq \begin{cases}
\Z & {\text{if $n$ is odd}};
\\
\Z \oplus \Z/2\Z &  
{\text{if $n$ is even}}, 
\end{cases} 
\quad \text{and} \quad 
\pi_1(G)^{\sigma}\simeq  \Z.
\end{equation}

  The simple roots are   
  $\Delta=\{\varepsilon_1-\varepsilon_2,
  \dots,\varepsilon_{n-1}-\varepsilon_{n}\}$ and the  simple coroots are  
  $\Delta^{\vee}=\{\varepsilon_1^*-\varepsilon_2^*,
  \dots,\varepsilon_{n-1}^*-\varepsilon_{n}^*\}$.  
 Hence,  a cocharacter $\lambda=\sum_{i=0}^n \lambda_i \varepsilon_i^* \in X_*(T)$ is dominant if and
  only if $\lambda_i\ge \lambda_{i+1}$ for all $i \geq 1$.

Let $[\mu]$ be the conjugacy class of  the cocharacter $\mu_h$ defined by $\mathscr D$, regarded as a $W$-orbit in $X_*(T)$ as in  Section~\ref{div}. 
Let $\mu$ be the dominant  representative of $[\mu]$. 
Then $\mu=\sum_{i=0}^r  \varepsilon_i^*$ or $\sum_{i=0}^s  \varepsilon_i^*$.  
Now we assume that $\mu=\sum_{i=0}^r  \varepsilon_i^*$. 
This assumption will be justified by the fact that replacing $r$ with $s=n-r$ does not change the final result, see Proposition \ref{Orbits}.   
Let $[b]$ be the  basic $\sigma$-conjugacy class in  $B(G, \mu)$. 
The 
Kottwitz point $\kappa_G(b)$ is equal to  the image of $\mu$ in $\pi_1(G)_{\sigma}$, and hence 
\begin{align}\label{kgb}
\kappa_G(b)=
\left\{
\begin{aligned} 
& 1   &  & \in \Z &  & {\text{if $n$ is odd}}; 
\\ 
& (1, r \bmod 2)    & &  \in \Z\oplus \Z/2\Z & & {\text{if $n$ is even}}. 
\end{aligned}
\right.
\end{align}

The Newton point $\nu_G(b) \in X_{*}(T)_{\Q}$ satisfies that  
\begin{equation*}\label{average} 
\nu_G(b)\leq 
\mu^{\diamond} 
=\varepsilon_0^{*}
+\varepsilon_{1}^{*}
+\cdots + \varepsilon_{\min(r,n-r)}^*+
\frac{1}{2} \Big(\varepsilon_{\min(r,n-r)+1}^*+\cdots+ \varepsilon_{\max(r,n-r)} \Big). 
\end{equation*} 
Further,  $\nu_G(b)$ factors through the $\sigma$-invariant part of the center $Z_G$ of $G$ since $[b]$ is basic. 
The subgroup 
  $X_*(Z_G)$ of $X_*(T)$ 
  is generated by $\varepsilon_0^*$ and   $\varepsilon_1^*+\dots+\varepsilon_n^*$. 
  It follows that 
 \begin{align*}
\nu_{G}(b) & = \varepsilon_0^*+\frac{1}{2}(\varepsilon_1^*+ \cdots \varepsilon_n^*) = (\varepsilon_0^*)^{\diamond}  \in X_*(T)_{\Q}^{\sigma}. 
\end{align*}

\begin{lemma}\label{apx:1}
We define an element $\tilde{\lambda}$ of $X_*(T)$ by   
\begin{align*}
\tilde{\lambda}=
\begin{cases}
 \varepsilon_0^*+ \varepsilon_{n/2+1}^* & {\text{if $n$ is even and $r$ is odd}}; 
 \\
 \varepsilon_0^* & {\text{otherwise}}. 
\end{cases}
\end{align*}
Then we have an equality as elements of $X_*(T)_{\sigma}$:  
\[\underline{\lambda}_G(b) =\ul{\tilde{\lambda}}.\]
\end{lemma}
\begin{proof}
The image of $\tilde{\lambda}$ in  $\pi_1(G)_{\sigma}$ is equal to the Kottwitz point $\kappa_G(b)$ by \eqref{kgb}. 
Further,  
\begin{align*}
\nu_G(b)- \tilde{\lambda}^{\diamond}= \begin{cases} -
(\varepsilon_{n/2+1}^*)^{\diamond} 
\\ 
0
\end{cases} \! \! \! \! \!
= \ \begin{cases}
\frac{1}{2}(
\varepsilon_{n/2}^*-\varepsilon_{n/2-1}^*) & \text{if $n$ is even and  $r$ is  odd}; 
\\
0 & \text{otherwise}, 
\end{cases}
\end{align*}
and  $\varepsilon_{n/2}^*-\varepsilon_{n/2-1}^*$ is a simple coroot. 
 The characterization of   $\ul{\lambda}_G(b)$ implies the assertion. 
\end{proof}
 \begin{prop}\label{Orbits}
We have that 
\begin{align*}
\lvert  J_b(\Q_p)\backslash {\rm Irr}(X_{\mu}(b)) \rvert =
\begin{cases}
\begin{pmatrix}
  n/2 -1 
 \\ 
 (r-1)/2 
\end{pmatrix} &  \text{if  $n$ is even and $r$ is odd};
\\
\begin{pmatrix}
\lfloor n/2 \rfloor 
\\ 
\lfloor r/2 \rfloor 
\end{pmatrix} & \text{if  $n$ is odd or $r$ is even}.
\end{cases}
\end{align*}
\end{prop}
\begin{proof}
We consider the following  collections of indices:
\begin{align}\label{ind}
 \calI \subset \begin{cases} 
 \{1, \ldots, n/2  -1 \}   &  {\text{if $n$ is even and $r$ is odd}};
 \\ 
 \{ 
 1, \ldots, \lfloor  n/2 \rfloor \}   &  \text{otherwise},
 \end{cases}
 \quad \text{and} \quad 
 \lvert \calI \rvert =\left\lfloor {r}/{2} \right\rfloor.
 \end{align}
  The number of such collections $\calI$ is equal to
  $\begin{pmatrix}
  n/2  -1 
 \\ 
 (r-1)/2 
\end{pmatrix}$ if $n$ is even and $r$ is 
odd, and to 
$
\begin{pmatrix}
\lfloor n/2 \rfloor 
\\ 
\lfloor r/2 \rfloor 
\end{pmatrix}$ otherwise. 

The $W$-orbits of $\mu$ are 
$\{ \varepsilon_0+\sum_{i \in \calJ}
 \varepsilon_{i}^* \in X_*(T)  \mid \calJ \subset \{1, \ldots, n\}, \lvert \calJ \rvert=r\}$. 
It follows from Lemma~\ref{apx:1}  that 
the assignment $\calI \mapsto \tilde{\lambda}+ 
\sum_{i \in \calI}
(\varepsilon_{i}^*+ \varepsilon_{i^{\vee}}^*) \in X_*(T)$  induces  a  bijection from the set of collections $\calI$  satisfying (\ref{ind}) to the intersection   $ 
 W\cdot \mu \cap  \{\lambda \in X_{*}(T)  \mid \ul{\lambda}  = \underline{\lambda}_G(b) \}$. 
The assertion thus  follows from  Theorem \ref{MVcycle} and equality \eqref{V_mu}.
\end{proof}
\subsubsection{} 
Next we assume that $p$ is split in $E$. Then $E_p=\Q_p \oplus \Q_p$. 
There is  an isomorphism $G\simeq 
\bbG_{{\rm m},\Q_p} \times \GL_{n, \Q_p}$ over   $\Q_p$, and in particular $G$ is split.  
We fix a maximal torus  $T$ and a Borel subgroup $B$ of $G$  as in the previous case.  Let $\mu \in X_*(T)^{+}$ be the coweight defined by $\mathscr D$, 
and 
let $[b]$ be the basic $\sigma$-conjugacy class in $B(G, \mu)$. 
%
 \begin{prop}\label{Orbits:2}
We have that 
\begin{align*}
\lvert  J_b(\Q_p)\backslash {\rm Irr}(X_{\mu}(b)) \rvert =
1.
\end{align*}
\end{prop}
\begin{proof}
Since the Frobenius $\sigma$ acts on $X_*(T)$ trivially, the set   
  $\{\lambda \in X_{*}(T)  \mid \ul{\lambda} =  \underline{\lambda}_G(b) \}$ is a singleton $\{\underline{\lambda}_G(b)\}$.  
  Hence the assertion follows from Theorem \ref{MVcycle} and equality \eqref{V_mu}.
\end{proof}
For a description of  $\underline{\lambda}_G(b)$ in this case,  see  {\cite[Example 2.3]{HV}}. 
 

 \subsection{Connected components of ADLVs for unitary similitude groups}
 \subsubsection{}
 Let $w : X_*(T) \to \pi_1(G)$ denote the canonical projection. 
 As in \cite{Kottwitz85}, there exists a map   $w_G : G(L) \to 
 \pi_1(G)$ sending an element $b \in G(\Z_p) \lambda(p) G(\Z_p) \subset G(L)$ for $\lambda \in X_*(T)$ to $w(\lambda)$. 
 For each $b \in G(L)$, the projection of  $w_G(b)$ to $\pi_1(G)_{\sigma}$ coincides with  $\kappa_G(b)$. 
  
 Let $\mu \in X_*(T)^+$ be the coweight  defined by $\mathscr D$,  and $[b]  \in B(G, \mu)$ be the basic $\sigma$-conjugacy class, as in Section \ref{ss:ur}. 
Then $\kappa_G(b)=\mu^{\natural}=w(\mu) \bmod (1-\sigma) \in \pi_1(G)_{\sigma}$. 
Hence 
 $w_G(b)-w(\mu)=(1- \sigma)(c_{b, \mu})$ 
 for an element $c_{b, \mu} \in \pi_1(G)$, whose $\pi_1(G)^{\sigma}$-coset is  uniquely determined by this relation. 
 
 We note that the adjoint group of $G$ is simple. 
 Further $\mu$ is minuscule and 
 the pair $(\mu, b)$ is Hodge-Newton indecomposable in the sense of \cite[Section~2.2.5]{CKV}. 
 By the work of Viehmann \cite{Viehmann} and Chen-Kisin-Viehmann  \cite[Theorem 1.1]{CKV},  the set $\pi_0(X_{\mu}(b))$  of connected components of $X_{\mu}(b)$ is   described as follows. 
 If  $rs >0$, then  $w_G$ induces  a bijection 
 \begin{equation}\label{eq:pi_0}
     \pi_0(X_{\mu}(b))  \simeq  c_{b, \mu} \pi_1(G)^{\sigma}. 
     \end{equation}
     If $rs=0$, then $X_{\mu}(b) \simeq G(\Q_p)/G(\Z_p)$ is discrete. 
  
Moreover, the group $J_b(\Q_p)$ acts transitively on $\pi_0(X_{\mu}(b))$ by   \cite[Theorem 1.2]{CKV}. 
 \begin{prop}\label{stab}
 Assume that $rs>0$. 
 
 {\rm (1)}
 There is a bijection
 \begin{align*}
     \pi_0(X_{\mu}(b)) \simeq 
     \begin{cases*}
     \Z & if $p$ is inert in $E$; 
     \\
     \Z \oplus \Z 
     & if $p$ is split in $E$. 
     \end{cases*}
 \end{align*}
 
 {\rm (2)} 
  As in Proposition  \ref{J_Minert}, we fix an   identification  \begin{align*}
   J_b(\Q_p)\simeq \begin{cases*}
   \GU(\calV, \phi)(\Q_p) & if   $p$ is inert in $E$;
   \\ 
     \GL_m(B) \times \Q_p^{\times} 
     & if $p$ is split in $E$. 
     \end{cases*}
     \end{align*}
     Let $J_b^0$ denote the stabilizer in $J_b(\Q_p)$ of a fixed connected component of $X_{\mu}(b)$. 
    Then 
 \[ J_b^0  \simeq  
 \begin{cases*}
 \{ g \in \GU(\calV, \phi)(\Q_p) \mid 
  v_p(\Sim g)=0\} & if $p$ is inert in $E$; 
  \\ 
  \{(g,c) \in \GL_m({B}) \times \Q_p^{\times}  \mid v_p({\rm nrd}(g))=v_p(c)=0\}
  & if $p$ is split in $E$,  
  \end{cases*}\]
  where $v_p$ denotes the $p$-adic valuation on $\Q_p$, and ${\rm nrd}$ denotes the reduced norm on $\GL_m({B})$. 
 \end{prop}
 \begin{proof}
 (1) 
       When $p$ is inert in $E$, the assertion follows from  \eqref{eq:pi_0} and  \eqref{eq:pi1s}.    
       A similar geometric  argument works for the case 
 $p$ is split, and hence we have  $\pi_1(G)^{\sigma}=\pi_1(G)\simeq \Z^2$ as in \eqref{eq:pi_1}. 
    
    (2) 
     Assume that $p$ is inert in $E$. We write $V_p=V \otimes_{\Q} \Q_p$. We fix an isomorphism $E_p \otimes_{\Q_p}L \simeq L \oplus L$, which induces a decomposition $V_p\otimes_{\Q_p} L \simeq V_1 \oplus V_2$. 
Then  $G(L)=\GU(V_p, \varphi)(L) \simeq   \GL_L(V_1) \times 
L^{\times}$. 
Under this identification, the map $w_G$ is 
    given by 
\begin{align*}
    w_{G} : G(L)  & \to \pi_1(G)    \simeq \Z \oplus \Z, 
    \\ 
    (h, c) &  \mapsto 
    \big(v_p(\det h), v_p(c) \big)
    \end{align*}
    where $v_p$ denotes the $p$-adic valuation on $L$. 
    
   On the other hand, as in the proof of Proposition  \ref{J_Minert} (1), we have an inclusion  $\End_{E_p}\calV  \hookrightarrow \End_L V_1$, 
  under which  the inclusion $J_b(\Q_p) \hookrightarrow G(L) \simeq \GL_L(V_1) \times L^{\times}$  can be regarded as the one sending $g  \in J_b(\Q_p)=\GU(\calV, \phi)(\Q_p) \subset \GL_{E_p}(\calV)$ to $(g,  \Sim(g))$. 
     It follows that,  
     for each element of $c_{b, \mu}\pi_1(G)^{\sigma} \subset \pi_1(G)$, 
     its  stabilizer in $J_b(\Q_p)$  consists of elements $g$  with  $v_p(\det(g))=0$ and $v_p(\Sim(g))=0$. 
 However, the latter condition is sufficient   since  $\bfN_{E_p/\Q_p}(\det(g))=\Sim(g)^n$. 
 
 The case  $p$ is split in $E$ can be proved similarly. 
\end{proof}
\section{Mass formula for the inner form associated to the  basic locus}\label{sec:inner}
  \subsection{The group of self quasi-isogenies of an abelian variety}
  \subsubsection{}
 Let $\mathscr D$ and   $p$ be as in Section ~\ref{moduli},  and $\bfG=\GU(\Lambda, \< \, , \, \>)$. 
 We assume that $\Lambda$ is self-dual with respect to  $\< \, , \, \>$. 
 We write  $\mu$ for the cocharacter defined by $\mathscr D$. 
 Let $\K\coloneqq \bfG(\Z_p)\K^p(N)$ with $N \geq 3$ and $p \nmid N$.

Let $(A, \iota,  {\lambda}, \bar{\eta})$ be a $k$-point of  $\mathcal M_{\rm K}$.    
Let $\End^0(A)$ denote the $\Q$-algebra of  self  quasi-isogenies of the abelian variety $A$. 
 It admits an injective homomorphism $\iota : E \hookrightarrow\End^0(A)$. 
We regard $E$ as a subalgebra of  $\End^0(A)$ via $\iota$, and write $\End_E^0(A)$ for its centralizer. 
\begin{lemma}\label{simple}
If $(A, \iota, \lambda, \bar{\eta})$ is lying on the basic locus $\mathcal M^{\rm bas}_{\rm K}$, then 
 $\End_E^0(A)$ is a central simple algebra over $E$. 
\end{lemma}
\begin{proof}
There exist a finite field $\F_q\subset k$,   and a model $(A', \iota')$ over $\F_q$ of $(A,  \iota)$   such that 
  $\End^0(A') \simeq  \End^0 (A)$.  
Let $A' \sim 
\prod_{i=i}^t (A'_i)^{l_i}$ be the decomposition into components 
up to isogeny, where each abelian variety $A'_i$ is simple and $A'_i \not\sim A'_j$ for any $i\neq j$. 
Then we have an isomorphism 
 $\End^0(A') \simeq \prod_{i=1}^t \Mat_{l_i}(D_i)$, where $D_i \coloneqq \End^0(A'_i)$ is a division algebra over $\Q$. 
 Let $\pi_i$ be the Frobenius
endomorphism of $A'_i$. 
Then the center $Z(\End^0 (A'))$ of $\End^0 (A')$ is equal to  $\prod_{i=1}^t \Q(\pi_i)$. 
Moreover, the assumption  and 
\cite[Proposition 4.2]{Yu:basic} imply that $Z(\End^0 (A'))$ is contained in $E$. 
Hence the abelian variety $A'$ has only one component, say $(A'_1)^{l_1}$,  and the algebra $\End^0 (A')=\Mat_{l_1}(D_1)$ is a central simple algebra over $Z(\End^0 (A'))=\Q(\pi_1)$. 
It follows from the double centralizer theorem  that  the ring $\End_E^0(A')\simeq \End_E^0(A)$ is a central simple algebra over $E$.  
\end{proof}
By Lemma \ref{simple}, there exist   a finite extension $E'/E$ and an isomorphism  $f : \End^0_E(A) \otimes_E E' \xrightarrow{\sim}  \Mat_{n}(E')$.   
The reduced norm of an element  $g \in \End^0_E(A)$ is defined by ${\rm nrd}(g)\coloneqq \det(f(g \otimes 1))$, which  
 takes value in $E$.  
  \subsubsection{}\label{I}
For a $k$-point $(A, \iota,  {\lambda}, \bar{\eta})$ of $\M_{\K}$,  we define an algebraic  group   $I$  over $\Q$ by 
 \begin{align*}
 I(R)
 \coloneqq \{(g,c) \in (\End_{E}^0(A) \otimes R)^{\times} \times R^{\times} \mid g^*\cdot g=c\cdot \id \}, 
 \end{align*} 
 where $R$ is a  $\Q$-algebra, and  $g \mapsto g^*$ is the Rosati involution of $\End_{E}^0(A)$ induced by the principal polarization $\lambda$. 
The similitude character ${\rm sim} :I \to {\bbG}_{{\rm m}, \Q}$ is defined by $(g, c) \mapsto c$. 
\begin{thm}\label{inner}
Assume that $(A, \iota,  {\lambda}, \bar{\eta})$ is lying on the basic locus $\M_{\K}^{\bas}.$

{\rm (1)} 
The group $I$ is an inner form of ${\mathbf  G}_{\Q}$, and is such that $I(\R)$ is compact modulo center. 
Further,  there are isomorphisms
\begin{align*}
I_{\Q_{\ell}} 
\simeq 
\begin{cases} 
\mathbf  G_{\Q_{\ell}}  & \  {\rm if} \ \ell \neq p; 
\\ 
 J_b & \ {\rm if} \ \ell=p,  
 \end{cases}
 \end{align*}
 where  $b \in \bfG(L)$ is a representative of  the unique basic class $[b] \in B(\bfG_{\Q_p}, \mu)$. 
 
 {\rm (2)}
 For any point $(A', \iota',  {\lambda'}, \bar{\eta}')  \in \mathcal M^{\rm bas}_{\rm K}(k)$, 
 the associated group $I'$ over $\Q$ is  isomorphic to $I$ as inner forms of ${\bf G}_{\Q}$. 
 
 {\rm (3)}
There is an isomorphism of perfect  schemes 
\[ \Theta : I(\Q) \backslash 
X_{\mu}(b) \times \mathbf  G(\A_f^{p})/{\rm K}^p(N)
\xrightarrow{\sim}
 \mathcal M_{{\rm K}}^{\bas, \rm pfn},  
\]
where $\mathcal M_{{\rm K}}^{\rm{bas}, {\rm pfn}}$ denotes the perfection of $\mathcal M_{{\rm K}}^{{\rm bas}}$. 

 {\rm (4)} 
 Asssume that $(A, \iota,  {\lambda}, \bar{\eta}) \in \mathcal M_{{\rm K}}^e(k)$. 
 Let $I_p^e$ be the stabilizer in $J_b(\Q_p)$ of the associated \dieu module, as in  \eqref{I_p^e}. 
Then 
  $\Theta$ induces a  bijection  
\begin{align*}
 I(\Q) \backslash 
I(\A_f) /{\rm I}_p^e\cdot{{\rm K}^p}(N) 
 \xrightarrow{\sim} \mathcal M_{\rm K}^e(k). 
\end{align*}
\end{thm}
\begin{proof}
Assertions (1), (2), and (3) are proved by   Rapoport and Zink in   \cite{RZ}. 
Note that they  constructed an isomorphism from 
a quotient of what is now called a Rapoport-Zink formal scheme to the completion of the integral model along
the basic locus (see also Remark~\ref{rem:MK} (2)). 
An isomorphism  using an affine  Deligne-Lusztig variety was proved in \cite[Corollary 7.2.16]{XZ} and \cite[Proposition 5.2.2]{HZZ} for a  Hodge-type Shimura variety. 

By \cite[Corollary 3.4]{TY}, the   morphism $\Theta$  induces a bijection from the double coset space in (4) 
to the central leaf  passing through $(A, \iota,  {\lambda}, \bar{\eta})$ in $\M_{\K}(k)$. 
This leaf coincides with  $\M_{\K}^e(k)$ by Proposition \ref{CL} (1), and hence assertion (4) follows.    
\end{proof}
\subsubsection{}
Assume that $(A, \iota, \lambda, \bar{\eta}) \in \M_{\K}^{\bas}(k)$.  
Since $I$ is an inner form of $\bfG_{\Q}$, its derived subgroup $I^{\der}$ is again simply connected, and 
 the quotient torus  $I/I^{\rm der}$ is  isomorphic to  $D =  \bfG_{\Q}/\bfG_{\Q}^{\rm der}$ (Section ~\ref{tori}).  
Let $\nu : I 
\to 
D$ denote the projection. 
The reduced norms of elements of  $\End^0_E(\mathcal A_x)$  induce a homomorphism ${\rm nrd} : I \to T^{E}$ of algebraic groups over $\Q$. 
We write ${\bf N}=\mathbf N_{E/\Q}$ for the norm map of $E/\Q$.  
Then each element $(g, c) \in I$ satisfies that ${\bf N}({\rm nrd}(g))=c^n$. 
The projection  $\nu$ is thus equal to the product  ${\rm nrd} \times {\rm sim}$   under the  identification in \eqref{def:D}. 


Let $\Q_{>0}$ (resp.~$\R_{>0}$) denote the group of positive rational (resp.~real) numbers. 
The norm group $\bfN(E^{\times})$ is contained in $\Q_{>0}$ since $E$ is an imaginary quadratic field. 
We write $E^1 \coloneqq \{ x \in E \mid \bfN(x)=1\}$. 
\begin{lemma}\label{sim:I}
We have that 
\begin{align*}
    \nu(I(\Q))=
    \begin{cases*}
    E^{\times} 
    \\
    E^{1} \times \Q_{>0},  
    \end{cases*}
    \quad 
    \Sim(I(\Q))=\begin{cases*}
    \bfN(E^{\times}) & if $n$ is odd; 
    \\ 
    \Q_{>0} & if $n$ is even. 
    \end{cases*} 
\end{align*}
\end{lemma}
\begin{proof}
The projection $\nu : I \to D$ induces a commutative diagram 
\[
 \begin{CD}
  I(\Q)   
  @>{\nu}>> 
 D(\Q) 
  @>>> 
  H^1(\Q, I^{\rm der})  
  \\ 
  @VVV 
   @VVV 
 @VVV
  \\ 
 I(\R)  
  @>{\nu}>>
  D(\R) 
  @>>>
 H^1(\R, I^{\rm der}), 
  \end{CD}
  \]
where the horizontal sequences  are exact, and the left and middle vertical  arrows are  injective. 
Furthermore,   Kneser's theorem and  Hasse principle imply  that the right vertical arrow   is also injective. 
A diagram chasing therefore shows $\nu(I(\Q))=D(\Q) \cap \nu(I(\R))$. 

Since $I(\R)$ is connected, its image $\nu(I(\R))$ is the  identity component $D(\R)^0$ of $D(\R)$. 
As in Section ~\ref{tori}, 
it follows that the group  
$\nu(I(\Q))=D(\Q)_{\infty}$ is isomorphic to  $E^{\times}$ if $n$ is odd and to 
  $E^1 \times \Q_{>0}$ if $n$ is even.  
Their similitude factors are $\bfN(E^{\times})$ and $\Q_{>0}$, respectively. 
\end{proof}
For each prime $\ell$ we write  $E_{\ell}= E \otimes _{\Q}\Q_{\ell}$.    
Similarly to the case of $\bfG_{\Q}$ (see \eqref{eq:nu_ell}), we have  \begin{align}\label{eq:nuI}
    \nu(I(\Q_{\ell}))=D(\Q_{\ell})=
    \begin{cases*}
    E_{\ell}^{\times} 
    \\
    E_{\ell}^{1} \times \Q_{\ell}^{\times},  
    \end{cases*}
    \quad 
    \Sim(I(\Q_{\ell}))=\begin{cases*}
    \bfN(E_{\ell}^{\times}) & if $n$ is odd; 
    \\ 
    \Q_{\ell}^{\times} & if $n$ is even. 
    \end{cases*} 
\end{align}
\begin{lemma}\label{I:dense}
The group $I(\Q)$ is dense in $I(\Qp)$.
\end{lemma}
\begin{proof}
Since the norm one subgroup $T^{E, 1}$ of $T^E=\Res_{E/\Q} \bbG_{{\rm m},E}$ is a unitary group, it is a $\Q$-rational algebraic variety and hence satisfies the weak approximation property \cite[Propositions 7.3 and 7.4, p.~402-403]{PR}. Thus, $E^1$ is dense in $E_p^1$. 
From  Lemma \ref{sim:I} and equality \eqref{eq:nuI} it follows that 
$\nu(I(\Q))$ is dense in $\nu(I(\Qp))$. 
On the other hand, since $I^{\rm der}$ is a $\Q$-simple group, it has weak approximation \cite[Proposition 7.11, p.~422]{PR}. Therefore, ${I^{\rm der}(\Q)}$ is dense in  $I^{\rm der}(\Qp)$. Since $\nu(\ol{I(\Q)})=\nu(I(\Qp))$ and $\ol{I(\Q)}\supset I^{\rm der}(\Qp)$, we have $\ol{I(\Q)}=I(\Qp)$. This proves the lemma.  
\end{proof}

\subsection{Mass formula for maximal parahoric subgroups at $p$}
\subsubsection{}\label{mass:setting}
Let $(A, \iota, \lambda, \bar{\eta}) \in \mathcal M^{\rm bas}_{\rm K}(k)$. 
We fix  identifications  \begin{align*}
I(\Q_p)=J_b(\Q_p)=\begin{cases}\GU(\calV, \phi)(\Q_p) & {\text{if $p$ is inert in $E$}}; 
 \\ 
 \GL_m(B) \times \Q_p^{\times}  & {\text{if $p$ is split in $E$}}. 
 \end{cases}
 \end{align*}
 Here $(\calV, \phi)$  is a Hermitian space over $E_p$
with $d(\calV)=p^s$, and $B$ is a division algebra
over $\Q_p$ of degree $(n/m)^2$  with $\inv(B) = r/n$, see  Proposition \ref{J_Minert}.

Let ${\rm I}_p$ be a maximal parahoric subgroup  of $I(\Q_p)$.  As in Section ~\ref{mprh}, we have that 
 \begin{align}\label{eq:K_p}
 {\rm I}_p \sim_{\rm conj}
 \begin{cases}
 \calP_t  \quad {\text{for some $0 \leq t\leq n$ with $t \equiv s \pmod 2$}} & {\text{if $p$ is inert in $E$}}; 
 \\ 
 \GL_m(O_B) 
 \times \Z_p^{\times} & {\text{if $p$ is split in $E$}}. 
 \end{cases}
 \end{align}
 Similarly,
 we fix an identification $I(\A_f^p)=\mathbf G(\A_f^p)$ and regard $\mathbf G(\widehat{\Z}^p)$ as a subgroup of $I(\A_f^p)$. 
\begin{lemma}\label{sim}
The similitude character induces a surjective map 
 \begin{align*}
{\rm sim}  :  I(\Q) \backslash I(\A_f) / 
{\rm I}_p{\bf G}(\widehat{\Z}^p) 
\to 
 \begin{dcases}
  \mathbf N(\A_{E, f}^{\times}) /
 {\bf N}(E^{\times}) \cdot 
 \mathbf N(\widehat{O}_E^{\times}) 
 & {\text{if $n$ is odd}}, 
 \\ 
 \A_f^{\times}/\Q_{>0} \cdot  \widehat{\Z}^{\times} &
 {\text{if $n$ is even}}. 
 \end{dcases}
 \end{align*}
 Moreover, the cardinality $\tau$ of the quotient group in the RHS  is given by 
 \begin{align*}
\tau= 
\begin{cases} 
 2^{w-1} \quad \text{where} \quad w \coloneqq \#  \{  \ell : {\text{prime}}, \ell \mid d_E\} & {\text{if $n$ is odd}};  
 \\
 1  & {\text{if  $n$ is even}}.  
 \end{cases} 
 \end{align*}
\end{lemma}
\begin{proof}
By \eqref{eq:nuI}, 
the group  $\Sim(I(\A_f))$ equals to either  $\bfN(\A^{\times}_{E, f})$ when $n$ is odd,   or  $\A_f^{\times}$ when   $n$ is  even. 
     For $\ell \neq p$, the subgroup   $\bfG(\Z_{\ell})$ is the stabilizer of the unimodular lattice $\Lambda_{\ell}$ in $\GU(V_{\ell}, \varphi)$, and therefore Lemma \ref{mult} shows that $\Sim(\bfG(\Z_{\ell}))$ equals either  $\bfN(O_{E_{\ell}}^{\times})$ when $n$ is odd, or  $\Z_{\ell}^{\times}$ when  $n$ is even. 
     Further, when $p$ is inert in $E$, then  $\calP_t$ is by definition the stabilizer of an $O_{E_p}$-lattice $\calL_t$, see \eqref{L_t}.  
     Hence  Lemma 
     \ref{multunram}  implies that  $\Sim({\rm I}_p)=\Sim(\calP_t)=\bfN(O_{E_p}^{\times})=\Z_p^{\times}$. 
     When $p$ is split in $E$, then $\Sim({\rm I}_p)=\Sim(\GL_m(O_B) \times \Z_p^{\times})=\Z_p^{\times}$. 
     These and 
Lemma \ref{sim:I} imply the first assertion. 
    
   If $n$ is even, the quotient group is trivial. 
Suppose that $n$ is odd. 
We consider the inclusions 
\[
 \mathbf N(\widehat{O}_E^{\times}) \cdot \Q_{>0} \subset 
 \mathbf N(\A_{E, f}^{\times}) \cdot \Q_{>0} 
\subset 
\A_{f}^{\times}.
\]
As in  \cite[Formula (4.7)]{GSY}, we have 
\begin{equation}\label{quot}
\mathbf N(\A_{E, f}^{\times}) /
 {\bf N}(E^{\times}) \cdot 
 \mathbf N(\widehat{O}_E^{\times}) \simeq 
\mathbf N(\A_{E, f}^{\times}) \cdot \Q_{>0} /
 \mathbf N(\widehat{O}_E^{\times})  \cdot \Q_{>0}. 
\end{equation}
Further, a direct computation shows that 
\begin{align}\label{eq:norm}
[\A_f^{\times} : \mathbf N(\widehat{O}_E^{\times}) \cdot \Q_{>0}]=
[ \widehat{\Z}^{\times} : 
\mathbf N(\widehat{O}_E^{\times})]= \prod_{\ell}[\Z_{\ell}^{\times} : \mathbf N(\widehat{O}_{E_\ell}^{\times})]
=2^w. 
\end{align}
Moreover, an equality ${\bf N}(\C^{\times})=\R_{>0}$ and the norm index theorem imply that 
\begin{align}\label{nit}
[ 
\A_f^{\times} : \mathbf  N(\A_{E, f}^{\times}) \cdot \Q_{>0}]
 =[\A^{\times} : 
\bfN(\A_{E}^{\times}) \cdot \Q^{\times}]=2. 
\end{align}
Equalities  \eqref{quot},  (\ref{eq:norm}),  and  (\ref{nit}) show that 
\[ [\mathbf N(\A_{E, f}^{\times}) : 
 {\bf N}(E^{\times}) \cdot 
 \mathbf N(\widehat{O}_E^{\times})]= 
[ \mathbf N(\A_{E, f}^{\times}) \cdot \Q_{>0} : 
 \mathbf N(\widehat{O}_E^{\times})  \cdot \Q_{>0} 
 ]=
 2^{w-1}.
 \]
\end{proof}
We write  $I^{1}$ for the kernel of the similitude character of  $I$. 
Let $\Gamma$ be an  arithmetic subgroup of $I(\Q)$. Then  $\Sim(\Gamma) \subset \Z^{\times} \cap \Q_{>0}=\{1\}$ by Lemma \ref{sim:I}, and hence   $\Gamma  \subset I^1(\Q)$.
 Since $I^1(\R)$ is compact, $\Gamma$ is finite.  

  Let $U$ be an open compact subgroup of $I(\A_f)$. 
  Let $[g]$ be a double coset in  $I(\Q) \backslash I(\A_f) / U$,   represented by $g\in I(\A_f)$. 
  Then   $\Gamma_{g}\coloneqq I(\Q) \cap g  U g^{-1}$ is finite. 
  A similar statement holds true also for an open compact subgroup $U^1$ of $I^1(\A_f)$ and a coset $[g] \in I^1(\Q) \backslash I^1(\A_f) / U^1$. 
 \begin{defn}\label{mass:I}
 The \emph{mass of $I$ with respect to $U$} is defined by 
\begin{align}\label{eq:mass}
\Mass(I, U)   \coloneqq  \sum_{[g] \in I(\Q) \backslash I(\A_f) / U}  \frac{1}{ \lvert \Gamma_{g} \rvert} .
\end{align}
The \emph{mass of $I^1$ with respect to $U^1$} is defined similarly  and denoted by $\Mass(I^1, U^1)$. 
\end{defn}
We write ${\rm I}_p^1 \coloneqq {\rm I}_p \cap I^1(\Q_p)$. 
\begin{prop}\label{mass:sim} 
Let $\tau$ be as in Lemma \ref{sim}. 
 Then  
\begin{align*}
\Mass \big(I, {\rm I}_p {\bf G}(\widehat{\Z}^p)\big)= \tau  \cdot 
\Mass \big(I^1, {\rm I}_p^1 {\bf G}^1(\widehat{\Z}^p)\big).
\end{align*}
\end{prop}
\begin{proof}
 We put $U \coloneqq {\rm I}_p {\bf G}(\widehat{\Z}^p)$ and $U^1 \coloneqq {\rm I}_p^1 {\bf G}^1(\widehat{\Z}^p)$. 
The map in Lemma \ref{sim}  induces a decomposition 
\[I(\Q) \backslash I(\A_f)/U= \bigcup_{i=1}^{\tau} I(\Q) \backslash I(\Q)   I^1(\A_f) f_i U /U, \]
where $f_1, \ldots, f_{\tau}$ are some elements of $I(\A_f)$. 
The right multiplication by  $f_i^{-1}$ induces a bijection 
\begin{align}\label{eq:f_i}
I(\Q) \backslash I(\Q)  I^1(\A_f)  f_i U / U  \xrightarrow{\sim} 
I(\Q) \backslash I(\Q) I^1(\A_f) f_i U f_i^{-1}/f_i U f_i^{-1}.\end{align}
We put $U_i^1 \coloneqq f_iUf_i^{-1} \cap I^1(\A_f)$. 
 Then a computation of similitude factors shows $U^1_i=f_iU^{1}f_i^{-1}$. 
Furthermore, there is a bijection 
 \begin{align}\label{eq:iden}
 I^1(\Q) \backslash I^1(\A_f) /U_i^{1} \xrightarrow{\sim}
 I(\Q) \backslash I(\Q) I^1(\A_f)  f_i U f_i^{-1}/ f_iU f_i^{-1}.\end{align}
 In fact, the natural projection  induces a well-defined  map. 
We show this map  is injective.  Suppose that two elements $x_1, x_2\in I^1(\A_f)$ satisfy $\gamma x_1 u=x_2$ for some $\gamma \in I(\Q)$ and $u\in f_iUf_i^{-1}$.  Then 
${\rm sim}(\gamma)\cdot {\rm sim}(u)=1$. 
 We have  ${\rm sim}(\gamma) \in \Q_{>0}$ as in  Lemma \ref{sim:I}, and $\Sim(u) \in \widehat{\Z}^{\times}$. 
  It follows that ${\rm sim}(\gamma)={\rm sim}(u)=1$, and  hence $\gamma \in I^1(\Q)$ and $u \in U_i^1$, as desired. 

For each $1 \leq i \leq \tau$, let $g_{i1}, \ldots, g_{im_i}$ be elements of   $I^1(\A_f)$ such that $g_{i1}f_i, \ldots, g_{im_i}f_i$ are representatives of double cosets in $I(\Q) \backslash I(\Q)   I^1(\A_f) f_i U /U$. 
Since $\Gamma_{g_{ij} f_i}  \subset I^1(\Q)$, we have 
\begin{align}\label{eq:Gamma}
\Gamma_{g_{ij} f_i}
=I^1(\Q) \cap g_{ij}U_i^1 g_{ij}^{-1}.\end{align}
 It follows that 
   \begin{align*} 
    \Mass(I,U) 
    &=\sum_{i=1}^{\tau} \sum_{j=1}^{m_i} 
\frac{1}{\lvert \Gamma _{g_{ij}f_i}\rvert} &  &  [\eqref{eq:mass}]
\\ 
& = 
\sum_{i=1}^{\tau} 
\sum_{j=1}^{m_i}  
\frac{1}{\lvert
I^1(\Q) \cap 
g_{ij} U^1_i g_{ij}^{-1}
\rvert} &  & [\eqref{eq:f_i}, \eqref{eq:iden}, \eqref{eq:Gamma}]
\\
& = 
\sum_{i=1}^{\tau}
\Mass (I^1, U_i^1)
= \tau
\cdot 
\Mass (I^1, U^{1}). & & 
\end{align*}
\end{proof}
\subsubsection{}
We write  $\bfG^1={\rm U}(\Lambda, \varphi)$ for the kernel of the similitude character of $\bfG$. 
Recall that $I_{\R}^1$  is a  compact inner form of $\mathbf G_{\R}^1$. 
 Let $\mu_{I^1(\R)}$ be the Haar measure on $I^1(\R)$ which gives this group volume one. 
    The pull-back of  $\mu_{I^1(\R)}$  via an inner twist gives a Haar measure on ${\bf G}^1(\R)$,  denoted by  $\mu_{{\bf G}^1(\R)}$. 
    Further,  for each prime $\ell$, let $\mu_{{\bf G}^1(\Q_{\ell})}$ be the Haar measure on ${\bf G}^1(\Q_{\ell})$ which gives ${\bf G}^1(\Z_{\ell})$ volume one. 
 Finally we put $\mu_{{\bf G}^1(\A)}
   \coloneqq \mu_{{\bf G}^1(\R)} \times \prod_{\ell} \mu_{{\bf G}^1(\Q_{\ell})}$. 
 By Definition \ref{def:mass} we have that
   \begin{align}
   \label{eq:G}
 \Mass (\Lambda) = \int_{{\bf G}^1(\Q) \backslash {\bf G}^1(\A)} \mu_{{\bf G}^1(\A)}. 
 \end{align}
 
   Let $\mu_{I^1(\Q_{p})}$ be the Haar measure on $I^1(\Q_p)$ which gives ${\rm I}_p^1$ volume one. 
   For $\ell\neq p$, we put  $\mu_{I^1(\Q_{\ell})}=\mu_{{\bf G}^1(\Q_{\ell})}$ under the fixed  identification $I^1(\Q_{\ell})\simeq {\bf G}^1(\Q_{\ell})$. 
  Further   
  we put  $\mu_{I^1(\A)}\coloneqq  \mu_{I^1(\R)} \times \prod_{\ell} \mu_{I^1(\Q_{\ell})}$. 
 Then  
   \begin{align} \label{eq:omega}
   \Mass(I^1,  {\rm I}_p^1{\bf G}^1(\widehat{\Z}^p))  =  \int_{I^1(\Q) \backslash I^1(\A)} \mu_{I^1(\A)}. 
 \end{align}

Fix an inner twisting  
$f : I^1_{\Q_p}\to {\bfG}_{\Q_p}^1$  over an extension of $\Q_p$,  and let $\mu^*_{I^1(\Q_p)}$ be the Haar measure on $I^1(\Q_p)$ defined by the pull-back $\mu^*_{I^1(\Q_p)} \coloneqq f^*(\mu_{{\bf G}^1(\Q_p)})$. 
Let
\begin{equation}\label{lpKp}
  \lambda_p({\rm I}_p) \coloneqq  \Bigg( 
  \int_{{\rm I}_p^1}
 \mu^*_{I^1(\Q_p)}\Bigg) ^{-1}.  
\end{equation}
It follows from \eqref{eq:G} and  \eqref{eq:omega} that  
\begin{align}\label{eq:GHY}
\frac{ \Mass \big(I^1,  {\rm I}_p^1  \mathbf G^1(\widehat{\Z}^p) \big) }{
\Mass \big(\Lambda
\big)}
 =\lambda_p({\rm I}_p). 
\end{align}

\begin{prop}\label{compare} 
We have that
\[\lambda_p({\rm I}_p)=
\begin{dcases} 
\left(
\prod_{i=1}^{n}(p^i-(-1)^i)
\right) 
\cdot 
\left( 
  \prod_{j=1}^{n-t} (p^j-(-1)^j) \cdot 
  \prod_{k=1}^{t}
  (p^k-(-1)^k) 
  \right)^{-1} & {\text{if $p$ is inert and ${\rm I}_p\sim_{\rm conj}\calP_t$}}; 
  \\ 
\left(
 \prod_{i=1}^n (p^i-1) 
 \right) 
 \cdot  
 \left( \prod_{j=1}^{m}(p^{\frac{n}{m} \cdot j}-1) \right) ^{-1},   \quad   m \coloneqq \gcd(r,s) & {\text{if $p$ is split}}.
 \end{dcases}\]
\end{prop}
\begin{proof}
We compute $\lambda_p({\rm I}_p)$ using Gan-Hanke-Yu's argument \cite{GHY}. 
 Note that  Prasad  \cite[Prop. 2.3]{Prasad} gave  a similar formula for  a semi-simple and simply connected algebraic group.
 
We recall the canonical Haar measures on ${\bf G}^1(\Q_p)$ and $I^1(\Q_p)$ constructed in \cite[Section  4]{Gross2}. 
 Let $\omega_{\bfG^1_{ \Q_p}}$ be an invariant differential of top degree on ${\bf G}^1_{\Q_p}$ with nonzero reduction on ${\bf G}^1_{\F_p}$ and let $\lvert \omega_{\bfG^1_{ \Q_p}} \rvert$ be the associated Haar measure on ${\bf G}^1(\Q_p)$. 
 Let $\lvert \omega^*_{I^1_{\Q_p}} \rvert$ on $I^1_{\Q_p}$ be the Haar measure associated with the pull-back 
   $f^*(\omega)$. 
  Since ${\bf G}^1_{\Q_p}$ is unramified with reductive model  ${\bfG}^1_{\Z_p}$,   
  we have 
 $\int_{\bfG^1(\Z_p)} \lvert \omega_{\bfG^1_{\Q_p}}\rvert= p^{-\dim \bfG^1_{\F_p}} \cdot  \lvert \bfG^1(\F_p)\rvert$ as in  \cite[p.~294]{Gross2}. 
It follows that  $
  {\mu}^*_{I^1(\Q_p)}=p^{\dim \bfG^1_{\F_p}} \cdot  \lvert \bfG^1(\F_p) \rvert^{-1}  \cdot \lvert \omega^*_{I^1_{\Q_p}} \rvert$.  
 
Now  
let  $\underline{{\rm I}}^1_{p}$  be the smooth  model over $\Z_p$ of ${\rm I}_p^1$, and let 
$\overline{{\rm I}_p^1}$ 
 be the maximal reductive quotient of  the special  fiber 
$\underline{{\rm I}}^1_{p} \otimes_{\Z_p}\F_p$. 
Further let
$N(\mathbf {G}^1_{ \F_p})$ (resp.~$N(\overline{{\rm I}_p^1})$) 
denote  the number of positive roots of $\mathbf {G}^1_{\F_p}$ (resp.~$\overline{{\rm I}_p^1}$). 
By a computation of the volume of an Iwahori subgroup of $I^1(\Q_p)$ \cite[(2.6), (2.11) and (2.12)]{GHY}, we have that 
  \begin{align}
  \begin{split}\label{vol}
 \lambda_p({\rm I}_p)
 =\Bigg( p^{\dim \bfG^1_{\F_p}} \cdot  \lvert \bfG^1(\F_p) \rvert^{-1} \cdot  \int_{{{\rm I}_p^1}} 
 \lvert
  \omega^*_{I^1_{\Q_p}}
 \rvert \Bigg)^{-1}= 
\frac{{p}^{-N(\mathbf {G}^1_{\F_p})}\cdot \lvert  \mathbf {G}^1(\F_{p}) \rvert}
 {{p}^{-N(\overline{{\rm I}_p^1})} \cdot \lvert  \overline{{\rm I}_p^1}(\F_{p}) \rvert}. 
\end{split}
\end{align}

Let ${\rm U}_t$ denote  the unitary group in $n$ variables over $\F_p$, and  $\F_q$ denote the field of order $q \coloneqq p^{n/m}$.  Then we have isomorphisms of  algebraic groups over $\F_p$ (Section ~\ref{mprh}): 
\begin{align*}
\mathbf {G}^1_{\F_p}  \simeq 
\begin{cases}
{\rm U}_n 
\\  
\GL_n, 
\end{cases} 
\ 
\overline{{\rm I}_p^1}
\simeq  
\begin{cases} 
{\rm U}_{n-t} \times {\rm U}_{t} 
 &  {\text{if $p$ is inert in $E$ and ${\rm I}_p^1 \sim_{\rm cong}\calP_t^1$}}; 
\\   \Res_{\F_q/\F_p}\GL_{m, \F_q}   & {\text{if $p$ is split in $E$}}.
\end{cases}
\end{align*}
Moreover, we have that 
\begin{align}\label{eq:n}
\begin{split}
N({\rm U}_n) =N(\GL_n)  = \frac{n(n-1)}{2}, 
\quad 
  N(\Res_{\F_{q}/\F_p}\GL_{m, \F_q})  =\frac{n}{m}\cdot \frac{m(m-1)}{2}.
 \end{split}
  \end{align}
  Equalities \eqref{vol} and  \eqref{eq:n}, and Table \ref{table:finite} imply the assertion.
\end{proof}
 \begin{remark}
 The rational function $\lambda_p({\rm I}_p)$ of $p$ is in fact a polynomial with integer coefficients. 
 Indeed, if $p$ is split in $E$, we can write 
 $
 \lambda_p({\rm I}_p)= \prod_{1 \leq i \leq n, m \nmid i}(p^i-1).
$

 Assume that $p$ is inert in $E$. 
  Formula \eqref{vol} shows that  $\lambda_p({\rm I}_p)$ equals the prime-to-$p$ factor of the fraction  $\lvert {\rm U}_n(\F_p) \rvert /\lvert {\rm U}_{n-t}(\F_p) \times {\rm U}_t(\F_p)\rvert$. 
Since ${\rm U}_{n-t} \times {\rm U}_t$ can be  embedded into ${\rm U}_n$, the fraction is an integer for any prime $p>2$. 
  From Gauss' lemma \cite[Ch.9, Ex.2]{AM} it follows that the  coefficients of the polynomial $\lambda_p({\rm I}_p)$  are integers, as desired.   
  We can also show this fact  by induction,  using the following relation: 
  If we write $\lambda_p(n,t)$ for the expression of $\lambda_p({\rm I}_p)$ (with $p$ inert) in Proposition~\ref{compare}, then   \[\lambda_p(n,t)= p^t \cdot \lambda_p(n-1, t)+ (-1)^{n-t} \cdot \lambda_p(n-1, t-1).\]
\end{remark}

\begin{thm}\label{Mass_inner}   
Let $\chi$ be the Dirichlet  character of  $E/\Q$. 
We use the convention that $\chi^j=1$ or $\chi$ according as $j$ is even or odd. 
 Let   $L(-, \chi^j)$ be the Dirichlet $L$-function associated to $\chi^j$. Let ${\rm I}_p$ be a maximal parahoric subgroup of $I(\Qp)$ and let $\lambda_p({\rm I}_p)$ be the associated number defined in \eqref{lpKp} whose formula is given in Proposition~\ref{compare}.
 
 {\rm (1)}  
The mass of $I$ with respect to  ${\rm I}_p {\bf G}(\widehat{\Z}^p)$   is 
\[ \Mass\left(I, {\rm I}_p {\bf G}(\widehat{\Z}^p)\right)= \varepsilon \cdot  \prod_{j=1}^n L(1-j, \chi^j)  \cdot \prod_{\ell \mid d_E} \kappa_{\ell}  \cdot \lambda_p({\rm I}_p)
\] 
where $d_E$ denotes the discriminant of $E$, and quantities $\varepsilon$ and  $\kappa_{\ell}$ for primes $\ell \mid d_E$  are given by 
\begin{align}\label{epsilon}
\varepsilon   
&
=
\begin{dcases}
\frac{1}{2^{n}} 
 &  {\text{if $n$ is  odd}}; 
\\
\frac{(-1)^{n/2}}{2^{n+w-1}}  &  {\text{if $n$ is  even, where}} \  w =
\# \{ \ell :  {\text{prime}},  \ \ell \mid d_{E}\}, 
\end{dcases}
\\ 
\label{kappa_ell}
\kappa_{\ell} 
& 
=
\begin{cases}
1 &
 {\text {if $n$ is odd}};   
\\
{\ell}^{n/2} +1 
 &  {\text {if $n$ is  even}},
  \ \ell  \neq 2, \ 
  d(\Lambda_{\ell})=(-1)^{n/2}; 
\\
\ell^{n/2} -1  
 & {\text {if $n$ is  even}},  \ 
  \ell \neq 2, \
   d(\Lambda_{\ell})\neq (-1)^{n/2}; 
 \\ 
2^n-1 &  {\text {if $n$ is  even}}, \  \ell =2, \   d_E \equiv 4 \pod 8, \ \Lambda_{2} \ {\text {normal}}; 
\\
2  &  {\text {if $n$ is even}}, \  \ell =2, \ d_E \equiv 4 \pod 8,  \ \Lambda_{2} \  {\text  {subnormal}}; 
\\
2^{n/2} \cdot (2^n -1) 
 &  {\text {if $n$ is even}}, \  \ell =2, \ d_E \equiv 0 \pod 8, \ \Lambda_{2} \  {\text {normal}}; 
\\ 
2^{n/2} +  1 
& 
 {\text {if $n$ is even}},  \  \ell =2, \  d_E \equiv 0 \pod 8, \ \Lambda_{2} 
\ {\text  {subnormal}}, \ d(\Lambda_2) =(-1)^{n/2}; 
\\ 
2^{n/2} -  1 
& 
 {\text {if $n$ is even}},  \  \ell =2, \  d_E \equiv 0 \pod 8, \ \Lambda_{2} 
\ {\text  {subnormal}}, \ d(\Lambda_2) \neq (-1)^{n/2}. 
\end{cases}
\end{align}
Here we write $\Lambda_{\ell}=\Lambda \otimes \Q_{\ell}$, and its determinant $d(\Lambda_{\ell})$ takes value in  $\Z_{\ell}^{\times}/\bfN_{E_{\ell}/\Q_{\ell}}(O_{E_{\ell}}^{\times})$. 
  
{\rm (2)} 
Let $\varepsilon$, $\kappa_{\ell}$, and $\lambda_p({\rm I}_p)$ be as above. 
For an integer $N\geq 3$ with $p \nmid N$, we have that 
\[ \lvert I(\Q) \backslash I(\A_f)/{\rm I}_p{\rm K}^p(N) \rvert = 
[\mathbf G(\widehat{\Z}^p) : {\rm K}^p(N)] \cdot \varepsilon \cdot   \prod_{j=1}^n L(1-j, \chi^j)  \cdot \prod_{\ell \mid d_E} \kappa_{\ell}  \cdot \lambda_p({\rm I}_p).\]
\end{thm}
\begin{proof}
  Theorem~\ref{SMS},  
  Proposition~\ref{mass:sim},  and equality \eqref{eq:GHY}  imply  assertion (1). 
  For any representative $g$ of a double coset in $I(\Q) \backslash I(\A_f)/{\rm I}_p{\rm K}^p(N)$, 
      the intersection  
  $\Gamma_g=I(\Q) \cap g {\rm I}_p {\rm K}^p(N) g^{-1}$ is trivial by Serre's lemma.  
  Hence  assertion (2) follows from    assertion (1). 
\end{proof}

 \subsection{Main theorems and examples}
For a scheme $S$ over $k$, we write $\Irr(S)$ for the set of irreducible components of $S$. 
 Recall that  $\M_{\K}^{\bas}$ and $\M_{\K}^e$ denote the basic locus and the $0$-dimensional EO stratum of $\M_{\K}$,  respectively, where $\K=\bfG(\Z_p)\K^p(N)$ with $N \geq 3$ and    $p\nmid N$. 
 \begin{thm}
   \label{intro} 
   We have that 
   \begin{align} \label{intro:bas}
   \lvert  {\rm Irr}(\mathcal M^{\rm{bas}}_{\rm K})  \rvert & = [\mathbf G(\widehat{\Z}^p) : {\rm K}^p(N)]   \cdot \varepsilon \cdot \prod_{j=1}^n L(1-j, \chi^j)  \cdot \prod_{\ell \mid d_{E}} \kappa_{\ell}  \cdot  \lambda^{\rm bas}_p \cdot \rho^{\rm bas}, 
    \\ 
    \label{intro:e}
    \lvert \mathcal M^{e}_{\rm K}(k) \rvert &= 
    [\mathbf G(\widehat{\Z}^p) : {\rm K}^p(N)]   \cdot \varepsilon \cdot \prod_{j=1}^n L(1-j, \chi^j)  \cdot \prod_{\ell \mid d_{E}} \kappa_{\ell}  \cdot  \lambda^{e}_p,
   \end{align} 
   where $d_E$ denotes the discriminant of $E/\Q$, $\chi$ denotes the Dirichlet character associated to $E/\Q$, 
    the  quantities $\varepsilon$ and $\kappa_{\ell}$ for primes $\ell \mid d_{E}$ are as defined in \eqref{epsilon} and \eqref{kappa_ell}, and 
 the quantities $\lambda_p^{\rm bas}$, $\lambda_p^e$, and $\rho^{\rm bas}$ are given by 
\begin{align}
\lambda_p^{\rm bas}&= 
\begin{dcases}
1    & \text{if $p$ is inert,  $rs$ is even}; 
\\ 
\frac{p^n-1}{p+1}   &  \text{if $p$ is  inert,  $rs$ is odd}; 
\\
\left(\prod_{h=1}^{n}(p^h-1)\right)
 \cdot  \left( \prod_{i=1}^{m}(p^{\frac{n}{m} \cdot i}-1) \right) ^{-1},  \quad   m\coloneqq \gcd(r, s) 
 & \text{if $p$ is  split,}
 \end{dcases}
 \\
   \label{lambda^e}
 \lambda^e_p & =
\begin{dcases}
 \left(\prod_{h=1}^{n}(p^h-(-1)^h) \right) \cdot \left(\prod_{i=1}^{r}(p^i-(-1)^i)\cdot \prod_{j=1}^{s}(p^j-(-1)^j)\right)^{-1} &  \text{if $p$ is inert}; 
 \\ 
 \lambda_p^{\rm bas} 
  & \text{if $p$ is split}, 
  \end{dcases}
 \\ 
 \rho^{\rm bas} &  = 
\begin{dcases}
\begin{pmatrix}
\lfloor n/2 \rfloor 
\\ 
\lfloor r/2 \rfloor 
\end{pmatrix}    & \text{if $p$ is  inert,  $rs$ is even}; 
\\ 
\begin{pmatrix}
  n/2 -1 
 \\ 
 (r-1)/2 
\end{pmatrix}   &  \text{if $p$  is inert, $rs$ is odd}; 
\\
1
 & \text{if $p$ is split}. 
 \end{dcases}
\end{align} 
In particular, if {\rm (i)} $p$ is split in $E$, {\rm (ii)} $rs=0$, or {\rm (iii)}   $n$ is even and $(r,s)=(1,n-1)$ or  $(n-1, 1)$, then   
\[\lvert {\rm Irr}(\calM_{\rm K}^{\rm bas})\rvert= \lvert \calM_{\rm K}^e(k) \rvert.\]
\end{thm} 
\begin{proof}
For each $Z \in {\rm Irr}(X_{\mu}(b))$,  let ${\rm I}_p^Z$ denote the stabilizer of $Z$ in $J_b(\Q_p)$. 
   Then we have    a bijection 
\[  
\coprod_{ [Z] \in  J_b(\Q_p)\backslash {\rm Irr} X_{\mu}(b)}  J_b(\Q_p)/{\rm I}_p^{Z} \xrightarrow{\sim} {\rm Irr}(X_{\mu}(b)).\]
We fix identifications ${I}(\Q_p)={J_b} (\Q_p)$ and $I(\A_f^p)= {\bf G}(\A_f^p)$. 
By Theorem \ref{inner} (2),  the Rapoport-Zink uniformization map induces  a bijection  (see \cite[Theorem A]{HZZ}) 
 \begin{align}\label{eq:irr}
 \coprod_{[{Z}] \in J_b(\Q_p)\backslash {\rm Irr}(X_{\mu}(b))}
 I(\Q) \backslash {I}(\A_f)/  {\rm I}_p^{{Z}}{\rm K}^p(N) \xrightarrow{\sim} \Irr(\mathcal M_{{\rm K}}^{\rm bas}).
 \end{align}
 By He-Zhou-Zhu  \cite[Theorem 4.1.2 and Proposition 2.2.5]{HZZ} and Nie \cite{Nie},   the stabilizer   ${\rm I}_p^{Z}$ is a parahoric subgroup  which has  the maximal volume among all the parahoric subgroups of $J_b(\Q_p)$. 
 In particular, the cardinality  $\lvert I(\Q) \backslash I(\A_f)/ {\rm I}_p^{{Z}}{\rm K}^p(N) \rvert$ does not depend on ${Z} \in {\rm Irr}X_{\mu}(b)$.
 If $p$ is split in $E$, then  ${\rm I}_p^Z$ is conjugate to  $\GL_m(O_B) \times \Z_p^{\times}$, as in \eqref{eq:K_p}. 
 Suppose that $p$ is inert in $E$. 
 By \eqref{eq:K_p} and \eqref{lpKp}, the maximal parahoric subgroup ${\rm I}_p^Z$ is conjugate to some $\calP_t$  such that  $\lambda_p(\calP_t)$ is minimal among  $(\calP_t)_{t}$ with  $t \equiv s \pmod 2$.   It follows from 
 Proposition \ref{compare} that
    ${\rm I}_p^Z$ is conjugate to  
     $\calP_0$ or $\calP_n$ if $rs$ is even, and to 
        $\calP_1$ or $\calP_{n-1}$  if $rs$ is odd. 
Hence Proposition \ref{Orbits} and Theorem~\ref{Mass_inner} (2) imply formula \eqref{intro:bas}. 

   Theorem \ref{inner} (3) and  Theorem~\ref{Mass_inner} (2) imply formula \eqref{intro:e}. 
\end{proof}
\begin{cor}
If $rs=0$,  or  $p$ is inert in $E$ and $rs$ is even,  then the number $|\Irr(\M_{\K}^{\bas} )|$ of irreducible components of the basic locus is independent of $p$ and is a constant number depending only on the input PEL datum $\mathscr D$ and $N$. 
For other cases, the number $|\Irr(\M_{\K}^{\bas} )|$ grows to infinity with $p$. 
\end{cor}
\begin{remark}\label{rem:un_basic}
The former  condition  is equivalent to that the basic element   $[b]\in B(G, \mu)$ 
is unramified in the sense of Xiao-Zhu \cite[Section 4.2]{XZ}, and is further equivalent to that 
$\dim_k \M_{\K}=2 \cdot \dim_k \M_{\K}^{\bas}.$
\end{remark}

Now we assume that $(r,s)=(1, n-1)$. 
The case where $p$ is split in $E$ has been studied by Harris and Taylor \cite{HT} for their proof of the local Langlands conjecture for $\GL_n$. 
In this case the EO and Newton stratifications coincide, and in particular $\M_{\K}^{\bas}=\M_{\K}^e$. 
If $p$ is inert in $E$, Vollaard and Wedhorn proved that 
   for each odd integer $1 \leq t \leq n$  
 there exists a unique EO stratum  of dimension 
 $\frac{1}{2}(t-1)$ in $\calM_{\rm K}^{\rm bas}$,   denoted by  $\mathcal M^{(t)}_{{{\rm K}}}$ \cite[Section ~6.3]{Vollaard}. 
 Let $\overline{\M}_{\K}^{(t)}$ be the Zariski closure of $\M_{\K}^{(t)}$ in $\M_{\K}$. 
 Note that $\M_{\K}^{\bas}=\overline{\M}_{\K}^{(n)}$ or $\ol{\M}_{\K}^{(n-1)}$ according as $n$ is odd or even.
 \begin{thm}\label{intro:t}
Assume that $(r,s)=(1, n-1)$  and $p$ is inert in $E$. 
Let $t$ be an odd integer such that $1 \leq t \leq n$. 
Then  
\[ \lvert {\rm Irr}(\overline{\M}_{\K}^{(t)}) \rvert=  [\mathbf G(\widehat{\Z}^p) : {\rm K}^p(N)]   \cdot \varepsilon \cdot \prod_{j=1}^n L(1-j, \chi^j)  \cdot \prod_{\ell \mid d_E} \kappa_{\ell}  \cdot \lambda^{(t)}_p\]
where $\varepsilon$ and  $\kappa_{\ell}$ are given as in \eqref{epsilon} and \eqref{kappa_ell}, and $\lambda_p^{(t)}$ is given by
\begin{align*}
\lambda^{(t)}_p  =\left( \prod_{h=1}^n (p^h-(-1)^h) \right) \cdot  
  \left(\prod_{i=1}^t (p^i-(-1)^i) \cdot 
  \prod_{j=1}^{n-t}(p^j-(-1)^j) \right)^{-1}. 
  \end{align*}
\end{thm}
\begin{proof}
The group $J_b$ is isomorphic to $\GU(\calV, \phi)$ with $d(\calV)=p^{n-1}$.  
By \eqref{phi}, we have  
\begin{align*}
\calV 
\simeq 
\begin{dcases}
\mathbb H^{{(n-1)}/{2}} \oplus 
\bbV_1
& {\text{if $n$ is odd}}; 
\\ 
\mathbb H^{n/2-1} \oplus 
\bbV_1 \oplus \bbV_p
& {\text{if $n$ is even}}. 
\end{dcases}
\end{align*} 
Further, the lattice $\calL_{n-t}$  is of orbit type $t$ in the sense of \cite{Vollaard}. 
   We regard its stabilizer   $\calP_{n-t}$ as a subgroup of $J_b(\Q_p)$. 
  By \cite[Proposition 6.3]{Vollaard}, 
  there is a  bijection 
  \begin{align*}
  I(\Q) \backslash J_b(\Q_p) \cdot  \bfG(\A_f^p)/ \calP_{n-t} \cdot {\rm K}^p(N) \xrightarrow{\sim} \Irr(\overline{\M}_{\K}^{(t)}). 
  \end{align*}
    Hence the assertion follows from Theorem~\ref{Mass_inner} (2). 
  \end{proof}
  \subsubsection{}\label{ss:cpt}
By \cite[Theorem 6.4.1.1]{Lan},  there is a toroidal compactification $\mathbf M_{{\rm K}}^{{\rm tor}}$ of the integral model $\mathbf M_{{\rm K}}$, 
which is proper and smooth  over $O_{\bfE, (p)}$. 
This   implies that (see \cite[Corollary  6.4.1.2]{Lan}) 
\[ \pi_0(\M_{\K}) \simeq  \pi_0(\Sh_{\K}(\bfG, X)_{\C}). 
 \]
When $rs=0$, the set $\Sh_{\K}(\bfG, X)_{\C}$ is discrete; see Example \ref{rs=0}.  
When $rs >0$, the map in \eqref{eq:comp_doublecoset}  induces a bijection   \begin{equation}\label{eq:c_d}
 \pi_0(\Sh_{\K}(\bfG, X)_{\C}) \simeq D(\Q)_\infty\backslash  D(\A_f)/
 \nu(\K)=D(\Q)_\infty\backslash  D(\A_f)/
 D(\Zp)  \nu(\K^p(N)). 
 \end{equation}
Here $D(\Q)_\infty$ is the intersection of $D(\Q)$ with the connected component $D(\R)^0$ of $D(\R)$. 
Further, $D(\Z_p)$ is the unique maximal open compact subgroup of $D(\Q_p)$, which satisfies that $D(\Z_p)=\nu(\bfG(\Z_p))$ by Lemma \ref{D=nu}. 
The  cardinality of $\pi_0(\Sh_{\K}(\bfG,X)_{\C})$ is given in Theorem \ref{connected} (2).

Let $J_b^0$ be the stabilizer in $J_b(\Q_p)=I(\Q_p)$ of  a connected component of $X_{\mu}(b)$. 
Then  $J_b(\Q_p)$ acts transitively on $\pi_0(X_{\mu}(b))$, and hence 
$\pi_0(X_{\mu}(b)) \simeq J_b(\Q_p)/J_b^0$. 
The Rapoport-Zink  uniformization therefore   induces a bijection 
\begin{equation}\label{eq:bas}
I(\Q) \backslash 
I(\A_f) 
/J_b^0 {\K}^p(N) \simeq \pi_0(\M_{\K}^{\bas}). 
\end{equation} 
Note that this bijection together with the description  of $J_b^0$ in Proposition  \ref{stab} (2)  gives  a generalization of the result of  Vollaard and Wedhorn in   \cite[Proposition 6.4]{Vollaard}, where they considered the case $p$ is inert in $E$ and 
$(r, s)=(1, n-1)$. 

Let $\pi_0(\Sh_{\bfG(\Z_p)}(\bfG, X)_{\C})\coloneqq  
\varprojlim\limits_{N} \pi_0(\Sh_{\bfG(\Z_p)\K^p(N)}(\bfG, X)_{\C})$ where $N$ runs through all prime-to-$p$ positive integers.  
We define  similarly $\pi_0(\M_{\bfG(\Z_p)}^{\bas})$ and 
$\M_{\bfG(\Z_p)}^e(k)$. 
\begin{lemma}\label{trans}
Assume that $rs>0$. 
Then 
the group $\bfG(\A_f^p)$ acts transitively on $\pi_0(\Sh_{\bfG(\Z_p)}(\bfG, X)_{\C})$,  $\pi_0(\M_{\bfG(\Z_p)}^{\bas})$, and $\M_{\bfG(\Z_p)}^e(k)$. 
\end{lemma}
\begin{proof}
The group $D(\Q)_{\infty}$ is dense in $D(\Q_p)$ 
by \eqref{D_infty} and \eqref{eq:nu_ell}. 
Hence the set $D(\Q)_{\infty} \backslash D(\Q_p)/ D(\Z_p)$ is a singleton. 
It follows from  \eqref{eq:c_d}  that $\bfG(\A_f^p)$ acts transitively on $\pi_0(\Sh_{\bfG(\Z_p)}(\bfG, X)_{\C})$. 

Similarly, the set $I(\Q) \backslash I(\Q_p) /U_p$ is a singleton for any open  subgroup $U_p$ of $I(\Q_p)$,  since $I(\Q)$ is dense in $I(\Q_p)$. 
It follows from  \eqref{eq:bas} and  Theorem \ref{inner} (4)  that 
the group $\bfG(\A_f^p)=I(\A_f^p)$ acts transitively on  $\pi_0(\M_{\bfG(\Z_p)}^{\bas})$ and  $\M_{\bfG(\Z_p)}^e(k)$. 
\end{proof} 
The natural
$\bfG(\A_f^p)$-equivariant map $\pi_0(\M_{\K}^{\bas}) \to \pi_0(\M_{\K})$ is therefore surjective. 
\begin{thm}\label{thm:con_comp}
Assume that $rs>0$. Then there is a bijection 
\[ \pi_0(\M_{\K}^{\bas}) 
\simeq \pi_0(\M_{\K}),  \]
unless $p$ is inert in $E$ and $(r, s)=(1, 1)$ or $p$ is split in $E$ and $\gcd(r, s)=1$, in which cases we have 
\[ 
\lvert \pi_0(\M_{\K}^{\bas}) \rvert = 
\lvert {\Irr}(\M_{\K}^{\bas})\rvert / \rho^{\bas} . \]
\end{thm}
\begin{proof}
The projection $\nu$ induces a surjective map 
\[ \nu : \pi_0(\M_{\K}^{\bas})\simeq 
I(\Q) \backslash 
I(\A_f) 
/J_b^0 {\K}^p(N) \to \nu(I(\Q)) 
\backslash 
D(\A_f) / \nu(J_b^0) \cdot \nu(\K^p(N)). 
\]
By equality \eqref{D_infty} and Lemma \ref{sim:I}, we have  $\nu(I(\Q))=D(\Q)_{\infty}$. Further, Proposition \ref{stab} implies that $\nu(J_b^0)=D(\Z_p)$. 
It follows from \eqref{eq:c_d} that the double coset space in the RHS has the same cardinality as the set 
$\pi_0(\Sh_{\K}(\bfG, X)_{\C}) \simeq \pi_0(\M_{\K})$.


On the other hand, 
for each $g \in I(\A_f)$, there is a surjective map
\begin{equation}\label{eq:fib}
I^{\der}(\Q) 
\backslash I^{\der}(\A_f^p)/(I^{\der}(\A_f) \cap g J_b^0 {\K}^p(N)g^{-1})
\to 
I(\Q) \backslash
I(\Q) I^{\der}(\A_f) g J_b^0 {\K}^p(N)/J_b^0 {\K}^p(N) = \nu^{-1}(\nu(g)). 
\end{equation}
The subgroup $J_b^0$ is normal in $I(\Q_p)$ and  contains  $I^{\der}(\Q_p)= 
\{g \in I(\Q_p) \mid {\rm nrd}(g)=1, \Sim(g)=1\}$. 
Hence $I^{\der}(\A_f) \cap g J_b^0 {\K}^p(N)g^{-1}=I^{\der}(\Q_p)\cdot I^{\der}(\A_f^p) \cap g  {\K}^p(N)g^{-1}$.

Suppose that  $I^{\der}({\Q_p})$ is not compact.  
The strong approximation theorem \cite{Kneser} can then  be applied to the simply connected group $I^{\der}$ and the non-compact group  $I^{\der}({\R})\times I^{\der}(\Qp)$. 
Hence 
the LHS of \eqref{eq:fib} consists of a single element. 
It follows that the map $\nu$ is  bijective, and in particular  $\pi_0(\M_{\K}^{\bas})$ and $\pi_0(\M_{\K})$ have  the same cardinality. 
Thus the surjective map 
$\pi_0(\M_{\K}^{\bas}) \to \pi_0(\M_{\K})$ in this case is  bijective.  

Suppose that $I^{\der}(\Q_p)$ is compact. 
Then $J_b^0$ is also compact 
since we have  an exact sequence $1 \to I^{\der}(\Q_p) \to J_b^0 \xrightarrow{\nu} D(\Z_p) \to 1$. 
Further, $J_b^0$ contains any parahoric subgroup of $J_b(\Q_p)$, and hence it is the unique maximal parahoric subgroup. 
It follows from 
\eqref{eq:irr} that 
\[ \lvert \pi_0(\M_{\K}^{\bas}) \rvert =
\lvert I(\Q) \backslash I(\A_f) /J_b^0 \K^p(N) \rvert= 
\lvert \Irr(\M_{\K}^{\bas}) \rvert/ \rho^{\bas}. \]

 The assertion thus follows from Lemma \ref{cpt}. 
\end{proof}

\begin{remark}\label{rem:equi-distributed} 
In an alternative definition, 
$I^\der(\Qp)$ is compact if and only if the corresponding basic element $b\in \bfG(L)$ is superbasic, see  
\cite[Section 4]{HV}.
Lemma \ref{trans} implies that 
irreducible  components of $\M_{\K}^{\bas}$ and points of $\M_{\K}^e(k)$ are equally distributed to each connected component $Y$ of  $\M_{\K}^{\bas}$.  
Therefore, 
in the superbasic case we have that 
\begin{align*}
    \lvert \Irr(Y) \rvert= 
    \rho^{\bas} \quad \text{and} \quad 
    \lvert (\M_{\K}^e \cap Y)(k) \rvert 
    = \lambda^e_p/\lambda_p^{\bas}. 
\end{align*}
\end{remark}
\begin{eg}[Shimura sets]\label{rs=0}
  Assume that 
  $(r,s)=(0, n)$ or $(n, 0)$. 
In this case, 
the group $\bfG(\R)$   
is compact modulo center, and  the complex Shimura variety $\Sh_{\K}(\bfG, X)_{\C} = \bfG(\Q) \backslash \bfG(\A_f)/\K$ for an open compact subgroup $\K \subset \bfG(\A_f)$ is a discrete set. 
Let $p$ be a prime which is unramified in $E$, and let $\K\coloneqq \bfG(\Z_p)  \K^p(N)$ with $N\geq 3$ and $p \nmid N$.
Then the set $\M_{\K}=\bfM_{\K} \otimes k$ has the same  cardinality  as the set  $\Sh_{\K}(\bfG, X)_{\C}$. 
Further we have $\M_{\K}=\M_{\K}^{\bas}=\M_{\K}^e$, 
whose cardinality is given by Theorem~\ref{intro} with $\lambda_p^{\bas}=\lambda_p^e=\rho^{\bas}=1$.

When $(r,s)=(0,1)$ or $(1,0)$,  
by the relation $L(0,\chi)=2h(E)/|\mu_E|$ we have  that  
\begin{equation}\label{eq:10}
 |\M_{\K}^e(k)|=[\bfG(\wh \Z^p):\K^p(N)]\cdot h(E)/|\mu_E|.   
\end{equation}
\end{eg} 
We use the following lemma in the examples below. 
\begin{lemma} \label{lm:C(Fp2)}
Let $q$ be a power of $p$ and $H\subset \bbP^n$ $(n\ge 1)$ be the Fermat hypersurface defined by $X_0^{q+1}+\dots +X_n^{q+1}=0$. Then 
\begin{equation}\label{eq:HFq2}
|H(\F_{q^2})|=\frac{(q^{n+1}+(-1)^n)(q^n+(-1)^{n+1})}{q^2-1}.    
\end{equation}
\end{lemma}
\begin{proof}
Let 
\begin{align*}
    S_n & :=\# \{[x_0 : \cdots  :  x_n]\in \bbP^n(\F_{p^2}) \mid  x_0^{q+1}+\dots+ x_{n-1}^{q+1}+x_n^{q+1}=0 \,\}=\# H(\F_{q^2}), 
    \\ 
    A_n & :=\# \{(x_0,\dots, x_{n-1})\in \bbA^n(\F_{q^2}) \mid  x_0^{q+1}+\dots+x_{n-1}^{q+1}+1=0 \,\}.
\end{align*}
 We separate the cases where $x_n=0$ or $x_n=1$; this gives $S_n=S_{n-1}+A_n$ and $S_n=A_1+\cdots+A_n$. 
 By solving the equation $-X_{n-1}^{q+1}=X_0^{q+1} +\cdots + X_{n-2}^{q+1}+1$ in $\bbA^n(\F_{q^2})$, we obtain
\begin{align*}
   A_n & =\#\{x_0^{q+1}+\cdots+x_{n-2}^{q+1}+1=0\}+\#\{ x_0^{q+1}+\cdots+x_{n-2}^{q+1}+1\neq 0 \}\cdot (q+1)
   \\
   & 
   =A_{n-1}+(q+1)(q^{2n-2}-A_{n-1})=q^{2n-1}+q^{2n-2}-q A_{n-1}.\end{align*}
Put $B_n:=A_n-q^{2n-1}$, then $B_1=A_1-q=1$ and $B_n=(-q)\cdot B_{n-1}$. Therefore, $B_n=(-q)^{n-1}$ and $A_n=q^{2n-1}+(-q)^{n-1}$. We compute
\[ \begin{split}
S_n&=A_1+\dots+A_n=\frac{q(q^{2n}-1)}{q^2-1}+\frac{(-q)^n-1}{-q-1} 
=\frac{(q^{n+1}+(-1)^n)(q^n+(-1)^{n+1})}{q^2-1}. 
\end{split} 
 \]

\end{proof}

\begin{eg}[Picard modular surfaces]\label{Picard}
Assume that $(r, s)=(1, 2)$. 
In this case,  the  complex Shimura variety $\bfM_
{{\rm K},\C}=\mathbf M_{\rm K} \otimes \C$ is of dimension two,  and called a Picard modular surface. 
When $p$ is inert in $E$,  $\mathcal M_{\rm K}^{\rm bas}$ coincides the supersingular locus, which is of dimension   one, and  $\mathcal M_{\rm K}^{e}$ coincides with the superspecial locus. 
Now we fix a connected component $\bfM_{{\rm K},\C}^{0}$ of $\bfM_{{\rm K},\C}$, and let $\calM_{\rm K}^0$ be the corresponding connected component of $\calM_{\rm K}$ through a smooth compactification $\ol{\bfM}_{\rm K}$ of $\bfM_{\rm K}$.
Write $\mathcal N$ for the number of irreducible components of  $\mathcal M_{\rm K}^{\rm bas}$ which are contained in $\mathcal M_{\rm K}^0$. 
When $p$ is inert in $E$, De Shalit and Goren \cite{DG} proved that 
\begin{equation}\label{eq:N}
    3\mathcal N=c_2(\bfM_{{\rm K},\C}^{0}),
\end{equation}
where $c_2(\bfM_{{\rm K},\C}^{0})$ is the top Chern class of the smooth model $\ol{\bfM_{{\rm K},\C}^{0}}$ of $\bfM_{{\rm K},\C}^{0}$, which depends only on $\bfM_{{\rm K},\C}^{0}$. 
They also showed that 
under the condition that $N$ is sufficiently large depending on $p$ (see~\cite[Theorem 2.1(iii)]{DG}), the number of superspecial points in  $\mathcal M_{\rm K}^0$ is equal to 
\begin{equation}\label{eq:MK0e}
   \frac{c_2(\bfM_{{\rm K},\C}^{0})}{3} \cdot \frac{p^3+1}{p+1}. 
\end{equation}
The class $c_2(\bfM_{{\rm K},\C}^{0})$ had been computed explicitly by Holzapfel  \cite[Main Theorem 5A.4.7]{Holzapfel}\footnote{Formula  \cite[(5A.4.3), p. 325]{Holzapfel}  (as well as  \cite[(1.12) and (1.14)]{DG}) 
contains an unnecessary factor `$3$' in its RHS, 
although it is computed correctly in the proof \cite[line 16, p. 326]{Holzapfel}.}:  
\begin{equation}
    \label{eq:c2MK0}
  c_2(\bfM_{{\rm K},\C}^{0})= [ \bfG^{\der}(\wh\Z) :\K^{\der}(N)]\cdot \frac{|d_E|^{5/2}}{32\pi^3} \cdot L(3, \chi), 
\end{equation}
where we write $\bfG^{\der}(\wh\Z)\coloneqq \bfG(\wh\Z) \cap \bfG^{\der}(\A_f)$ and $\K^{\der}(N)\coloneqq \ker \Big(\bfG^{\der}(\wh\Z) \to \bfG(\wh\Z/N\wh\Z)\Big)$.  
Formulas  \eqref{eq:N}, \eqref{eq:MK0e},  \eqref{eq:c2MK0}, and the functional equation imply that 
\begin{equation}\label{eq:N_MKe0}
\calN=-[\bfG^{\der}(\wh\Z) :  
\K^{\der}(N)]\cdot \frac{1}{48} \cdot L(-2, \chi), \quad \text{and} \quad |(\calM^0_{\rm K}\cap \calM^e_{\rm K})(k)|=\calN \cdot (p^2-p+1).   
\end{equation}

On the other hand, Theorem~\ref{intro} shows that if $p$ is inert in $E$ then 
\begin{equation}\label{eq:12inert}
\begin{split}
|\Irr(\calM_{\rm K}^{\bas})|&=-[\bfG(\wh \Z^p):\K^p(N)]\cdot \frac{1}{48} \frac{h(E)}{|\mu_E|}\cdot L(-2,\chi), \\ |\calM_{\rm K}^{e}(k)|&=-[\bfG(\wh \Z^p):\K^p(N)] \cdot \frac{1}{48} \frac{h(E)}{|\mu_E|}\cdot L(-2,\chi)\cdot (p^2-p+1).     
\end{split}
\end{equation}
Here we have an equality 
\[ [\bfG(\wh\Z^p):  \K^p(N)]=
[\bfG(\wh\Z):  \K(N)]=[\bfG^{\der}(\wh \Z):\K^{\der}(N)]\cdot [D(\widehat{\Z}) : \nu(\K(N))].\] 
Further, Theorem~\ref{connected} shows that  $\calM_{\K}$ has 
$[D(\widehat{\Z}) : \nu(\K(N))] \cdot \lvert \mu_E \rvert^{-1} \cdot  h(E)$ connected components. 
From these results one deduces \eqref{eq:N_MKe0}.  

 In this case, every irreducible component of $\M_{\K}^{\bas}$ is isomorphic to the Fermat curve $C: X_0^{p+1}+X_1^{p+1}+ X_2^{p+1}=0$ of degree $p+1$ and it has $|C(\F_{p^2})|=p^3+1$ superspecial points. Every superspecial point is contained in $p+1$ irreducible components; see \cite[Theorem 4]{Vollaard1}. 

When $p$ is split in $E$, the basic locus $\M_{\K}^{\bas}$ is zero-dimensional and we have
\begin{equation}\label{eq:12split}
|\Irr(\calM_{\rm K}^{\bas})|=|\calM_{\rm K}^{e}(k)|=-[\bfG(\wh \Z^p):\K^p(N)]\cdot \frac{1}{48} \frac{h(E)}{|\mu_E|}\cdot L(-2,\chi)\cdot (p-1)(p^2-1).     
\end{equation}

\end{eg}

\begin{eg}[$n=2$ or $n=4$] For simplicity we assume that $2$ is unramified in $E$; this implies that $\kappa_2=1$ in \eqref{kappa_ell}. Set $S_E^{\rm qs}:=\{\ell|d_E: \bfG^1_{\Q_\ell} \text{ is quasi-split} \}$ and $S_E^{\rm nq}:=\{\ell|d_E: \bfG^1_{\Q_\ell} \text{ is not quasi-split} \}$, and set
\[ \begin{split}
    L(2,E,N)&:=[\bfG(\wh \Z):\K(N)]\cdot \frac{1}{2^w\cdot 12}\frac{h(E)}{|\mu_E|}\cdot \prod_{\ell\in S_E^{\rm qs}} (\ell+1) \prod_{\ell\in S_E^{\rm nq}} (\ell-1), \quad \text{and} \\
    L(4,E,N)&:=-[\bfG(\wh \Z):\K(N)]\cdot \frac{1}{2^w\cdot 5760}\frac{h(E)}{|\mu_E|}\cdot L(-2,\chi)\cdot \prod_{\ell\in S_E^{\rm qs}} (\ell^2+1) \prod_{\ell\in S_E^{\rm nq}} (\ell^2-1).
\end{split} \]

For $(r,s)=(1,1)$, the basic locus $\M_{\K}^{\bas}$ is zero-dimensional and we have (for both $p$ inert and split)
\begin{equation}\label{eq:11}
|\M_{\K}^e(k)|=|\Irr(\M_{\K}^{\bas})|=L(2,E,N) \cdot (p-1). 
\end{equation}

For $(r,s)=(1,3)$, if $p$ is inert in $E$, then $\M_{\K}^{\bas}$ is one-dimensional and we have
\begin{equation}\label{eq:13inert}
  |\M_{\K}^e(k)|=|\Irr(\M_{\K}^{\bas})|=L(4,E,N) \cdot (p-1)(p^2+1).
\end{equation}
In this case, every irreducible component of $\M_{\K}^{\bas}$ is isomorphic to the Fermat curve $C: X_0^{p+1}+X_1^{p+1}+ X_2^{p+1}=0$ of degree $p+1$ and it has $|C(\F_{p^2})|=p^3+1$ superspecial points. Every superspecial point is contained in $p^3+1$ irreducible components; see \cite[Example G, p.~595]{Vollaard}.

If $p$ is split in $E$, then $\M_{\K}^{\bas}$ is zero-dimensional and we have
\begin{equation}\label{eq:13split}
  |\M_{\K}^e(k)|=|\Irr(\M_{\K}^{\bas})|=L(4,E,N) \cdot (p-1)(p^2-1)(p^3-1). 
\end{equation} 

For $(r,s)=(2,2)$, if $p$ is inert in $E$, then $\M_{\K}^{\bas}$ is two-dimensional and we have
\begin{equation}\label{eq:22inert}
  |\M_{\K}^e(k)|=L(4,E,N) \cdot (p^2-p+1)(p^2+1), \quad |\Irr(\M_{\K}^{\bas})|=2 L(4,E,N). 
\end{equation}
In this case, every irreducible component of $\M_{\K}^{\bas}$ is isomorphic to the Fermat surface $S: X_0^{p+1}+X_1^{p+1}+ X_2^{p+1}+X_3^{p+1}=0$ of degree $p+1$ and it has $|S(\F_{p^2})|=(p^2+1)(p^3+1)$ superspecial points; see \cite{HP0} .

If $p$ is split in $E$, then $\M_{\K}^{\bas}$ is 1-dimensional and we have
\begin{equation}\label{eq:22split}
  |\M_{\K}^e(k)|=|\Irr(\M_{\K}^{\bas})|=L(4,E,N) \cdot (p-1)(p^3-1).
\end{equation} 
In this case, every irreducible component of $\M_{\K}^{\bas}$ is isomorphic to $\bbP^1$ and it
has $p^2+1$ superspecial points. Every superspecial point is contained in $p^2+1$ irreducible components \cite{Fox}.   
\end{eg} 

\begin{eg}[Basic EO strata for $\GU(2,2)$]\label{2,2}  Assume that $p$ is inert in $E$. By \cite{HP0}, the EO stratification agrees with the Bruhat-Tits stratification on the basic locus. There are two 2-dimensional EO strata, one $1$-dimensional EO stratum, and the unique $0$-dimensional EO stratum. The union of all basic EO strata is the basic locus $\M_{\K}^{\bas}$. 

Let $\M_{\K}^w$ be one of the two  $2$-dimensional basic EO strata.
Then 
the closure of every irreducible component of $\M_{\K}^w$ is  an irreducible component of $\M_{\K}^{\bas}$. 
Hence 
the stabilizer in  $J_b(\Q_p)$ of an irreducible component  of its preimage in the Rapoport-Zink space is  an open compact subgroup with maximal volume. 
This implies that $\M_{\K}^{w}$ has a multiple of $L(4,E,N)$ irreducible components. 
From  \eqref{eq:22inert} it  follows that   $\M_{\K}^w$  has $L(4,E,N)$ irreducible components.

We know the number of the points in $\calM_{\K}^e(k)$ from \eqref{eq:22inert}. Every irreducible component of the closure $X$ of the $1$-dimension EO stratum is isomorphic to $\bbP^1$ and has $p^2+1$ superspecial points. On the other hand, every superspecial point is contained in $p+1$ irreducible components of $X$. Using the incidence relation, the $1$-dimensional basic EO stratum has $L(4,E,N) \cdot  (p^3+1)$ irreducible components. 

Assume that $p$ is split in $E$. By \cite{Fox}, the basic locus is the union of the $1$-dimensional basic EO stratum and the unique $0$-dimensional EO stratum. Thus, by \eqref{eq:22split} both the $1$-dimensional stratum and the $0$-dimensional stratum have $L(4,E,N) \cdot (p-1)(p^3-1)$ irreducible components.

\end{eg}
\section{Upper bound on the number of Hecke eigensystems of  mod $p$ automorphic forms}\label{sec:bound}
Let $\mathscr D$ and   $p$ be as in Section ~\ref{moduli}, and let $\K \coloneqq \bfG(\Z_p) \K^p(N)$ where $N\geq 3$ and $p\nmid N$. 
Let $[\mu]$ be the conjugacy class of the cocharacter $\mu_h$ defined by $\mathscr D$. 
We can
find a representative $\mu$ of $[\mu]$ that extends to a cocharacter  $\mu : \bbG_{{\rm m}, W(k)} \to \bfG_{W(k)}$. 
Let $L(\mu)$ be the centralizer of $\mu$ in $\bfG_{W(k)}$. 
Then $L(\mu)$ is a reductive group over $W(k)$ \cite[A.6]{VW}. 
We fix a maximal torus $T$ of the $k$-group  $L(\mu) \otimes_{W(k)} k$,  and a Borel subgroup $B$  containing $T$. 
Let $X^*(T)^+$ denote the set of dominant weights with respect to $B$. To each $\xi \in X^*(T)^+$, 
one can associate  a  vector bundle  $\mathscr V(\xi)$  on $\mathcal M_{{\rm K}}$, called the \emph{automorphic bundle of weight $\xi$}; see 
\cite[Definition  6.7]{Lan2} and  \cite[Section 4.1]{TY}. 
Let $\bfM_{\K}^{\rm tor}$ be a  toroidal compactification of $\bfM_{\K}$, 
 and let $\mathcal M_{{{\rm K}}}^{\rm tor}:=\mathbf M_{\rm K}^{\rm tor} \otimes k$. 
By  \cite[Section 6B]{Lan2}, 
 there exits a canonical extension  of $\mathscr V(\xi)$ to $\mathcal M_{{{\rm K}}}^{\rm tor}$, denoted by $\mathscr V^{{\rm can}}(\xi)$. 
 The space of \emph{mod $p$ automorphic forms}  is then  defined as
\[ 
\mathcal A(\mathscr D, N)\coloneqq \bigoplus_{\xi \in X^*(T)^+} 
H^0(\mathcal M_{{{\rm K}}}^{\rm tor}, \mathscr{V}^{can}(\xi)). 
\] 
We remark that 
 $\mathcal M_{{{\rm K}}}$ is compact or has the  boundary with codimension  larger than one, and hence  
Koecher's principle holds for $\mathcal M_{{\rm K}}$ by \cite[Theorem~2.3]{Lan3}:
\[H^0(\mathcal M_{{\rm K}}^{\rm tor}, \mathscr{V}^{\rm{can}}(\xi) )
\simeq 
H^0(\mathcal  M_{{\rm K}}, \mathscr{V}(\xi)).\]
The space $\mathcal A(\mathscr D, N)$ admits an action of the unramified Hecke algebra 
\[\mathcal H\coloneqq 
{\bigotimes_{\ell \neq p}}'  
\mathcal H_{\ell}({\bfG}_{ \Q_{\ell}}, {{\rm K}}_{\ell} ; \Z_p).\] 
We say that a \emph{system 
of Hecke eigenvalues 
$(b_T) _{T \in{\mathcal H}} \in k^{\mathcal H}$ appears in} $\mathcal A(\mathscr D, N)$  if there exists an element $f \in \mathcal A(\mathscr D, N)$ such that $T f = b_T f$ for all
$T \in \mathcal  H$. 
\begin{thm}\label{bound}
Let $\mathcal  N(\mathscr D, N)$  denote  the  number of the systems of prime-to-$p$ Hecke eigenvalues appearing in $\mathcal A(\mathscr D, N)$. 
Then  
\[\mathcal N(\mathscr D, N) \leq 
 [ \mathbf G(\wh{\Z}^p) : \K^p(N)] \cdot \varepsilon  
 \cdot \prod_{j=1}^n L(1-u, \chi^j) \cdot  \prod_{\ell \mid d_E} \kappa_{\ell} \cdot 
\lambda^e_p \cdot \nu_p \]
where $\varepsilon$,  $\kappa_{\ell}$, $\lambda_p^e$ are as defined in  \eqref{epsilon},  \eqref{kappa_ell},  \eqref{lambda^e}, and   $\nu_p$ is given by 
\begin{align*}
\nu_p=
\begin{dcases}
 p^{(r(r-1)+s(s-1))/{2}} 
 \cdot 
 p^{n-2}(p-1)(p+1)^2 & \ {\rm if} \ p \ {\rm is \ inert \ and} \  rs \neq 0;
 \\
 p^{{n(n-1)}/{2}} \cdot p^{n-1}(p-1)(p+1) & \ {\rm if} \ p \ {\rm is \ inert \ and} \ rs=0; 
 \\
  p^{{n(m-1)}/{2}}\cdot p^{n-{n}/{m}} \cdot (p^{{n}/{m}}-1) & \ {\rm if} \ p \ {\rm is \ split}. 
  \end{dcases}
 \end{align*} 
\end{thm}
\begin{proof}
We fix a point $(A, \iota, \lambda, \bar{\eta})$ of the $0$-dimensional EO  stratum $\mathcal  M_{{{\rm K}}}^{e}(k)$. 
Let $I$ be the associated  algebraic group over $\Q$. 
  Then $I(\Q_p) \simeq J_b(\Q_p)$ for an element $b$ representing the basic class $[b] \in B(\bfG_{\Q_p}, \mu)$. 
  Let  $(A[p^{\infty}], \iota, \lambda)$ be the  associated   $p$-divisible group with $\mathscr D$-structure, and 
let $M$ be its  \dieu module. 
 By Proposition  \ref{J_Minert},  the stabilizer ${\rm I}_p^e$ of $M$  in $I(\Q_p)$ is a maximal parahoric subgroup. 
Let $\underline{{\rm I}}_p^e$ be its smooth model over $\Z_p$ and 
  $\overline{{\rm I}}_p^{e}$ be the  maximal reductive quotient of the special fiber $\underline{{\rm I}}_p^e \otimes_{\Z_p} \F_p$. 
Then 
 \begin{align*}
\overline{{\rm I}}_p^{e}  \simeq \begin{cases} 
 {\rm G}({\rm U}_r \times {\rm U}_s)  & \text{if $p$ is inert in $E$};
 \\ 
 (\Res_{\F_q/\F_p}\GL_m)  \times {\mathbb G}_{\rm m, \F_p} & 
 {\text{if $p$ is split in $E$,}}
 \end{cases}
  \end{align*}
  where ${\rm G}({\rm U}_r \times {\rm U}_{s})$ denotes the group of pairs of matrices  $(g_1,  g_2) \in  \GU_r \times \GU_{s}$ 
having the same similitude factor, and $q \coloneqq  p^{n/m}$. 

On the other hand, the group  $\underline{{\rm I}}_p^e(\F_p)$     acts on the $k$-vector space $M/pM$, preserving the subspace ${\sfV}M/pM$. We write $I(p)$ for the image of the induced homomorphism   $\underline{{\rm I}}_p^e(\F_p) \to \GL_k({\sfV}M/pM) \times \GL_k(M/{\sfV}M)$. 
Then a  computation in the proof of Proposition  \ref{J_Minert} shows that $I(p) \simeq \overline{{\rm I}}_p^{e}(\F_p)$. 
 %

 Let 
  ${\rm Irr}_k(I(p))$ denote   the set of isomorphism classes of simple $k[I(p)]$-modules over $k$.
By \cite[Theorem 0.1]{TY},  the systems of prime-to-$p$ Hecke eigenvalues appearing in the space $\mathcal A(\mathscr D, N)$ 
  are the same as those appearing in  the space    of \emph{algebraic modular forms} on $I$  with varying weights  $V_{\tau}\in {\rm Irr}_k(I(p))$. 
As a corollary, we have an inequality (\cite[Corollary 5.7]{TY})
\begin{align}\label{eq:bound}
\mathcal N(\mathscr D, N) \leq  \lvert \mathcal  M_{\rm K}^e(k) \rvert \cdot  \sum_{V_{\tau}\in  \Irr_k(I(p))} \dim_k V_{\tau}.  
\end{align} 

When $p$ is inert in $E$,   Reduzzi  showed that (\cite[Section ~5.2 (5)]{Reduzzi}): 
 \begin{align}\label{diminert}
 \sum_{V_\tau \in \Irr_k(I(p))}\dim_k V_{\tau} \leq 
 \begin{dcases}
 p^{({r(r-1)+s(s-1)})/{2}} 
 \cdot 
 p^{n-2}(p-1)(p+1)^2 & \ {\rm if} \ rs \neq 0; 
 \\
 p^{{n(n-1)}/{2}} \cdot p^{n-1}(p-1)(p+1) & \ {\rm if} \ rs=0.
 \end{dcases}
  \end{align}
  
Now we  
assume  that $p$ is split in $E$. We  first compute  the cardinality of ${\rm Irr}_k(I(p))$.     
By  \cite[Corollary 3 of Theorem~42]{SerreL}, 
 this cardinality  equals  the number of $p$-regular conjugacy classes of $I(p)$, where an element of $I(p)$  is said to be $p$-regular if its order is prime to $p$. 
 Further, this number 
 is equal to $p^l \cdot\lvert Z(\F_{p}) \rvert$ where 
 $Z$ be the center of $\overline{{\rm I}}_p^e$ and $l$ be the semisimple rank of $\overline{{\rm I}}_p^{e}$  (\cite[Theorem 3.7.6 (ii)]{Carter}).
We have  $Z(\F_p) \simeq \F_q^{\times} \times \F_{p}^{\times}$, and 
the derived subgroup of $\overline{{\rm I}}_p^{e}$ is isomorphic to $\Res_{\F_q/\F_p}\SL_{m, \F_q}$, whose   rank is 
$l =({n}/{m}) \cdot (m-1)=n-{n}/{m}.$ 
It follows that 
\begin{align}\label{irr}
\lvert \Irr_k (I(p))\rvert =
p^{n-n/m}\cdot (q-1)(p-1).
\end{align}

Next we give an upper bound of $\dim_k V_{\tau}$. 
The group $I(p)$ is a finite group of Lie type, and hence  has a structure of a split $BN$-pair of characteristic $p$ by \cite[Section~1.18]{Carter}. 
It follows from 
 \cite[Corollaries 3.5 and 5.11]{Curtis} that 
 the dimension of  a simple $k[I(p)]$-module $V_{\tau}$ is  no larger than the order of a $p$-Sylow subgroup of $I(p)=\GL_m(\F_q) \times \F_p^{\times}$, which is  equal to  $q^{{m(m-1)}/{2}}$. 
 Hence we have 
\begin{align}\label{dimVsplit}
\dim_{k}V_{\tau} \leq q^{{m(m-1)}/{2}}=p^{{n(m-1)}/{2}}.
\end{align} 

Formulas (\ref{eq:bound}), (\ref{diminert}), (\ref{irr}), (\ref{dimVsplit}), and \eqref{intro:e} imply the  assertion.
 \end{proof}
 \begin{cor}[Asymptotics]\label{asy}
 If we fix $N\geq 3$ and $n=\dim_E V$,  then 
 \[ \mathcal N(\mathscr D, N)=O(p^{{n(n+1)}/{2}+1})  \quad  {\text{as}}  \quad p \to \infty.\]
 \end{cor}
 \begin{remark}
 Ghitza \cite{Ghitza2}  gave an explicit upper bound for the number of the systems of Hecke eigenvalues of mod $p$ Siegel modular forms using   Hashimoto-Ibukiyama-Ekedahl's mass formula \cite{HI, Ekedahl}. 
 For  automorphic forms on $\M_{\K}$  
  with inert $p$, 
Reduzzi \cite{Reduzzi} has  proved   that 
 $ 
  \mathcal N(\mathscr D, N)=O(p^{n^2+n-rs+1})$ as $p \to \infty$. 
  He observed that the  superspecial locus of $\M_{\K}$ (which is  $\M_{\K}^e$) is  embedded into the one of a Siegel modular variety,  and used  Hashimoto-Ibukiyama-Ekedahl's mass formula.  
 One sees that his bound is improved by  Corollary  \ref{asy}.   
  \end{remark}

\end{document}